\documentclass[a4paper,12pt]{extarticle}
    \newdimen\pp \pp=2.5cm
    \usepackage[top=\pp,bottom=\pp,left=\pp,right=\pp]{geometry}
    \usepackage[T1]{fontenc}
    \usepackage{ucs}
    \usepackage[utf8]{inputenc}
    \usepackage[english]{babel}
    \usepackage{amsfonts}
    \usepackage{amssymb}
    \usepackage{amsthm}
    \usepackage{amsmath}
    \usepackage{mathtools}
    \usepackage{fancyhdr}
    \usepackage[normalem]{ulem}
    \usepackage{enumitem}
    \usepackage{ marvosym }
    \usepackage[nottoc]{tocbibind}

    \usepackage{caption}
    \usepackage{float}
    
    \usepackage{txfonts}
    
    \usepackage{pifont}
    
    \usepackage[outline]{contour}
    \contourlength{2pt}
    
    \usepackage[dvipsnames]{xcolor}
    
    \usepackage{lmodern}
    
    \usepackage{tikz}
    \usepackage{pgfplots}
    \usetikzlibrary{hobby, patterns}
    \pgfplotsset{compat=newest}
    
    \usepackage{hyperref}
    \usepackage{cleveref}
    \crefname{equation}{equation}{equations}
    
    \hypersetup{
        colorlinks=true,
        linkcolor=pantone,
        filecolor=pantone,      
        urlcolor=pantone,
        citecolor=pantone,
        pdftitle={The Pólya Web},
        }
    
    \usepackage{calc}
    \usepackage{float}
    \usepackage{pgfplots}
    \usepackage{dsfont}
    \usepackage{dsfont}
    \usepackage{lipsum}
    \usepackage{titlesec}
    \usepackage{graphicx}
    \usepackage{wrapfig}
    \usepackage{svg}
    \usepackage{subcaption}
    \usetikzlibrary{shapes.geometric}
    \usetikzlibrary{positioning}
    \usetikzlibrary{graphs, quotes, arrows.meta}
    \usetikzlibrary{automata}
    
    \newcommand{\defeq}{\vcentcolon=} % definiálás balra
    \newcommand{\eqdef}{=\vcentcolon} % definiálás jobbra

    \def\zarojel#1{{\left({#1}\right)}} % zárójel
    \def\halm#1{{\left\{{#1}\right\}}} % halmaz
    \def\norm#1{{\left\|{#1}\right\|}} % norma
    \def\R{\mathbb{R}} % valós számok
    \def\Z{\mathbb{Z}} % egész számok
    \def\N{\mathbb{N}} % természetes számok
    \def\valseg#1{\mathbb{P}\zarojel{#1}} % valószínűség
    \def\varhato#1{\mathbb{E}\zarojel{#1}} % várható érték
    \def\var#1{\mathbb{V}\hspace{-0.05cm}\text{ar}\,\zarojel{#1}} % variancia
    \def\kek#1{\mathrm{proj}_1 {#1}} % az első
    \def\piros#1{\mathrm{proj}_2 {#1}} % a második
    
    \definecolor{pantone}{RGB}{134,38,51}
    
    \DeclarePairedDelimiter\floor{\lfloor}{\rfloor}
    \DeclarePairedDelimiter\ceil{\lceil}{\rceil}
    
    \theoremstyle{plain}
    \newtheorem{thm}{Theorem}[subsection]
    \newtheorem{lmm}[thm]{Lemma}
    \newtheorem{prp}[thm]{Proposition}
    \newtheorem{dfn}[thm]{Definition}
    
    \newtheorem{crl}[thm]{Corollary}

    \usepackage{nicematrix}
    \usepackage[table]{xcolor}
    
    \author{Ákos Urbán}
    \title{The Pólya Web}
    
\begin{document}

    \pagenumbering{roman}
    
    \begin{titlepage}
      \begin{center}
        \null
        \vspace{2cm}
            {\fontsize{40}{48} \selectfont \fontfamily{cmr} \textsc{\textbf{\textcolor{pantone}{The Pólya Web}}}}\\
            \vspace{0.5cm}
            \Large{\textsc{Master's Thesis}}\\
            \vspace{2.5cm}
            \Large{\textsc{Author}}\\
    	\large{\textbf{Ákos Urbán}}\\
    	\small{\textit{MSc student, TU Budapest}}\\
            \vspace{1cm}
            \Large{\textsc{Supervisors}}\\
            \vspace{0.75cm}
                    %\hspace{1.25cm}
    				\begin{minipage}{0.45\textwidth}
                    \begin{center}
    				\large{\textbf{Balázs Bárány}}\\
    				\small{\textit{associate professor,\\Department of Stochastics, TU Budapest}}
                    \end{center}
    				\end{minipage}
    				%\hspace{1cm}
                    \hfill
    				\begin{minipage}{0.45\textwidth}
                    \begin{center}
    				\large{\textbf{Bálint Tóth}}\\
    				\small{\textit{research professor,\\Alfréd Rényi Institute of Mathematics}}
                    \end{center}
    				\end{minipage}
           \vfill
           \includegraphics[width=0.5\textwidth]{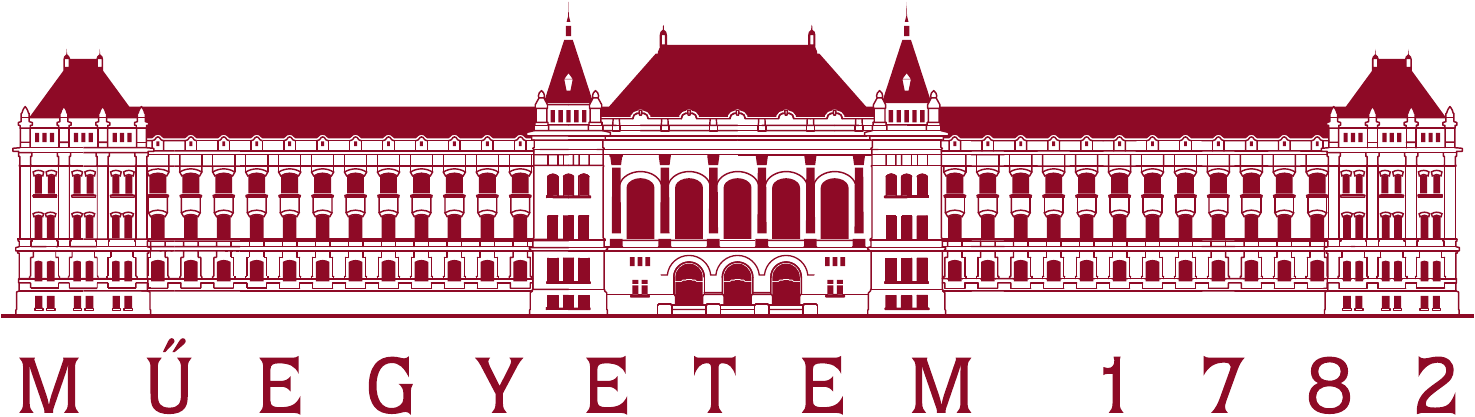}\\
           \large{Budapest University of Technology and Economics} \\
           \large{2025} 
               
       \end{center}
    \end{titlepage}
    
    \newpage
    \thispagestyle{empty}
    \null
    \vfill
    Ákos Urbán: \textit{The Pólya Web}, Master's Thesis, December 2025.
    \newpage
    \setcounter{page}{3}
    \thispagestyle{plain}
    \tableofcontents
    
    \newpage
    \thispagestyle{plain}
    \null
    \vspace{4cm}
    \begin{flushleft}
    \begin{large}
    \textbf{Acknowledgements}\\
    \end{large}
    \end{flushleft}
    I would like to take this opportunity to express my gratitude to my supervisors, \textit{Bálint Tóth} and \textit{Balázs Bárány}. For their help, advice and insights in more than two years of working together.
    In addition, for making probability theory my main interest during this time.
    I have learnt a lot from them.\\
    I would also like to thank \textit{Balázs Ráth} and \textit{Bálint Vető} for answering my questions. Even if I caught them off guard in the corridor or in front of building H.\\
    Specials thanks to \textit{Michelle} for her patience, support and understanding.
    \newpage

    \pagenumbering{arabic}
    
    \section{Introduction and motivation}
    \thispagestyle{plain}
    The notion of the Pólya Urn is well-known.
    Consider an urn with some blue and red balls in it.
    At each step take a uniform sample from the urn and add an extra ball to the urn of the same colour as the sample.
    Therefore, at each step the total number of balls in the urn increases by one.

    Formally, if we denote by $x_0$ and $y_0$ the initial number of blue and red balls and by $X_n$ and $Y_n$ the number of red and blue balls after step $n$ then it leads to the following Markov chain on $\N \times \N$
    \[
    (X_0,Y_0) = (x_0,y_0)
    \]
    and the transition kernel is
    \begin{equation} \label{eq:polya_urn_kernel}
        p_{x,y}(x+1,y) = \frac{x}{x+y} \qquad \text{and} \qquad
    p_{x,y}(x,y+1) = \frac{y}{x+y}.
    \end{equation}
    Notice that we can represent a Pólya Urn as a random walk on the integer lattice $\N\times\N$.
    We call this the Pólya Walk.

    We have a vast knowledge and applications of the Pólya Urn.
    Since the focus of the thesis is not to present those, we just recall the most important ones in this introduction. First of all, if we consider
    \[
    Z_n \defeq \frac{X_n}{X_n+Y_n}
    \]
    that is, the ratio of the blue balls to the total number of balls at time $n$, then $Z_n$ is a martingale with respect to the natural filtration of the Pólya Urn process.
    Since it is a non-negative and bounded martingale, it converges almost surely.
    Moreover, we know that the distribution on the limiting random variable is $\text{Beta}(x_0,y_0)$. Thus
    \[
    \lim_{n \to \infty} Z_n = Z \sim \text{Beta}(x_0,y_0) \quad \text{almost surely}.
    \]

    As mentioned before, probably everything has been said about the Pólya Urn since it was first introduced in \cite{polya_eggenberger:1923}.
    However, we hope that we can add something new with the following construction.
    Consider Pólya Walks started from every point of the positive quadrant of the integer lattice $\N \times \N$ in the following way.
    Every walk is independent from the other until they meet at a point.
    From that time the two walks coalesce and will be the same for the rest of the walk.
    In this way we get a web of coalescing Pólya Walks.
    We call this the \textit{Pólya Web} (\Cref{fig:samples_of_the_polya_web}).
    The main goal of the thesis is to investigate the properties of this object which turns out to be quite interesting.
    
    \begin{figure}[H]
      \centering
      \includegraphics[width=0.49\textwidth]{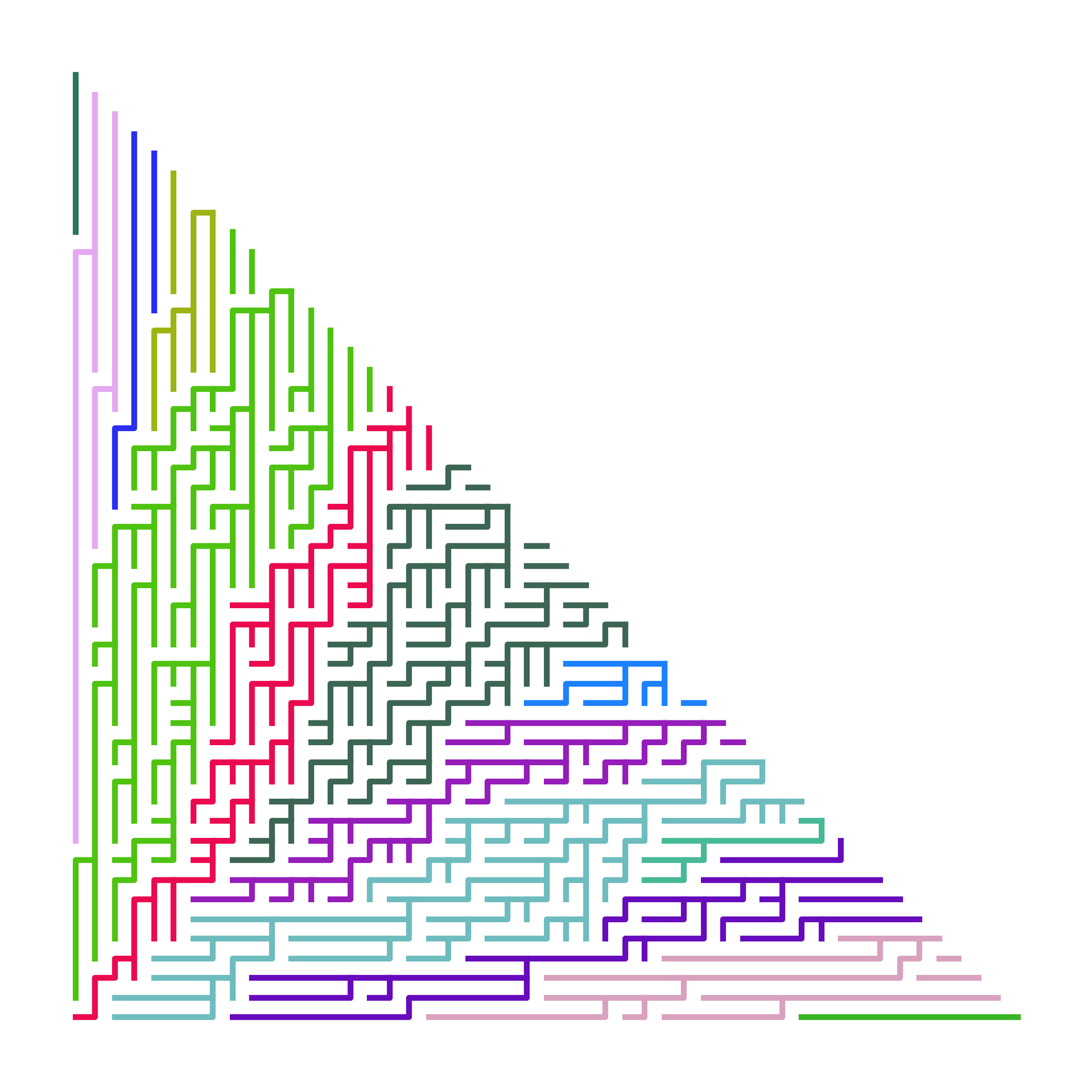}
      \hfill
      \includegraphics[width=0.49\textwidth]{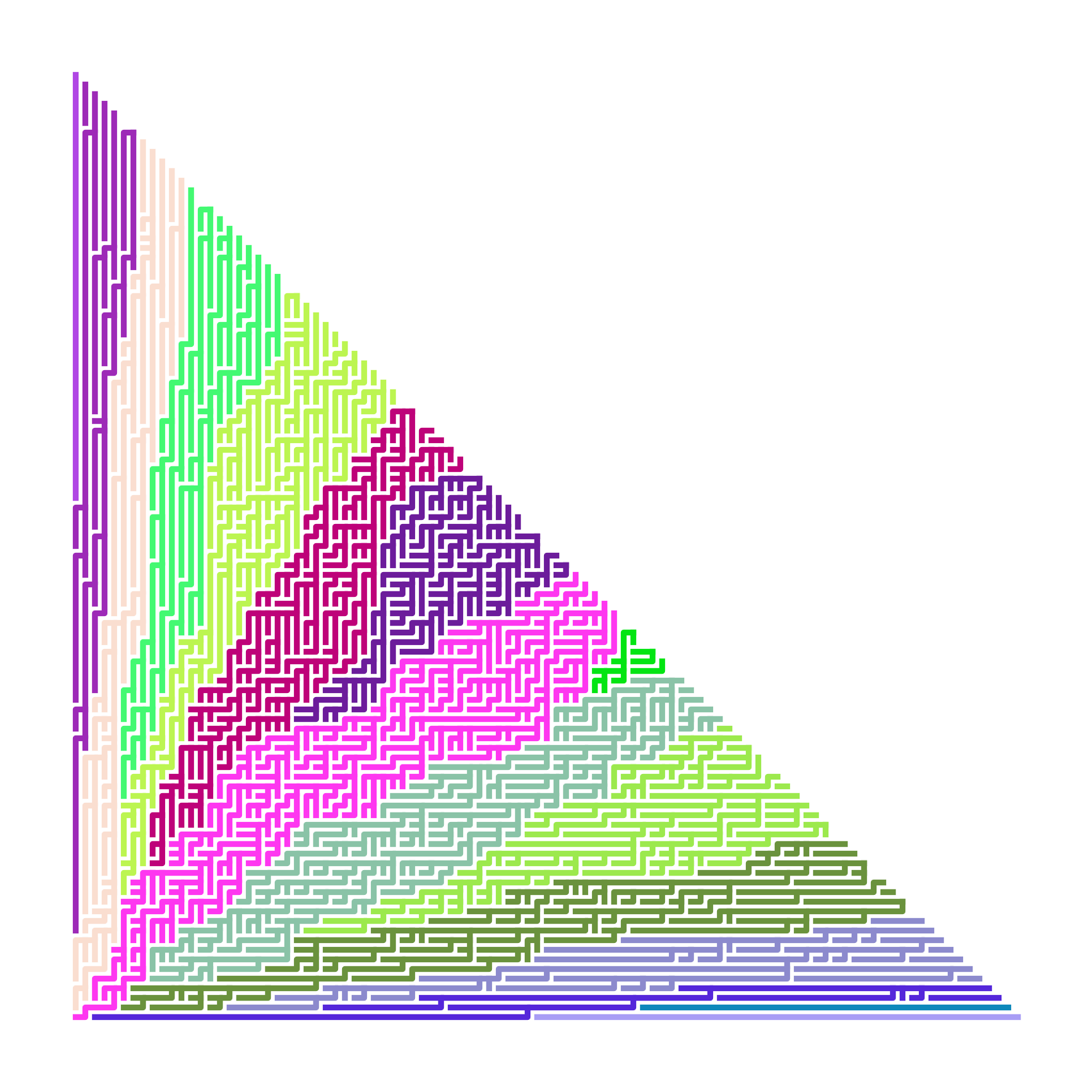}
      \caption{Two samples of the Pólya Web until different times (50 and 100). Two path have the same colour if they met in finite time (they might have met at a later point, not included in the figure).}
      \label{fig:samples_of_the_polya_web}
    \end{figure}
    One can notice that this construction is somewhat similar to the Random Walk Web introduced in \cite{toth_wendelin:1998} but the simple symmetric random walks are replaced with Pólya Walks.
    We do not need to mention that the Random Walk Web on its own is a very rich structure which lead to many interesting studies in the past and recently (a recent example is \cite{cannizzaro_hairer:2023}).
    Just to mention the Brownian Web and the Brownian Net.
    This makes the Pólya Web even more interesting. It is a more complicated structure compared to the Random Walk Web since the coalescing walks do not have the time homogeneity property.

    Finally notice that both the Random Walk Web and the Pólya Web are special cases of a more general web of upright oriented coalescing random walks.
    In winch at each point of a the positive quadrant of the integer lattice we decide about the next step by coin toss depending on the current point, and independently of the previous tosses.
    In the case of the Random Walk Web we toss an unbiased coin, while in the Pólya Web we toss a biased coin according to the probabilities of the transition kernel in \cref{eq:polya_urn_kernel}.
    Therefore, it is worth studying some properties of this more general web.
    First, we start with this and later we move on to study the concrete case of the Pólya Web.
    The basic probabilistic tools we use without citation can be found in \cite{durett:2010} and \cite{williams:1991}.
    \markright{\thesection. Introduction and motivation}

    \newpage

    \section{The web of upright oriented coalescing random walks}
    \thispagestyle{plain}
    We introduce the web of upright oriented coalescing random walks.
    The Pólya Web is a special type of this more general web.
    First we only studied the properties of the Pólya Web.
    However, it has turned out that many of our results expand to a more general type of web.
    In this section we state the basic definition of the web with some initial geometric property.

    \subsection{The web}
    Let us consider the sets
    \[
    \Lambda \defeq \N \times \N
    \qquad \text{and} \qquad
    \Omega \defeq \left\{(1,0), (0,1)\right\}^{\Lambda}.
    \]
    Let us introduce the following notation for the coalescing random walks in the general web.
    By $\|\cdot\|$ we denote the 1-norm of the vector. That is, $\norm{(a,b)} = |a|+|b|$. Therefore, for some $(k,l) = \lambda \in \Lambda$ it is just the sum of the coordinates: $\|\lambda\| = k + l$.
    A walk started from $\lambda \in \Lambda$ is defined as the following process
    \[
    S_\lambda(\|\lambda\|) \defeq \lambda \qquad \text{and} \qquad
    S_\lambda(n+1) \defeq S_{\lambda}(n) + \omega_{S_\lambda(n)} \quad (n \geq \| \lambda\|).
    \]
    Therefore, the walk starts at the point $\lambda$ and at step $n+1$ it takes a step in the direction of the currently allowed step $\omega_{S_\lambda(n)} \in \left\{(1,0),(0,1)\right\}$.
    Important to notice, that we use a universal time.
    That is time $n$ means the line $x+y=n$ on the plane. Also notice that by the definition the random walk started from $\lambda \in \Lambda$ is born at time $n = \| \lambda\|$.
    
    Then $\Omega$ is the set of all possible upright oriented coalescing walks.
    Therefore, if two walks started from different points of $\Lambda$ following the trajectory of some $\omega \in \Omega$ meet at a certain point, they will have the same path for the rest of the trajectory.

    Now consider independent Bernoulli measures at each point of $\Lambda$.
    Therefore, for every $\lambda \in \Lambda$ for some $0 \leq p_\lambda \leq 1$, let
    \[
    \mathbb{P}_\lambda((1,0)) = 1-\mathbb{P}_\lambda((0,1)) = p_\lambda.
    \]
    We define our probability measure of $\Omega$ as the direct product of the Bernoulli measures $\mathbb{P}_\lambda$,
    \[
    \mathbb{P} \defeq \bigotimes_{\lambda \in \Lambda} \mathbb{P}_\lambda.
    \]
    In this way we get a random web of upright oriented coalescing random walks in the positive quadrant of the lattice $\N \times \N$ (\Cref{fig:sample_general_web}).

    \begin{figure}[H]
      \centering
      \includegraphics[width=0.49\textwidth]{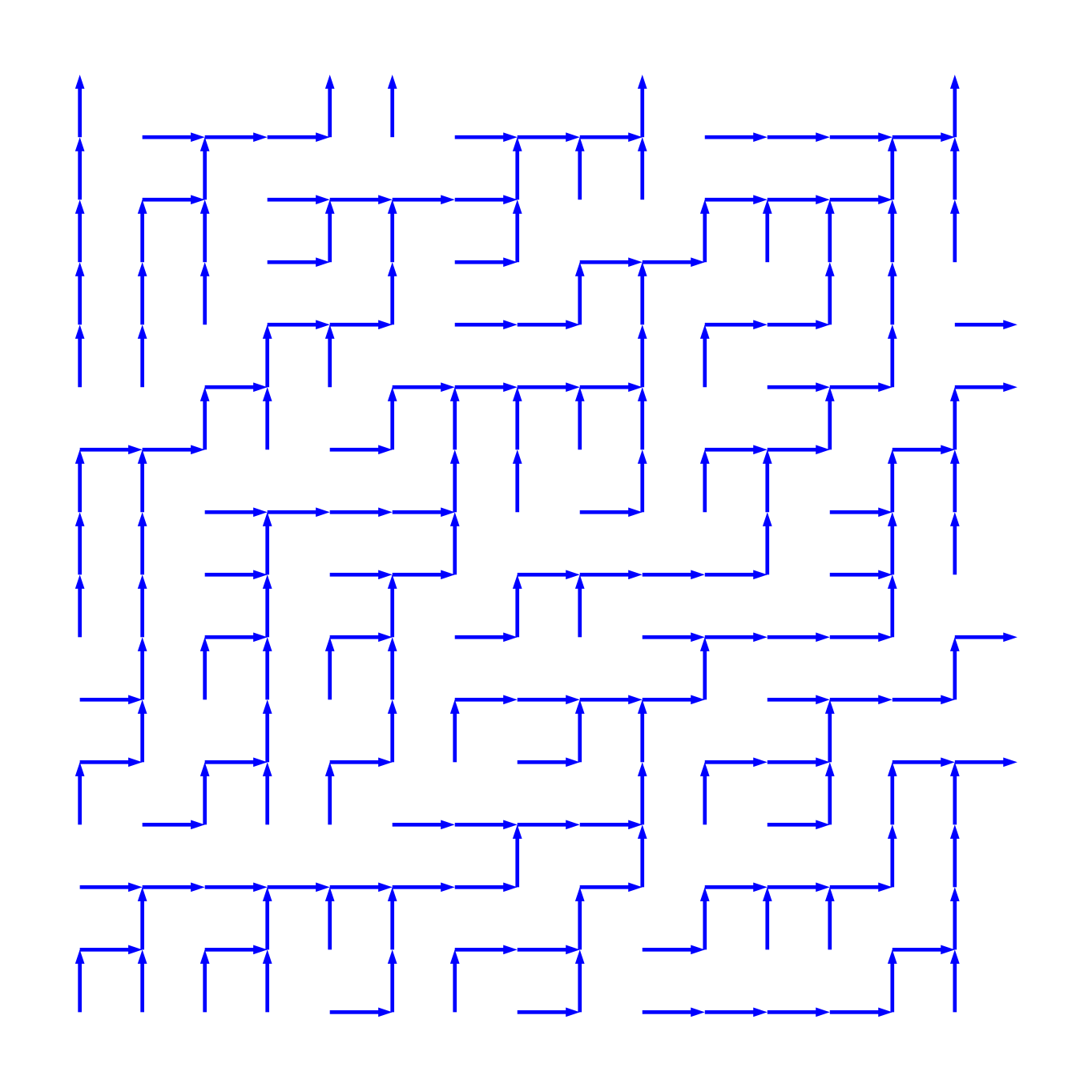}
      \caption{A sample of the web of upright oriented coalescing random walks.}
      \label{fig:sample_general_web}
    \end{figure}

    The Pólya Web is a certain special type of this web. Namely when we choose for $(k,l) = \lambda \in \Lambda$
    \[
    p_{(k,l)} = \frac{k}{k+l},
    \]
    while in the case of the Random Walk Web we have
    \[
    p_{(k,l)} = \frac{1}{2}.
    \]

    \subsection{The dual web}
    We define the dual web of the upright oriented coalescing random walks.
    Not surprisingly it will be a web of down-left oriented coalescing random walks.
    Let the points of the dual web be
    \[
    \Lambda_D \defeq \N \times \N + \left(\frac{1}{2},\frac{1}{2}\right).
    \]
    We define the dual web on $\Lambda_D$ so that a realization of the original web determines a realization of the dual web and conversely.
    That is for any $\lambda \in \Lambda$ and $\omega \in \Omega$ let for $\lambda_D = \lambda + \left(\frac{1}{2},\frac{1}{2}\right) \in \Lambda_D$
    \[
    \omega_{\lambda_D} \defeq -\omega_\lambda.
    \]
    That is if we take a step to the right at $\lambda$ we take a step to the left at $\lambda_D$ in the dual web. Similarly a step up at $\lambda$ determines a step down at $\lambda_D$ (\Cref{fig:arrows_on_the_web_and_its_dual}).

    \begin{figure}[H]
      \centering
      \includegraphics[width=0.45\textwidth]{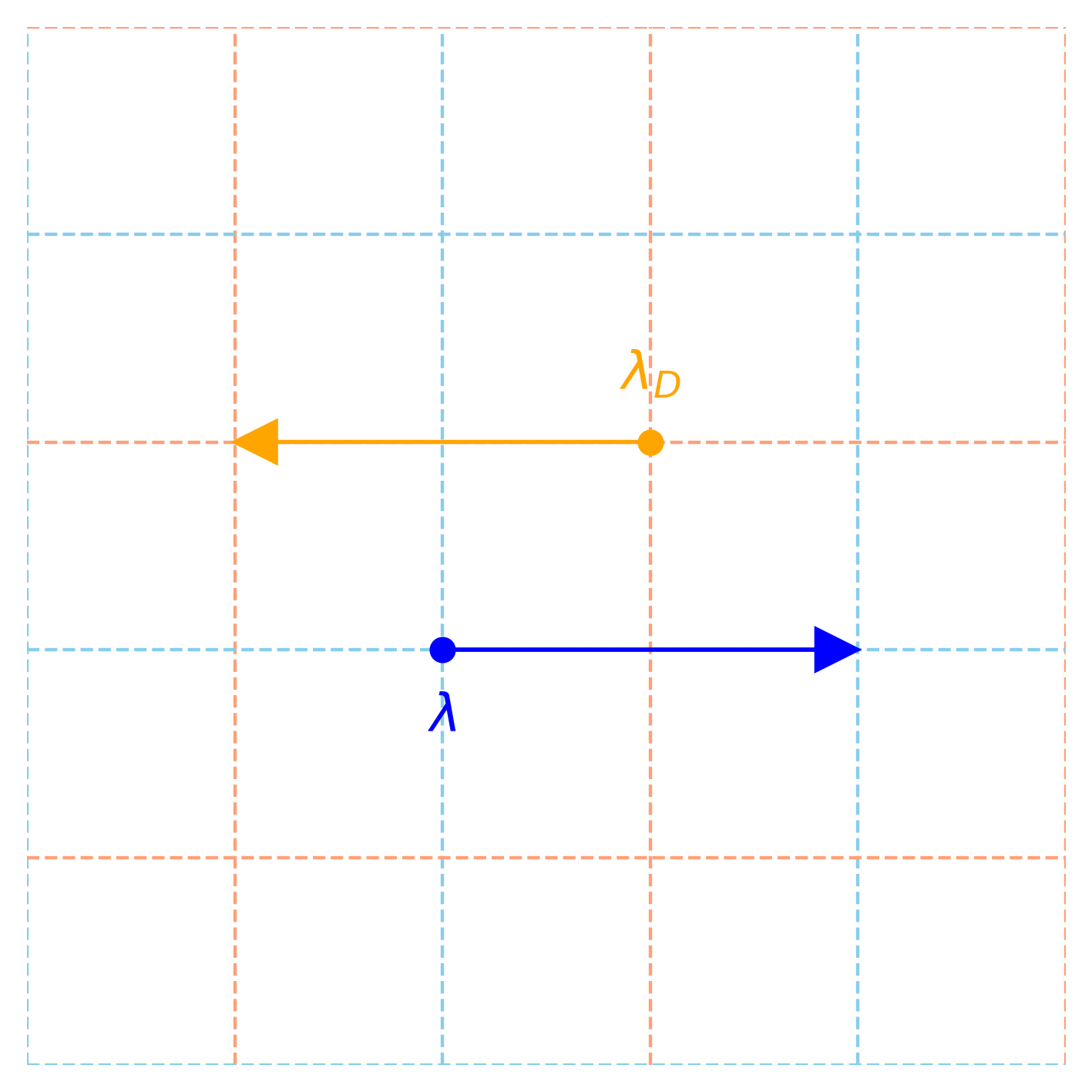}
      \hfill
      \includegraphics[width=0.45\textwidth]{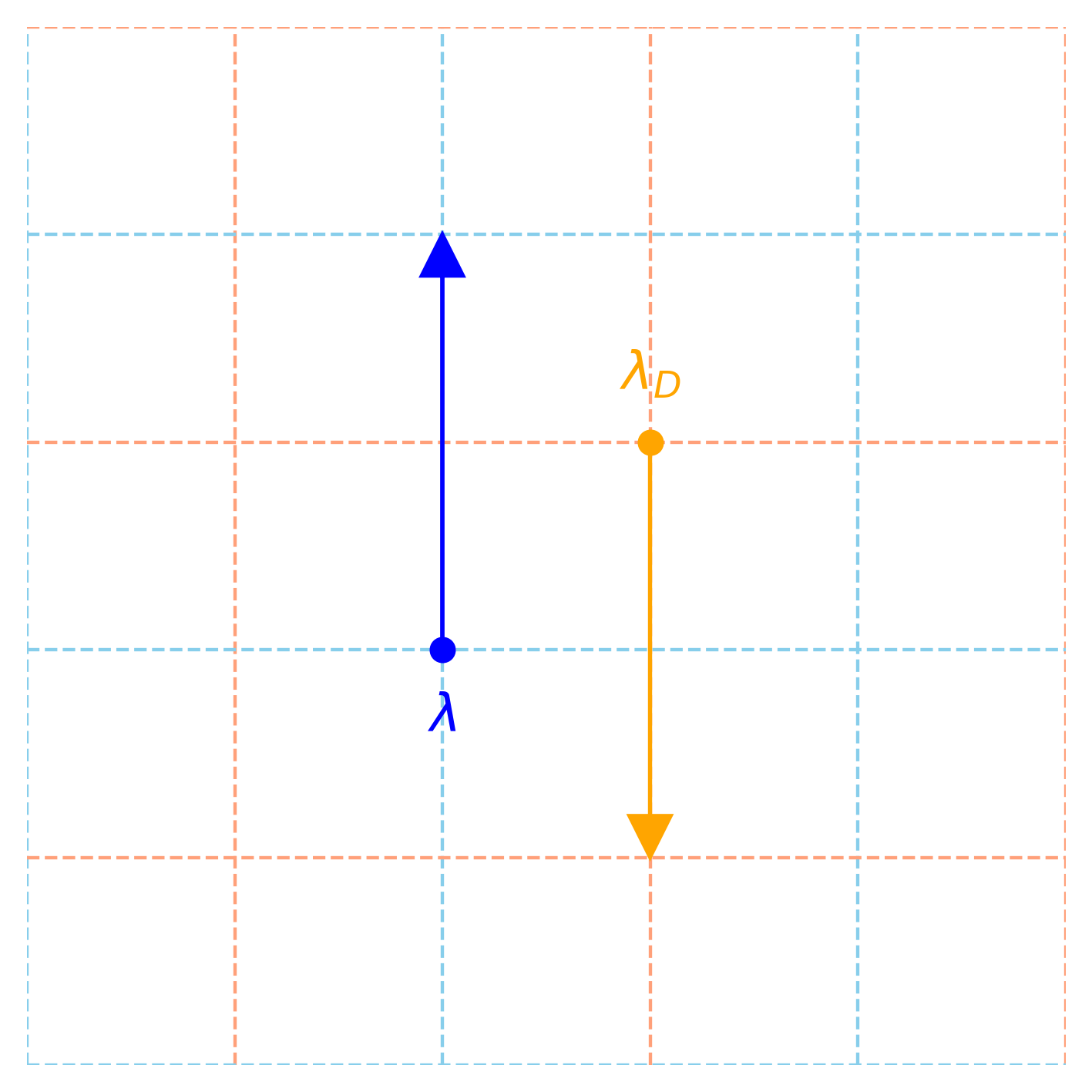}
      \caption{The two possible realizations at a point $\lambda$ of the web (blue) and the corresponding dual point $\lambda_D$ (orange).}
      \label{fig:arrows_on_the_web_and_its_dual}
    \end{figure}
    \begin{figure}[H]
      \centering
      \includegraphics[width=0.75\textwidth]{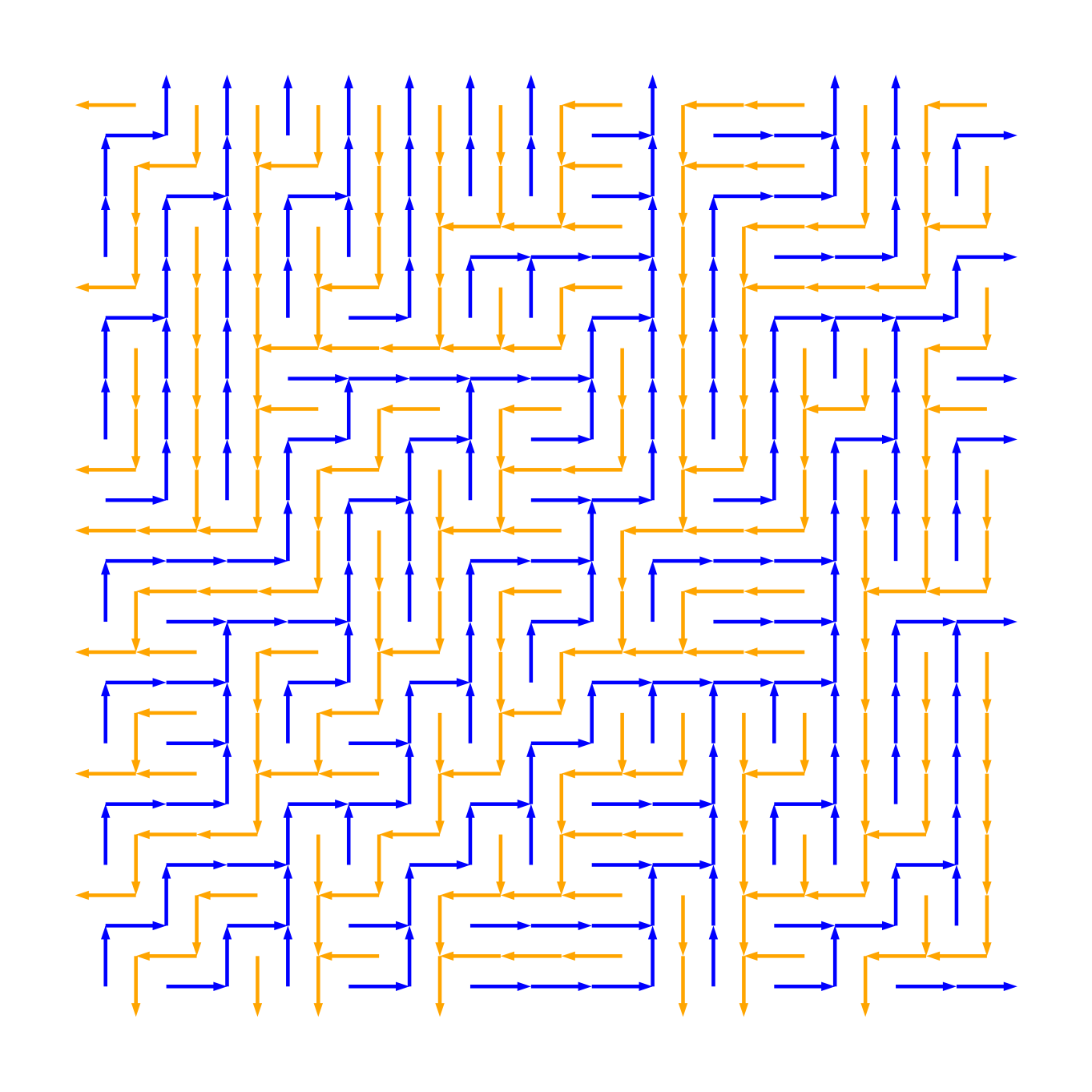}
      \caption{A sample of the web (blue) and its dual (orange).}
      \label{fig:sample_with_the_dual}
    \end{figure}

    We can similarly define the coalescing random walks on the dual lattice.
    Namely for $\lambda_D \in \Lambda_D$ let for $n \leq \norm{\lambda_D}$
    \[
    T_{\lambda_D}(\|\lambda\|) \defeq \lambda_D \qquad \text{and} \qquad
    T_{\lambda_D}(n-1) \defeq T_{\lambda_D}(n) + \omega_{T_{\lambda_D}(n)},
    \]
    until it crosses the $x$-axis or the $y$-axis. After that it is not defined.
    Notice that the dual walks only run for finite time since after some time they reach the boundary of the web and "exit" (\Cref{fig:sample_with_the_dual}).
    
    \subsection{The geometry of the web and its dual}
    The propositions of this subsection are trivial on their own, and their proofs are a bit unnecessarily formal. Yet we present them. Most of them should be clear from the picture of the web and its dual (\Cref{fig:sample_with_the_dual}).
    Therefore, we advise to skip the proofs if the propositions seem to be straightforward.
    
    An interesting question is to decide if two walks started from different points meet in finite time or not.
    We will investigate this issue later.
    First let us introduce a partial order on the sets $\Lambda$ and $\Lambda_D$.
    This will be a great help in the upcoming sections.
    \begin{dfn} \label{dfn:partial_order_on_Lambda}
        For any $(k_1,l_1) = \lambda_1, (k_2,l_2) = \lambda_2 \in \R^2$ we say that
        \[
        \lambda_1 \succeq \lambda_2,
        \]
        if
        \[
        k_1 \leq k_2 \qquad \text{and} \qquad l_1 \geq l_2.
        \]
    \end{dfn}
    Moreover, we use the notation $\lambda_1 \succ \lambda_2$, if at least one of the inequalities in \Cref{dfn:partial_order_on_Lambda} is strict (i.e. $\lambda_1 \neq \lambda_2$).
    And easy observation is the following proposition.
    \begin{prp} \label{prp:partial_order}
        $(\R^2,\succeq)$ is a partially ordered set.
    \end{prp}
    \begin{proof}
        Easy to check. Left to the reader.
    \end{proof}
    The reason we defined the partial order on $\Lambda$ this way is that the trajectories of the random walks preserve the order.
    Formally it is stated in the following proposition.
    \begin{prp} \label{prp:partial_order_and_trajectories}
        For any $\lambda_1 \succeq \lambda_2 \in \Lambda$ and for any $n \geq \max\left\{ \|\lambda_1\|, \| \lambda_2\|\right\}$
        \[
        S_{\lambda_1}(n) \succeq S_{\lambda_2}(n).
        \]
        Moreover, for any $\lambda_{D1} \succeq \lambda_{D2} \in \Lambda_D$ and for any $1 \leq n \leq \min\left\{ \|\lambda_{D1}\|, \| \lambda_{D2}\|\right\}$
        \[
        T_{\lambda_{D1}}(n) \succeq T_{\lambda_{D2}}(n).
        \]
    \end{prp}
    \begin{proof}
    Without loss of generality, we can suppose $\norm{\lambda_1} \leq \norm{\lambda_2}$.
    Then it is easy to see that
    \[
    S_{\lambda_1}(\norm{\lambda_2}) \succeq S_{\lambda_2}(\norm{\lambda_2}).
    \]
    If we have $S_{\lambda_1}(n) = S_{\lambda_2}(n)$ for some $n \geq \max\left\{ \|\lambda_1\|, \| \lambda_2\|\right\}$ then by definition of the walks
    \[
    S_{\lambda_1}(n+1) = S_{\lambda_2}(n+1).
    \]
    If we have $S_{\lambda_1}(n) \succ S_{\lambda_2}(n)$ for some $n \geq \max\left\{ \|\lambda_1\|, \| \lambda_2\|\right\}$ , then by definition
    \[
    S_{\lambda_1}(n+1)  - S_{\lambda_2}(n+1)= S_{\lambda_1}(n) - S_{\lambda_2}(n) + \omega_{S_{\lambda_1}(n)} - \omega_{S_{\lambda_2}(n)},
    \]
    where $S_{\lambda_1}(n) - S_{\lambda_2}(n)$ has first coordinate at most minus one and second coordinate at least one. Moreover, the possible values of $\omega_{S_{\lambda_1}(n)} - \omega_{S_{\lambda_2}(n)}$ are $(0,0),(1,-1)$ and $(-1,1)$. It follows that
    \[
    S_{\lambda_1}(n+1) \succeq S_{\lambda_2}(n+1).
    \]
    The argument for the dual walks is very similar.
    \end{proof}
    
    Clearly just by definition the walks on the web and the walks on the dual web can never cross each other (easy to see on \Cref{fig:arrows_on_the_web_and_its_dual}). This fact stated formally is the following.
    \begin{prp}\label{prp:the_dual_can_never_intersect}
        For any $\lambda \in \lambda$ and $N \geq \norm{\lambda}$ if $\lambda_{D_1} \succ S_\lambda(N) \succ \lambda_{D2}$, then for any $\norm{\lambda} \leq n \leq N$
        \[
        T_{\lambda_{D1}}(n) \succ S_{\lambda}(n) \succ T_{\lambda_{D2}}(n).
        \]
    \end{prp}
    \begin{proof}
    We show the upper bound. The lower bound can be shown similarly.
    Notice that they can only cross each other at time $n$ if
    \[
    T_{\lambda_{D1}}(n) = S_{\lambda}(n-1) + (1/2,1/2)
    \]
    since only one of their coordinates will be changed by one step.
    However, then
    \begin{align*}
        T_{\lambda_{D1}}(n-1) &= T_{\lambda_{D1}}(n) - \omega_{S_{\lambda}(n-1)}\\
        S_{\lambda_{D1}}(n-1) &= S_{\lambda_{D1}}(n) - \omega_{S_{\lambda}(n-1)},
    \end{align*}
    then it follows
    \[
    T_{\lambda_{D1}}(n-1) = T_{\lambda_{D1}}(n) - \omega_{S_{\lambda}(n-1)} \succ
    S_{\lambda}(n) - \omega_{S_{\lambda}(n-1)} = S_{\lambda}(n-1).
    \]
    Therefore, even in this case they will not intersect in the previous level.
    \end{proof}
    
    Finally we state an important lemma.
    It connects the paths of the walks on the web and its dual.
    Namely, two coalescing walks started from different points are in different points in the given time if there is a duals walk separating them.
    It is a very easy fact to see, but we present a formal proof.
    \begin{lmm} \label{lmm:paths_on_the_web_and_its_dual}
        For any $\lambda_1 \succeq \lambda_2 \in \Lambda$ and $N \geq \max\left\{ \|\lambda_1\|, \| \lambda_2\|\right\}$
        \[
        S_{\lambda_1}(N) \succ S_{\lambda_2}(N)
        \]
        if and only if there is a $\lambda_D \in \Lambda_D$ such that $\norm{\lambda_D} = N$
        and for any $\max\left\{ \|\lambda_1\|, \| \lambda_2\|\right\} \leq n \leq N$
        \[
        S_{\lambda_1}(n) \succ T_{\lambda_D}(n) \succ S_{\lambda_2}(n).
        \]
    \end{lmm}
    \begin{proof}
    Suppose that $S_{\lambda_1}(N) =(a_1,b_1) \succ (a_2,b_2) = S_{\lambda_2}(N)$. Then for
    $\lambda_D = (a_2,b_2) + (1/2,-1/2) \in \Lambda_D$ clearly $\norm{\lambda_D} = N$ and
    \[
    (a_1,b_1) \succ \lambda_{D} \succ (a_2,b_2).
    \]
    Therefore, since the walks on the web and the dual cannot cross each other by \Cref{prp:the_dual_can_never_intersect}, for any $\max\left\{ \|\lambda_1\|, \| \lambda_2\|\right\} \leq n \leq N$
    \[
    S_{\lambda_1}(n) \succ T_{\lambda_D}(n) \succ S_{\lambda_2}(n).
    \]

    Suppose that $S_{\lambda_1}(N) = S_{\lambda_2}(N) = (a,b)$. Then for $\lambda_{D1} = (a,b) + (1/2,-1/2)$ and $\lambda_{D_2} = (a,b)+(-1/2,1/2)$ we have $\norm{\lambda_{D1}} = \norm{\lambda_{D2}} = N$ and
    \[
    \lambda_{D1} \succ (a,b) \succ \lambda_{D2}.
    \]
    Then by \Cref{prp:partial_order_and_trajectories} and again the fact that walks on the web and its dual cannot cross paths for any $\tilde{\lambda}_{D1} \succeq \lambda_{D1}$ and 
    $\lambda_{D2} \succeq \tilde{\lambda}_{D2}$ for which $\norm{\tilde{\lambda}_{D1}}=\norm{\tilde{\lambda}_{D2}} = N$ we have for any
    $\max\left\{ \|\lambda_1\|, \| \lambda_2\|\right\} \leq n \leq N$
    \[
    T_{\tilde{\lambda}_{D1}}(n) \succeq T_{\lambda_{D1}} \succ S_{\lambda_1}(n) \qquad \text{and} \qquad
    S_{\lambda_2}(n) \succ T_{\lambda_{D2}} \succeq T_{\tilde{\lambda}_{D2}}(n).
    \]
    Moreover, notice there are no other $\tilde{\lambda}_{D} \in \Lambda_D$ such that $\norm{\lambda_D} = N$.

    \end{proof}
    
    \newpage

    \section{The strong law for the number of components} \label{section_3}
    \thispagestyle{plain}
    In this section we investigate the number of components of the infinite random graph which we get by starting upright oriented coalescing random walks from all the points $\lambda \in \Lambda$ with $\norm{\lambda} \leq n$.
    Let us denote the number of components of this graph by $C_n$.
    First, observe that since the walks are coalescing, it is enough to consider points $\lambda \in \Lambda$ for which $\norm{\lambda} = n$.
    We get the same $C_n$ this way (\Cref{fig:components_general}).

    \begin{figure}[H]
      \centering
      \includegraphics[width=0.6\textwidth]{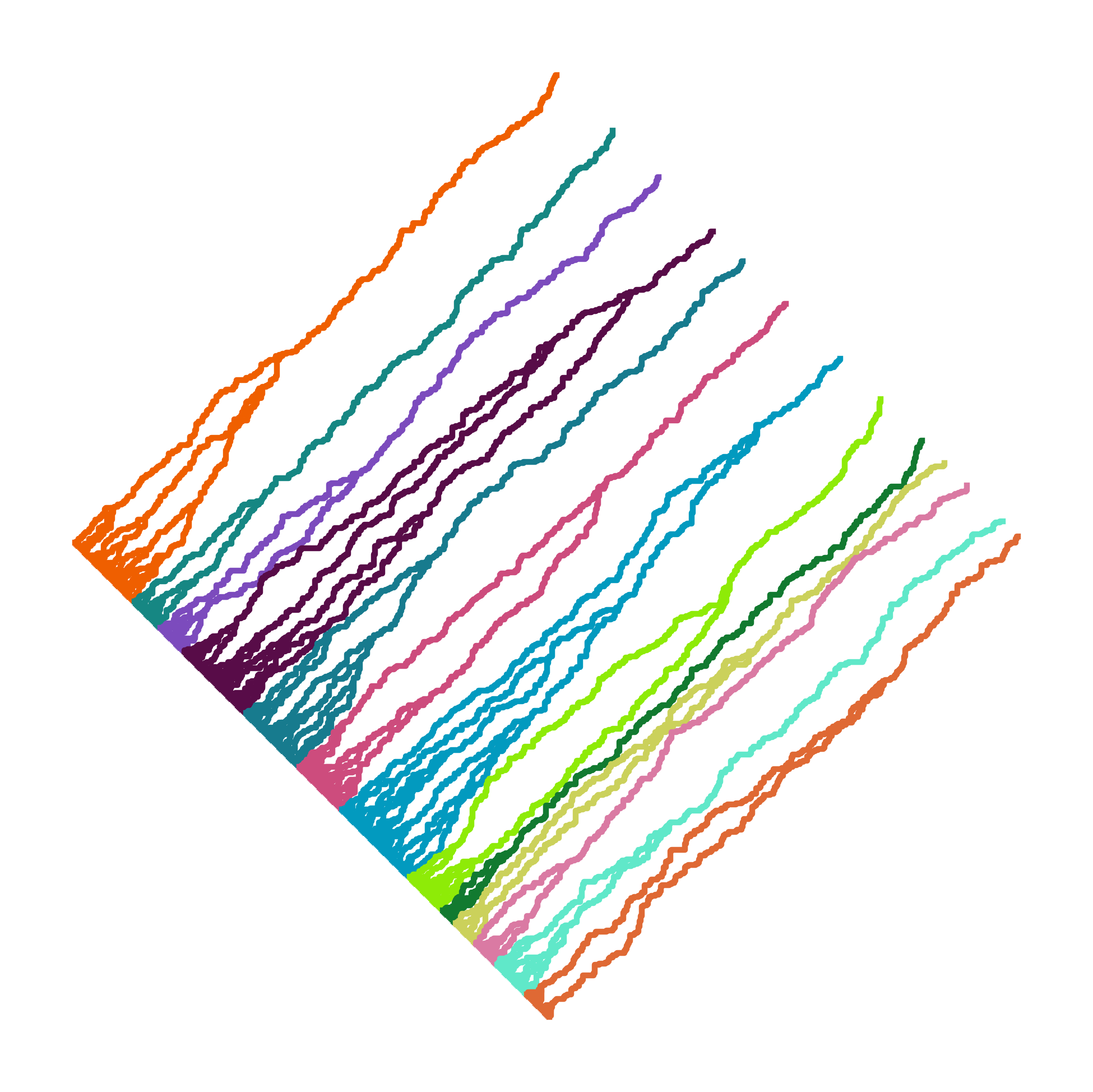}
      \caption{Coalescing random walks started from the same time. The different colours denote different components.}
      \label{fig:components_general}
    \end{figure}
    
    We have an exact formula for $C_n$, but first we need some notations.
    Therefore, let us introduce for $\lambda_1, \lambda_2 \in \Lambda$ the time of the meeting until time $n$
    \[
    \tau_{\lambda_1,\lambda_2}(n) \defeq
    \inf\left\{k \in \N: \;\max\left\{\| \lambda_1 \|, \|\lambda_2\|\right\} \leq k \leq n : \; S_{\lambda_1}(k) = S_{\lambda_2}(k)\right\},
    \]
    with the convention $\inf \emptyset = \infty$.
    Moreover, let
    \[
    \tau_{\lambda_1,\lambda_2} \defeq
    \inf\left\{k \in \N : \;k \geq \max\left\{\| \lambda_1 \|, \|\lambda_2\|\right\} : \; S_{\lambda_1}(k) = S_{\lambda_2}(k)\right\},
    \]
    the time of the meeting in finite time.
    Notice that both $\tau_{\lambda_1,\lambda_2}(n)$ and $\tau_{\lambda_1,\lambda_2}$ take values in $\N\cup \left\{ \infty \right\}$.
    As said before, we are interested in the probabilities
    \[
    \valseg{\tau_{\lambda_1,\lambda_2} < \infty} = 1 - \valseg{\tau_{\lambda_1,\lambda_2} = \infty}.
    \]
    A straightforward observation is that $\tau_{\lambda_1,\lambda_2}(n+1) \leq \tau_{\lambda_1,\lambda_2}(n)$ and thus the events $\left\{\tau_{\lambda_1,\lambda_2}(n) < \infty\right\}$ are increasing in $n$. Moreover,
    \[
    \left\{\tau_{\lambda_1,\lambda_2} < \infty\right\} = \bigcup_{n \geq \max\left\{\| \lambda_1 \|, \|\lambda_2\|\right\}} \left\{\tau_{\lambda_1,\lambda_2}(n) < \infty\right\}.
    \]
    Therefore, by the continuity of the probability measure
    \begin{equation} \label{eq:tau_as_a_limit}
        \begin{split}
        \valseg{\tau_{\lambda_1,\lambda_2} < \infty} &= \lim_{n \to \infty}
        \valseg{\tau_{\lambda_1,\lambda_2}(n) < \infty},\\
        \valseg{\tau_{\lambda_1,\lambda_2} = \infty} &= \lim_{n \to \infty}
        \valseg{\tau_{\lambda_1,\lambda_2}(n) = \infty}.
        \end{split}
    \end{equation}

    Now observe that the following equality holds
    \begin{equation} \label{eq:number_of_components}
    C_n = 1 + \sum_{k=1}^{n} \mathds{1}\left[\tau_{\lambda_{k-1},\lambda_{k}} = \infty \right],
    \qquad \text{where} \; \lambda_{k} = (k,n-k).
    \end{equation}
    It is indeed true since there will be at least one component but we get exactly as many additional components as many neighbouring points do not meet in finite time.

    Now we want to investigate the asymptotic behaviour of $C_n$.
    The key is that the random variables in the sum of \cref{eq:number_of_components} are  \textit{negatively associated}.
    We show this property in the upcoming subsections.
    But first we need to recall an important theorem from percolation theory.

    \subsection{Application of the van den Berg-Kesten-Reimer inequality}
    Let us introduce the following notations.
    Fix an $N \in \N$ and let $\Lambda_N = \left\{\lambda \in \Lambda: \, \norm{\lambda} \leq N \right\}$.
    Let $\Omega_N = \left\{(1,0), (0,1)\right\}^{\Lambda_N}$. Moreover, consider $\mathbb{P}_N$ as the restriction of $\mathbb{P}$ on the set $\Omega_N$. That is, $ \mathbb{P}_N \defeq \mathbb{P}\left|_{\Omega_N} \right.$. Clearly, it is just the restriction of the original web until time $N$.
    Then for $\omega \in \Omega_N$ and $K \subseteq \Lambda_N$ let
    \begin{equation} \label{eq:def_of_restriction_operation}
    [\omega]_K \defeq
    \left\{\tilde{\omega} \in \Omega_N : \, \omega\left|_K \right. = \tilde{\omega}\left|_K  \right. \right\}.
    \end{equation}
    Moreover, for $A,B \subseteq \Omega_N$ let
    \[
    A \square B \defeq
    \left\{\omega \in \Omega_N : \, \text{exist $K,L \subseteq \Lambda_N$ disjoint such that $ [\omega]_{K} \subseteq A$ and $ [\omega]_{L} \subseteq B$} \right\}.
    \]
    We should imagine $A \square B$ as the event in which the events $A$ and $B$ can be tested on disjoint sets. We also want to emphasise that $\square$ is not a set operation since the geometry of the product space matters. We list some of the properties of the $\square$ operation.
    By definition it is clearly commutative
    \[
    A \square B = B \square A.
    \]
    In general it is not associative
    \[
    A \square (B \square C) \neq (A \square B) \square C \qquad \text{in general}.
    \]
    An example for which the equality does not hold can be easily constructed.
    Moreover, the following inclusions can be easily seen.
    \begin{equation} \label{eq:box_and_union}
        (A \square B) \cup (A \square D) \cup (B \square C) \cup (B \square D)\subseteq (A \cup B) \square (C \cup D)
    \end{equation}
    and
    \[
        (A \cap B) \square (C \cap D) \subseteq
        (A \square B) \cap (A \square D) \cap (B \square C) \cap (B \square D).
    \]

    Clearly by definition
    \begin{equation} \label{eq:basic_box_and_intersection}
    A \square B \subseteq A \cap B.
    \end{equation}
    In general the opposite inclusion does not hold.
    There are some special cases when the two are equal.
    Namely, if we introduce the order $(1,0) \leq (0,1)$, and the induced partial order on $\Omega_N$ ($(\omega_1,\dots,\omega_N) \leq (\tilde{\omega}_1,\dots \tilde{\omega}_N)$ if $\omega_j \leq \tilde{\omega}_j, \; j=1,\dots,N$), then if $A$ is increasing and $B$ is decreasing, we have
    \[
    A \square B = A \cap B.
    \]
    In this case the following inequality holds.
    \begin{equation} \label{eq:we_want_to_prove_in_general}
    \mathbb{P}_N(A \square B) \leq \mathbb{P}_N(A) \cdot \mathbb{P}_N(B).
    \end{equation}
    This is due to the theorem of Harris in \cite{harris:1960}.
    \begin{thm}[Harris inequality]
        If $A$ and $B$ are of opposite monotonicity, then
        \[
        \mathbb{P}_N(A \cap B) \leq \mathbb{P}_N(A) \cdot \mathbb{P}_N(B).
        \]
    Moreover, if $A$ and $B$ are both increasing or decreasing
    \[
    \mathbb{P}_N(A) \cdot \mathbb{P}_N(B) \leq \mathbb{P}_N(A \cap B).
    \]
    \end{thm}
    However, we will show that in a certain case of ours the inequality in \cref{eq:we_want_to_prove_in_general} holds even if we do not assume anything about the monotonicity.
    For that we will use the following theorem in \cite{van_den_berg_kesten:1985}, \cite{reimer:2000} and \cite{borgs1999}.
    \begin{thm}[van den Berg-Kesten-Reimer inequality] \label{thm:van_den_berg_kesten}
        For any $A,B \subseteq \Omega_N$,
        \[
        \mathbb{P}_N\left(A \square B\right) \leq \mathbb{P}_N\left(A\right) \cdot\mathbb{P}_N\left(B\right).
        \]
    \end{thm}
    A consequence of this theorem with the Harris inequality for events $A$ and $B$ of the same monotonicity is the following.
    \[
    \mathbb{P}_N(A \square B) \leq \mathbb{P}_N(A) \cdot \mathbb{P}_N(B) \leq \mathbb{P}_N(A \cap B).
    \]

    Now we focus on our case and do not assume anything about the monotonicity of the events induced by the partial order.
    We want to apply the theorem for certain type of events.
    For this let us fix the integers $1 \leq n \leq N < \infty$.
    Then, let us denote for $j=1,\dots,n$
    \begin{equation} \label{eq:def_of_the_indicators_for_taus}
    I_j(N) \defeq \left\{\tau_{\lambda_{j},\lambda_{j-1}}(N) = \infty \right\}
    \qquad
    \text{where $\lambda_j = (j,n-j)$}.
    \end{equation}
    That is, we consider all pairs of the walks started from time $n$ and the events that they do not coalesce until time $N$. For a set $A \subseteq \{1,\dots,n \}$ let us use the notation
    \[
    I_A(N) \defeq \bigcap_{j \in A}I_j(A).
    \]
    Then the following important lemma holds.
    \begin{lmm} \label{lmm:important_lemma}
         For any disjoint sets $A,B \subseteq \{1,\dots,n\}$
        \[
            I_A(N) \cap I_B(N) = I_A(N) \square I_{B}(N).
        \]
    \end{lmm}
    \begin{proof}
        Let $K = \left\{p^{(j)}_{m}\right\}_{m=n,\dots,N, \; j\in A} \subseteq \Lambda_N$ be a set of those points (\Cref{fig:separation_01}) on which for all $j \in A$ there is an allowed path in the dual walks (1st condition later), started from time $N$ (2nd condition later) such that it separates the points $\lambda_j$ and $\lambda_{j-1}$ (3rd condition later).
        Moreover, for different indices in $A$ the paths are disjoint (4th condition later).

        Formally the set of those points in $\Lambda_N$
        \begin{enumerate}
            \item $p^{(j)}_{m+1}-p^{(j)}_{m} \in \left\{(1,0),(0,1)\right\}$ for all $j\in A, \; m=n, \dots,N$,
            \item $\norm{p^{(j)}_N +(1/2,1/2)} = N$ for all $j \in A$
            \item $\lambda_{s-1} \succ p^{(j)}_{n} + (1/2,1/2) \succ \lambda_{s}$ for all $j \in A$
            \item if $j_1 \neq j_2$, then $p^{(j_1)}_m \neq p^{(j_2)}_m$ for all $m=1,\dots,N$
        \end{enumerate}
        \begin{figure}[H]
          \centering
          \includegraphics[width=0.8\textwidth]{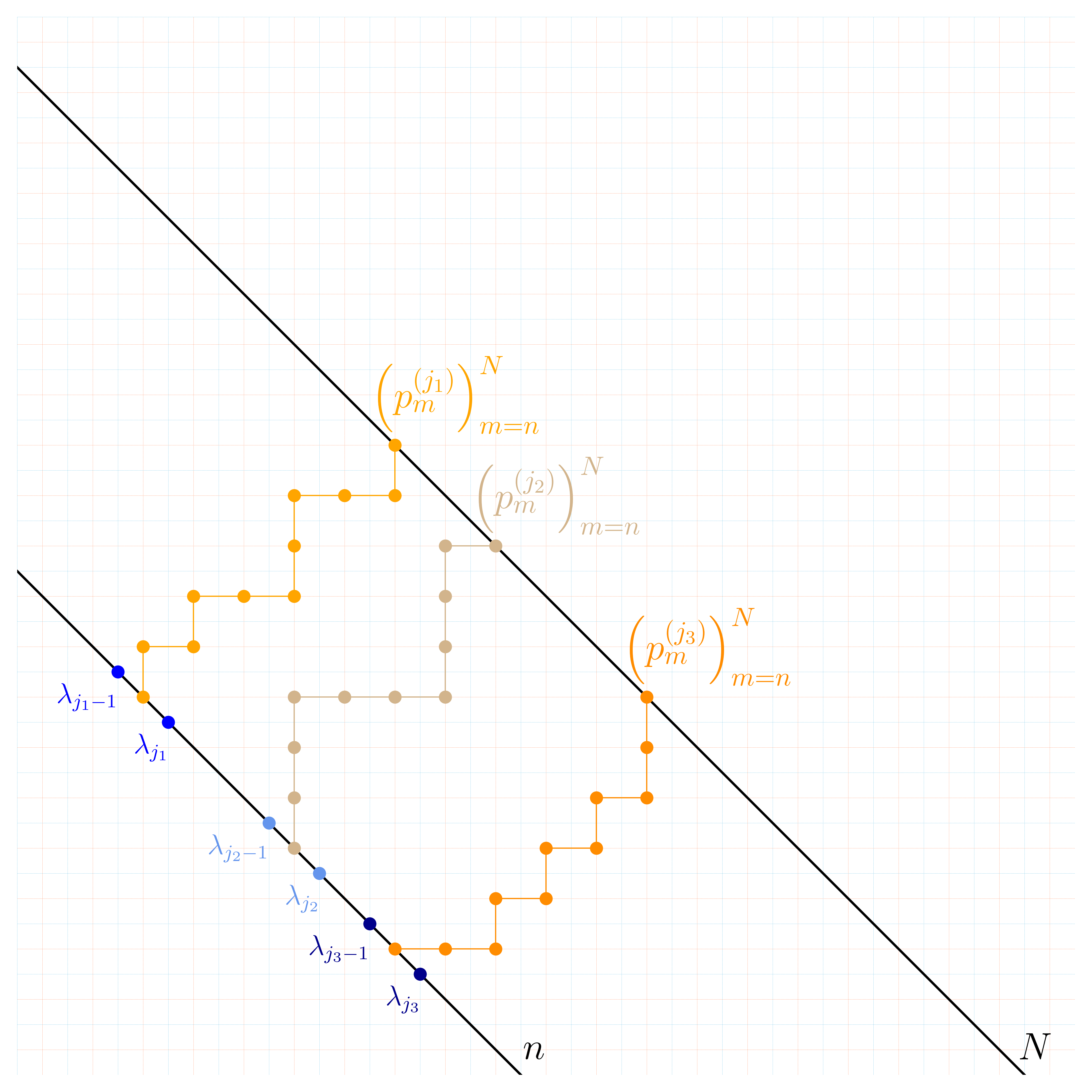}
          \caption{The points on which there is an allowed path in the dual web for the case $A=\left\{j_1,j_2,j_3 \right\}$. The points are on the web but we visualized the corresponding points on the dual web since it captures the meaning better.}
          \label{fig:separation_01}
        \end{figure}
        Then let $\mathcal{K}$ denote the set of all such $K \subseteq \Lambda_N$.
        Then clearly by \Cref{lmm:paths_on_the_web_and_its_dual} and the fact that one walk on the dual web can only separate one pair of points since the dual walks are coalescing too, it follows that
        \[
        I_A(N) = \bigcup_{K \in \mathcal{K}}
        \left\{\omega \in \Omega_N : \; p^{(j)}_{m+1}-p^{(j)}_m = \omega_{p^{(j)}_m}, \; \text{for all $p^{(j)}_m \in K$}\right\}
        \]
        Then if we define the set $\mathcal{L}$ of those sets $L = \left\{q^{(i)}_{m}\right\}_{m=n,\dots,N, \; i\in B} \subseteq \lambda_N $ for which the analogues of the previous conditions hold, we have
        \[
        I_B(N) = \bigcup_{L \in \mathcal{L}}
        \left\{\omega \in \Omega_N : \; q^{(i)}_{m+1}-q^{(i)}_m = \omega_{q^{(i)}_m}, \; \text{for all $q^{(i)}_m \in L$}\right\}.
        \]
        Then comparing the two
        \begin{multline*}
        I_A(N)\cap I_B(N)\\
        =
        \bigcup_{K \in \mathcal{K}, \; L \in \mathcal{L}}
        \left\{\omega \in \Omega_N : \; p^{(j)}_{m+1}-p^{(j)}_m = \omega_{p^{(j)}_m}, \; q^{(i)}_{m+1}-q^{(i)}_m = \omega_{q^{(i)}_m}\; \text{for all $p^{(j)}_m \in K, q^{(i)}_m \in L$}\right\}.
        \end{multline*}
        Moreover, notice that using the argument that one dual walk can only separate one pair and the fact that $A \cap B = \emptyset$, the above equality reduces to the case (\Cref{fig:separation_02})
        \begin{multline*}
        I_A(N)\cap I_B(N)\\
        =
        \bigcup_{\substack{K \in \mathcal{K}, \; L \in \mathcal{L},\\ K \cap L = \emptyset}}
        \left\{\omega \in \Omega_N : \; p^{(j)}_{m+1}-p^{(j)}_m = \omega_{p^{(j)}_m}, \; q^{(i)}_{m+1}-q^{(i)}_m = \omega_{q^{(i)}_m}\; \text{for all $p^{(j)}_m \in K, q^{(i)}_m \in L$}\right\}.
        \end{multline*}
        \begin{figure}[H]
          \centering
          \includegraphics[width=0.8\textwidth]{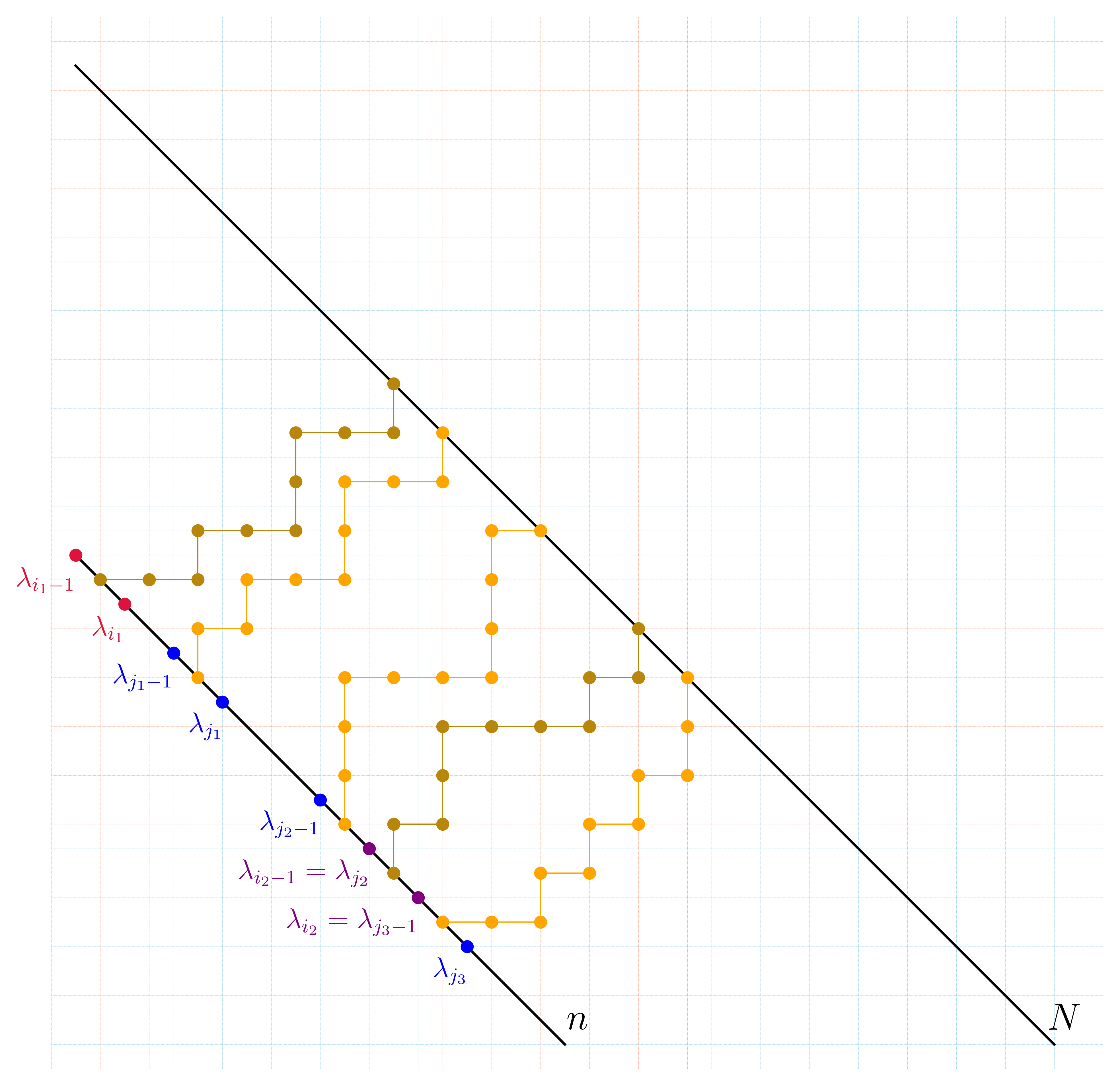}
          \caption{The points can only be separated with disjoint paths. In this case $A=\left\{j_1,j_2,j_3\right\}$ and $B=\left\{i_1,i_2\right\}$ having $i_2 = j_{3}-1$ and $j_2 = i_2-1$.}
          \label{fig:separation_02}
        \end{figure}
        However, by the definition in \cref{eq:def_of_restriction_operation} clearly
        \[
        \begin{split}
        I_A(N)\cap I_B(N)
        &\subseteq
        \bigcup_{\substack{K \in \mathcal{K}, \; L \in \mathcal{L},\\ K \cap L = \emptyset}}
        \left\{\omega \in \Omega_N : \; [\omega]_K \subseteq I_A(N), \; [\omega]_L \subseteq I_B(N)\right\}\\
        &\subseteq
        \bigcup_{\substack{K \in \Lambda, \; L \in \Lambda,\\ K \cap L = \emptyset}}
        \left\{\omega \in \Omega_N : \; [\omega]_K \subseteq I_A(N), \; [\omega]_L \subseteq I_B(N)\right\}\\
        &= I_A(N) \square I_B(N).
        \end{split}
        \]
        The other inclusion holds by definition (\cref{eq:basic_box_and_intersection}).
    \end{proof}
    Now, we want to apply this equality to increasing events of those defined in \cref{eq:def_of_the_indicators_for_taus}.
    First, we need recall the definition of increasing events. This is not related to the monotonicity induced by the partial order on $\Omega$ discussed before.
    For this (and for the following propositions) we use the following partial order on $\R^d$
    \begin{equation} \label{eq:def_of_partial_order_in_R^d}
    (x_1,\dots,x_d) \geq (y_1,\dots,y_d), \quad \text{if $\quad x_j \geq y_j$ \; for all $j=1,\dots,d$}.
    \end{equation}
    Then the event $A$ is increasing event of the events $A_1, \dots,A_n$, if its indicator $\mathds{1}_A$ is and increasing function of the indicators of $\mathds{1}_{A_1},\dots,\mathds{1}_{A_n}$. That is
    \[
    \mathds{1}_{A} = f(\mathds{1}_{A_1}, \dots, \mathds{1}_{A_n}),
    \]
    where $f : \{0,1\}^n \to \{0,1\}$ in an increasing function with respect to the partial order defined in \cref{eq:def_of_partial_order_in_R^d}.

    Now if $\mathcal{A}$ is an increasing event of $I_j(N) \; (j \in A)$ for some $A \subseteq \{1,\dots,n\}$, then it can be written as
    \[
    \mathcal{A} = \bigcup_{k=1}^{a}I_{A_k}(N),
    \]
    where $a$ is some integer and $A_k \subseteq A$ for all $k=1,\dots,a$.
        This follows from the fact that there is a one to one association between the increasing events of the events $I_j(N)$ and the elements of (the distributive lattice)
        \[
        \mathcal{M} = \left\{K \subseteq \{1,\dots,n\}: \text{if $K \cap L=\emptyset$, then $K \nsubseteq L \nsubseteq K$}\right\}.
        \]
        Then if
        \[
        \mathcal{M}_A = \left\{K \in \mathcal{M}: I_K(N) \subseteq A, \text{if $L \subsetneqq K$, then $I_L(N)\nsubseteq A$} \right\},
        \]
        we have that
        \[
        \mathcal{A} = \bigcup_{K\in\mathcal{M}_A}I_K(N).
        \]
    \begin{prp} \label{prp:square_operation_and_intersection}
        For any disjoint sets $A,B \subseteq \{1,\dots,n\}$, if $\mathcal{A}$ and $\mathcal{B}$ are increasing events of $I_j(N) \; (j \in A)$ and $I_i(N) \; (i \in B)$ respectively, then
        \[
        \mathcal{A} \cap \mathcal{B} = \mathcal{A} \square \mathcal{B}.
        \]
    \end{prp}
    \begin{proof}
    Since the events are both increasing, they can be written as
    \[
    \mathcal{A} = \bigcup_{k=1}^{a}I_{A_k}(N)
    \qquad \text{and} \qquad
    \mathcal{B} = \bigcup_{l=1}^{b}I_{B_l}(N).
    \]
    Therefore,
    \[
    \mathcal{A} \cap \mathcal{B} 
    = \bigcup_{k=1}^{a} \bigcup_{l=1}^{b} I_{A_k}(N) \cap I_{B_l}(N).
    \]
    Notice, since $A \cap B = \emptyset$, for all $1 \leq k \leq a$ and $1 \leq l \leq b$, we have $A_k \cap B_l = \emptyset$.
    Therefore, by \Cref{lmm:important_lemma} it follows
    \[
    \mathcal{A} \cap \mathcal{B} =\bigcup_{k=1}^{a} \bigcup_{l=1}^{b} I_{A_k}(N) \square I_{B_l}(N).
    \]
    Now using the elementary property of the $\square$ operation in \cref{eq:box_and_union}
    \[
    \mathcal{A} \cap \mathcal{B} \subseteq \left(\bigcup_{k=1}^{a}I_{A_k}(N) \right) \square \left(\bigcup_{l=1}^{b}I_{B_l}(N)\right)
    = \mathcal{A} \square \mathcal{B}.
    \]
    The other inclusion holds again by definition (\cref{eq:basic_box_and_intersection}).
    \end{proof}
    We can apply the van den Berg-Kesten-Reimer inequality.
    \begin{crl} \label{crl:bound_on_the_intersection}
    For the events in \Cref{prp:square_operation_and_intersection}
    \[
        \mathbb{P}_N(\mathcal{A} \cap \mathcal{B}) \leq
        \mathbb{P}_N(\mathcal{A}) \cdot \mathbb{P}_N(\mathcal{B}).
    \]
    \end{crl}
    \begin{proof}
        Applying first \Cref{prp:square_operation_and_intersection} and then \Cref{thm:van_den_berg_kesten}
        \[
        \mathbb{P}_N(\mathcal{A} \cap \mathcal{B}) =
        \mathbb{P}_N(\mathcal{A} \square \mathcal{B})
        \leq
        \mathbb{P}_N(\mathcal{A}) \cdot \mathbb{P}_N(\mathcal{B}).
        \]
    \end{proof}
    
    \subsection{Negative association and the limit}

    First we need to recall the definition of negatively associated random variables.
    \begin{dfn}[Negative association] \label{dfn:NA}
        The random variables $X_1, \dots, X_n$ are negatively associated if for any disjoint subsets $A,B \subseteq \left\{1,\dots,n\right\}$ and functions $f:\R^{|A|} \to \R$ and $ g:\R^{|B|} \to \R$ both monotone increasing or decreasing
        \[
        \varhato{f\left(X_j : \; j \in A\right) \cdot g\left(X_i : \; i \in B\right)} \leq
        \varhato{f\left(X_j : \; j \in A\right)} \cdot \varhato{g\left(X_i : \; i \in B\right)}
        \]
    \end{dfn}

    Negative association is a very strong property (see \cite{wajc:2017}, \cite{dubhashi-ranjan:1998} and \cite{shao:200}).
    It is almost as good as if the random variables were independent.
    For example the bound
    \begin{equation} \label{eq:bound_on_the_marginals}
        \varhato{\prod_{j \in A}X_j} \leq \prod_{j\in A}\varhato{X_j}
    \end{equation}
    follow for any $A \subseteq \{1,\dots,n\}$.
    Moreover, it provides many large deviation bounds.
    Since we are investigating the sum of indicators, the following large deviation theorem (see \cite{wajc:2017} or \cite{dubhashi-ranjan:1998}) is that one which will come handy.
    \begin{prp} \label{thm:large_deviation_for_the_sum}
        Let $X_1,\dots,X_n$ be negatively associated random variables taking values in $\left\{0,1\right\}$ and denote $S_n = X_1 + X_2 + \dots + X_n$.
        Then for any $0 < \delta < 1$ the following two inequality hold.
        \[
        \valseg{S_n \geq (1+\delta)\varhato{S_n}} \leq
        \left( \frac{e^{\delta}}{(1+\delta)^{1+\delta}}\right)^{\varhato{S_n}}
        \]
        and
        \[
        \valseg{S_n \leq (1-\delta)\varhato{S_n}} \leq
        \left( \frac{e^{-\delta}}{(1-\delta)^{1-\delta}}\right)^{\varhato{S_n}}.
        \]
    \end{prp}
    The corollary follows.
    \begin{crl} \label{crl:almost_sure_convergence_for_NA}
        Let $X_1,\dots,X_n$ be negatively associated random variables and denote $S_n = X_1 + \dots + X_n$.
        If for some $C,c > 0$
        \[
        \varhato{S_n} \geq C \cdot n^c(1+o(1)),
        \]
        then
        \[
        \lim_{n \to \infty}\frac{S_n}{\varhato{S_n}} = 1 \qquad \text{almost surely}.
        \]
    \end{crl}
    \begin{proof}
        By \Cref{thm:large_deviation_for_the_sum} for any $0 < \delta < 1$
        \[
        \valseg{\frac{S_n}{\varhato{S_n}} \geq 1+\delta}\leq
        \left( \frac{e^{\delta}}{(1+\delta)^{1+\delta}}\right)^{\varhato{S_n}}
        \; \text{and} \quad
        \valseg{\frac{S_n}{\varhato{S_n}} \leq 1-\delta} \leq
        \left( \frac{e^{-\delta}}{(1-\delta)^{1-\delta}}\right)^{\varhato{S_n}}.
        \]
        However, it is easy to check that for $0 < \delta < 1$
        \[
        \frac{e^{\delta}}{(1+\delta)^{1+\delta}}, \frac{e^{-\delta}}{(1-\delta)^{1-\delta}} \in (0,1).
        \]
        Thus if we denote any of the two by $q \in (0,1)$, then we have  by the assumption in the statement of the theorem
        \[
        \sum_{n=1}^{\infty}q^{\varhato{S_n}} \leq \sum_{n=1}^{\infty}q^{C \cdot n^c(1+o(1))} < \infty.
        \]
        Therefore, by the first Borel-Cantelli lemma the statement follows.
    \end{proof}
    Now, we want to apply this result to the number of components.
    First we need to show that the appropriate random variables are negatively associated.
    For this we will use the result of the previous subsection.
    \begin{prp} \label{prp:indicators_are_negatively_associated}
        The indicators of the events defined $\mathds{1}_{I_j(N)} \; (j=1,\dots,n)$ in \cref{eq:def_of_the_indicators_for_taus} are negatively associated.
    \end{prp}
    \begin{proof}
    Let $A,B \subseteq \{1,\dots,n\}$ be disjoint sets.
    Then we need to show that \Cref{dfn:NA} holds.
    Notice that it is enough to show for monotone increasing functions $f:\{0,1\}^{|A|} \to \R$ and $g:\{0,1\}^{|B|} \to \R$ since taking the negative sign the inequality does not change but the functions become monotone decreasing.
    Moreover, it is enough to show that for $f:\{0,1\}^{|A|} \to \{0,1\}$ and $g:\{0,1\}^{|B|} \to \{0,1\}$ monotone increasing since all monotone functions are linear combinations of such with non-negative coefficients.
    That is, it is enough to show for $\mathcal{A}$ and $\mathcal{B}$ both increasing events of $I_j(N) \; (j\in A)$ and $I_i(N) \; (i\in B)$ respectively that
    \[
    \valseg{\mathcal{A} \cap \mathcal{B}} \leq \valseg{\mathcal{A}} \cdot \valseg{\mathcal{B}}.
    \]
    However, since $\mathbb{P}_N$ was just the restrictions of $\mathbb{P}$ on $\Omega_N$, we have by \Cref{crl:bound_on_the_intersection}
    \[
     \valseg{\mathcal{A} \cap \mathcal{B}} =
     \mathbb{P}_N(\mathcal{A} \cap \mathcal{B}) \leq
     \mathbb{P}_N(\mathcal{A}) \cdot \mathbb{P}_N(\mathcal{B}) =
     \valseg{\mathcal{A}} \cdot \valseg{\mathcal{B}}.
    \]
    \end{proof}
    Now notice that for any $j=1,\dots,n$ trivially, if
    \[
    I_j \defeq \left\{\tau_{\lambda_{j},\lambda_{j-1}} = \infty\right\},
    \]
    then
    \begin{equation} \label{eq:def_of_the_limit_of_the_indicators}
    \lim_{N \to \infty}\mathds{1}_{I_j(N)} = \mathds{1}_{ I_j} \qquad \text{almost surely}.
    \end{equation}
    Then the following corollary of the previous \Cref{prp:indicators_are_negatively_associated}
    follows.
    \begin{crl} \label{crl:limit_indicators_are_negatively_associated}
        The indicators $\mathds{1}_{I_j} \; (j=1,\dots,n)$ of the events defined in \cref{eq:def_of_the_limit_of_the_indicators} are negatively associated.
    \end{crl}
    \begin{proof}
        Let $A,B \subseteq \{1,\dots,n\}$ be disjoints sets and $f:\{0,1\}^{|A|} \to \R$ and $g:\{0,1\}^{|B|} \to \R$ both monotone increasing or decreasing.
        Then notice
        \begin{align*}
        \lim_{N \to \infty} f\left(\mathds{1}_{I_j(N)}: \; j \in A\right) &=
        f\left(\mathds{1}_{I_j}: \; j \in A\right) \qquad \text{almost surely}\\
        \lim_{N \to \infty} g\left(\mathds{1}_{I_i(N)}: \; i \in B\right) &=
        g\big(\mathds{1}_{I_i} :\; i \in B\big) \qquad \text{almost surely}.
        \end{align*}
        Moreover, clearly $f$ and $g$ are bounded functions. Then by dominated convergence and \Cref{prp:indicators_are_negatively_associated}
        \[
        \begin{split}
        \varhato{ f\left(\mathds{1}_{I_j}: \; j \in A\right) \cdot g\big(\mathds{1}_{I_i} :\; i \in B\big)} &=
        \lim_{N\to\infty} \varhato{f\left(\mathds{1}_{I_j(N)}: \; j \in A\right) \cdot g\left(\mathds{1}_{I_i(N)} :\; i \in B\right)} \\ &\leq
        \lim_{N \to \infty} \varhato{f\left(\mathds{1}_{I_j(N)}: \; j \in A\right)} \cdot \varhato{g\left(\mathds{1}_{I_i(N)} :\; i \in B\right)}\\ &=
        \varhato{f\left(\mathds{1}_{I_j}: \; j \in A\right)} \cdot \varhato{g\big(\mathds{1}_{I_i} :\; i \in B\big)}.
        \end{split}
        \]
    \end{proof}
    We want to emphasize that this is clearly a geometric property of the web.
    The property does not depend on the underlying Bernoulli measures.
    From this theorem the main theorem of this section follows.
    \begin{thm} \label{thm:convergence_of_the_number_of_components}
        Suppose for $\varhato{C_n}$ we have that for some $C,c > 0$
        \[
        \varhato{C_n} \geq C \cdot n^c(1+o(1)).
        \]
        Then
        \[
        \lim_{n \to \infty}\frac{C_n}{\varhato{C_n}} = 1 \quad \text{almost surely}.
        \]
    \end{thm}
    \begin{proof}
        Recall by \cref{eq:number_of_components} and using the notation in \cref{eq:def_of_the_limit_of_the_indicators}
        \[
        C_n = 1 + \sum_{k=1}^{n} \mathds{1}_{I_k}.
        \]
        Moreover, by \Cref{crl:limit_indicators_are_negatively_associated} we know that the variables $\mathds{1}_{I_k}$ are negatively associated.
        Therefore, if we denote their sum by $S_n$, then $C_n = 1 + S_n$.
        Since for some $C,c > 0$ we have $\varhato{C_n} \geq C \cdot n^c(1+o(1))$, clearly
        $\varhato{S_n} \geq C \cdot n^c(1+o(1))$.
        Therefore, by \Cref{crl:almost_sure_convergence_for_NA}
        \[
        \lim_{n\to\infty}\frac{S_n}{\varhato{S_n}} = 1 \qquad \text{almost surely}.
        \]
        Since $\lim_{n \to \infty}\varhato{S_n} = \infty$, it follows that
        \[
        \lim_{n \to \infty}\frac{C_n}{\varhato{C_n}}
        = \lim_{n \to \infty} \frac{1+S_n}{\varhato{S_n}} \cdot \frac{\varhato{S_n}}{1+\varhato{S_n}}
        = 1 \qquad \text{almost surely}.
        \]
    \end{proof}
    We will give exact calculations for $\varhato{C_n}$ in \Cref{section_5} when discussing the Pólya Web.

    \newpage

    \section{The Pólya Web} \label{section_4}
    \thispagestyle{plain}
    From now on, we focus our attention only on the Pólya Web.
    As mentioned before, it is a special case of the previous web with the following Bernoulli measures.
    For $\lambda = (a,b) \in \Lambda$ using the shorthand notation
    \begin{equation} \label{eq:def_of_psi}
    \psi(\lambda) \defeq \frac{a}{a+b}
    \end{equation}
    for the ratio of the first coordinate to the sum of the coordinates we have
    \[
    \mathbb{P}_\lambda\left((1,0)\right) = 1 - \mathbb{P}_\lambda\left((0,1)\right) = \psi(\lambda).
    \]
    We will call the coalescing random walks $S_\lambda(n)$ coalescing Pólya Walks.
    As the title of this section suggests, we will focus on local properties of the Pólya Web.
    It means in this section we only focus on finite number of coalescing Pólya Walks.
    
    Let us introduce the following notations. We denote by $X_\lambda(n)$ and $Y_\lambda(n)$ the $x$ and $y$ coordinates of a Pólya Walk started from $\lambda \in \Lambda$ at time $n \geq \norm{\lambda}$. That is
    \[
    X_\lambda(n) \defeq \kek{S_\lambda(n)} \qquad \text{and} \qquad
    Y_\lambda(n) \defeq \piros{S_\lambda(n)}.
    \]
    Notice by the usage of universal time $X_\lambda(n) + Y_\lambda(n) = n$.
    Finally the ratio of the first coordinate to the sum of the coordinates of the Pólya Walk at time $n$ by $Z_\lambda(n)$. Therefore,
    \begin{equation} \label{eq:def_of_the_ration}
    Z_\lambda(n) \defeq \psi(S_\lambda(n)) = \frac{X_\lambda(n)}{X_\lambda(n)+Y_\lambda(n)} = \frac{X_\lambda(n)}{n}.
    \end{equation}
    An immediate consequence of the properties of the Pólya Urn discussed in the introduction
    if that for $\lambda = (a,b) \in \Lambda$
    \begin{equation} \label{eq:beta_convergence}
    \lim_{n \to \infty}Z_\lambda(n) \eqdef Z_\lambda \sim \text{Beta}(a,b)
    \qquad \text{exists almost surely}.
    \end{equation}
    Moreover, let us denote the probability density function of the distribution of $Z_\lambda$ by
    \[
    f_\lambda(x) \defeq \frac{\text{d} \,\mathbb{P}_*Z_\lambda}{\text{d}\,\mathbf{Leb}}(x)
    = \frac{1}{B(a,b)}x^{a-1}(1-x)^{b-1} \cdot \mathds{1}_{[0,1]}(x),
    \]
    where $B(a,b)$ is the normalizing factor, i.e.
    \[
    B(a,b) = \int_{0}^{1}x^{a-1}(1-x)^{b-1}\,dx = \frac{\Gamma(a)\Gamma(b)}{\Gamma(a+b)},
    \]
    and its cumulative distribution function by
    \[
    F_\lambda(x) \defeq \int_{-\infty}^{x}f_\lambda(u) \, du.
    \]

    \subsection{Limiting variables and the trajectories}
    In this subsection we are investigating the relationship of the trajectories of the coalescing Pólya Walks and the random variables $Z_\lambda$ obtained by the limit in \cref{eq:beta_convergence}. It will turn out that these random variables describe well the behaviour of the coalescing Pólya Walks. The key is that the function $\psi$ defined in \cref{eq:def_of_psi} turns out to be a good choice. It works well with the partial order of the web defined in \Cref{dfn:partial_order_on_Lambda}.
    \begin{prp}\label{prp:ratio_and_order}
        For any $(a_1,b_1) = \lambda_1 \succ \lambda_2 = (a_2,b_2)$
        \[
        \psi(\lambda_1) < \psi(\lambda_2).
        \]
    \end{prp}
    \begin{proof}
        By definition
            \[
            a_1 < a_2, \quad b_1 \geq b_2 \qquad \text{or} \qquad a_1 \leq a_2, \quad b_1 > b_2.
            \]
            Then
            \[
            \psi(\lambda_2) - \psi(\lambda_1) = \frac{a_2}{a_2+b_2} - \frac{a_1}{a_1+b_1}
            = \frac{a_2b_1-a_1b_2}{(a_1+b_1)(a_2+b_2)} > 0.
            \]
    \end{proof}
    Therefore, the trajectories inherit this monotone property.
    \begin{prp} \label{prp:witness_monoton}
        For any $\lambda_1 \succeq \lambda_2$ and for any $n \geq \max\left\{ \norm{\lambda_1}, \norm{\lambda_2}\right\}$
        \[
        Z_{\lambda_1}(n) \leq Z_{\lambda_2}(n).
        \]
        Moreover, equality holds if and only if $S_{\lambda_1}(n) = S_{\lambda_2}(n)$.
    \end{prp}
    That is, for Pólya Walks started from two different ordered points the ratio of the first coordinate will stay monotone for the whole trajectory, and they become equal if and only if the two walks stay in the same point of the plane.
    \begin{proof}
        By \Cref{prp:partial_order_and_trajectories} for any $n \geq \max\left\{ \norm{\lambda_1}, \norm{\lambda_2}\right\}$
        \[
        S_{\lambda_1}(n) \succeq S_{\lambda_2}(n)
        \]
        then by the first property of $\psi$ in \Cref{prp:ratio_and_order}
        \[
        Z_{\lambda_1}(n) = \psi(S_{\lambda_1}(n)) \leq
        \psi(S_{\lambda_2}(n)) = Z_{\lambda_2}(n) .
        \]
        Clearly equality holds if and only if $S_{\lambda_1}(n) = S_{\lambda_2}(n)$.
    \end{proof}
    However, we will show that this property will survive if we take the limit and consider the variables $Z_\lambda$.
    \begin{prp} \label{prp:limiting_monoton}
        For any $\lambda_1 \succeq \lambda_2$
        \[
        Z_{\lambda_1} \leq Z_{\lambda_2} \quad \text{almost surely}.
        \]
    \end{prp}
    \begin{proof}
        By the previous \Cref{prp:witness_monoton}, for any $n \geq \max\left\{ \norm{\lambda_1}, \norm{\lambda_2}\right\}$
        \[
        Z_{\lambda_1}(n) \leq Z_{\lambda_2}(n).
        \]
        Therefore,
        \[
        \limsup_{n \to \infty}Z_{\lambda_1}(n) \leq \limsup_{n \to \infty}Z_{\lambda_2}(n).
        \]
        However, by \cref{eq:beta_convergence} both $Z_{\lambda_1}(n)$ and $Z_{\lambda_2}(n)$ converge almost surely to $Z_{\lambda_1}$ and $Z_{\lambda_1}$ respectively.
        Therefore, the statement holds.
    \end{proof}
    Now we connect the stopping times defined in the beginning of \Cref{section_3} and the limiting random variables $Z_\lambda$.
    That is for a pair $\lambda_1 \succeq \lambda_2$ the events that $Z_{\lambda_1}$ and $Z_{\lambda_2}$ are equal is almost surely the same as the event that the two coalescing Pólya Walks started from $\lambda_1$ and $\lambda_2$ meet in finite time.
    Formalizing this results in the following lemma.
    \begin{lmm} \label{lmm:witness_functions_and_stopping_times}
        For any $\lambda_1 \succeq \lambda_2$
        \[
        \left\{\tau_{\lambda_1,\lambda_2} < \infty\right\} = \left\{Z_{\lambda_1} = Z_{\lambda_2}\right\} \qquad \text{almost surely}.
        \]
        Moreover, as a consequence
        \[
        \left\{\tau_{\lambda_1,\lambda_2} = \infty\right\} = \left\{Z_{\lambda_1} < Z_{\lambda_2}\right\} \qquad \text{almost surely}.
        \]
    \end{lmm}
    \begin{proof}
        Notice that if for some finite $N$ we have
        $
        S_{\lambda_1}(N) = S_{\lambda_2}(N)
        $,
        then for all $n \geq N$ we have
        $
        Z_{\lambda_1}(n) = Z_{\lambda_2}(n).
        $
        Consequently $Z_{\lambda_1} = Z_{\lambda_2}$. It means
        \[
        \left\{\tau_{\lambda_1,\lambda_2} < \infty \right\} \subseteq \left\{Z_{\lambda_1} = Z_{\lambda_2}\right\}.
        \]
        Therefore, it is enough to prove
        \[
        \valseg{\tau_{\lambda_1,\lambda_2} = \infty, \; Z_{\lambda_1} = Z_{\lambda_2}} = 0.
        \]
        Consider two \textit{independent} Pólya Walks on the web started from $\lambda_1$ and $\lambda_2$ and denote them by $\tilde{S}_{\lambda_1}$ and $\tilde{S}_{\lambda_2}$.
        Then consider the following stopping time
        \[
        \tilde{\tau}_{\lambda_1,\lambda_2} \defeq
        \inf\left\{k \geq \max\left\{\| \lambda_1 \|, \|\lambda_2\|\right\} : \; \tilde{S}_{\lambda_1}(k) = \tilde{S}_{\lambda_2}(k)\right\}.
        \]
        Now consider the following construction (\Cref{fig:construction_from_independent})
        \begin{equation} \label{eq:construction_from_independent}
        \hat{S}_{\lambda_1}(n) \defeq \tilde{S}_{\lambda_1}(n) \qquad \text{and} \qquad
        \hat{S}_{\lambda_2}(n) \defeq 
        \begin{cases}
        \tilde{S}_{\lambda_2}(n) & \hbox{if $m < \tilde{\tau}_{\lambda_1,\lambda_2}$} \\
        \tilde{S}_{\lambda_1}(n) & \hbox{if $m \geq \tilde{\tau}_{\lambda_1,\lambda_2}$}
        \end{cases}.
    \end{equation}
    Then for the following stopping time
        \[
        \hat{\tau}_{\lambda_1,\lambda_2} \defeq
        \inf\left\{k \geq \max\left\{\| \lambda_1 \|, \|\lambda_2\|\right\} : \; \hat{S}_{\lambda_1}(k) = \hat{S}_{\lambda_2}(k)\right\}
        \]
    we have $\tilde{\tau}_{\lambda_1,\lambda_2} = \hat{\tau}_{\lambda_1,\lambda_2}$.
    \begin{figure}[h!]
      \centering
      % First row
      \begin{subfigure}[b]{0.33\textwidth}
        \centering
        \includegraphics[width=\textwidth]{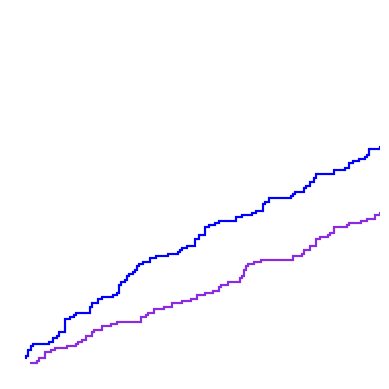}
        %\caption{Image 1}
      \end{subfigure}
      %\hfill
      \begin{subfigure}[b]{0.33\textwidth}
        \centering
        \includegraphics[width=\textwidth]{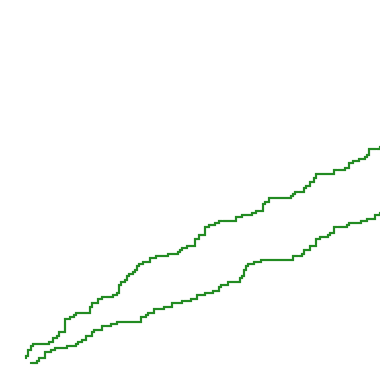}
        %\caption{Image 2}
      \end{subfigure}
    
      % Second row
      \vskip\baselineskip
      \begin{subfigure}[b]{0.33\textwidth}
        \centering
        \includegraphics[width=\textwidth]{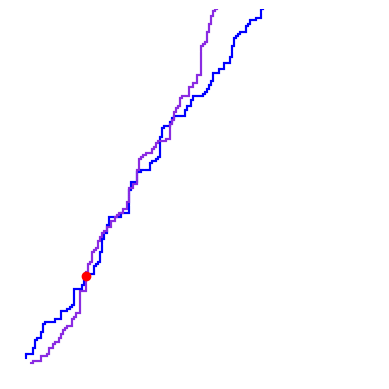}
        %\caption{Image 3}
      \end{subfigure}
      %\hfill
      \begin{subfigure}[b]{0.33\textwidth}
        \centering
        \includegraphics[width=\textwidth]{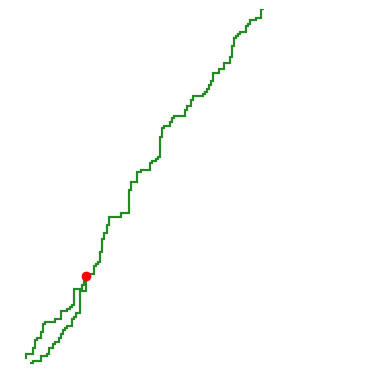}
        %\caption{Image 4}
      \end{subfigure}
    
      \caption{The construction in \cref{eq:construction_from_independent}.}
      \label{fig:construction_from_independent}
    \end{figure}
    
    Notice, the random walks $S_{\lambda_1}(n)$ and $S_{\lambda_2}(n)$ in the original web follow and independent trajectories until the point they meet. From that point they follow the same trajectory. Therefore, their joint distribution with the recently defined new variables are the same
    \[
    (S_{\lambda_1}(n), S_{\lambda_2}(n), \tau_{\lambda_1,\lambda_2}) \overset{d}{=}
    (\hat{S}_{\lambda_1}(n), \hat{S}_{\lambda_2}(n), \hat{\tau}_{\lambda_1,\lambda_2}),
    \]
    and in particular $\tau_{\lambda_1,\lambda_2} \overset{d}{=} \hat{\tau}_{\lambda_1,\lambda_2}$.

    Thus after applying $\psi$ to the original and the new random walks
    \[
    \valseg{\tau_{\lambda_1,\lambda_2} = \infty, \; Z_{\lambda_1} = Z_{\lambda_2}} =
    \valseg{\hat{\tau}_{\lambda_1,\lambda_2} = \infty, \; \hat{Z}_{\lambda_1} = \hat{Z}_{\lambda_2}}.
    \]
    Moreover, notice that by construction in \cref{eq:construction_from_independent} the event of the right hand side
    \[
    \left\{\hat{\tau}_{\lambda_1,\lambda_2} = \infty, \; \hat{Z}_{\lambda_1} = \hat{Z}_{\lambda_2}\right\}
    =
    \left\{\tilde{\tau}_{\lambda_1,\lambda_2} = \infty, \; \hat{Z}_{\lambda_1} = \hat{Z}_{\lambda_2}\right\}
    \subseteq
    \left\{\tilde{Z}_{\lambda_1} = \tilde{Z}_{\lambda_2}\right\}.
    \]
    However, since $\tilde{Z}_{\lambda_1}$ and $\tilde{Z}_{\lambda_2}$ are two independent random variables with absolutely continuous distributions, we have
    \[
    \valseg{\tilde{Z}_{\lambda_1} = \tilde{Z}_{\lambda_2}} = 0.
    \]
    \end{proof}
    The strength of \Cref{lmm:witness_functions_and_stopping_times} is that if we want to find the probabilities of the finiteness of the stopping times it is enough to find the joint distributions of the limiting random variables $Z_\lambda$.

    \subsection{The joint density of the limiting variables}
    In this subsection, we derive a formula for the joint density of $Z_{\lambda_1}, Z_{\lambda_2}, \dots, Z_{\lambda_n}$ for $\lambda_1 \succeq \lambda_2 \succeq \dots \succeq \lambda_n$.
    By \Cref{prp:limiting_monoton}, we know that $Z_{\lambda_1} \leq Z_{\lambda_2} \leq \dots \leq Z_{\lambda_n}$.
    The monotonicity implies that the joint density has different mappings on $2^{n-1}$ different linear subspaces of $\R^n$ since from the $n-1$ inequalities we need to choose which ones are the strict ones.
    The total number of possibilities is $2^{n-1}$.

    Moreover, notice that we do not need to calculate all the $2^{n-1}$ densities.
    Since if there are more than one consecutive equalities, then we can reduce it to a lower dimensional density in $\R^{n-1}$. It is true since
    \[
    \left\{Z_{\lambda_{i-1}} = Z_{\lambda_{i}} =Z_{\lambda_{i+1}}\right\}
    = \left\{Z_{\lambda_{i-1}} =Z_{\lambda_{i+1}}\right\}.
    \]
    As a consequence, we only need to focus on those $Z_{\lambda_i}$, which have a strict inequality in front of them or after them (\Cref{fig:density_and_trajectories}).
    
    Therefore, consider $l_1 \succeq r_1 \succ l_2 \succeq r_2 \succ \dots \succ l_n \succeq r_n$.
    Then its enough to determine the density for the case
    \[
    Z_{l_1} = Z_{r_1} < Z_{l_2} = Z_{r_2} < \dots < Z_{l_n} = Z_{r_n}.
    \]
    Notice that $l_i$ and $r_i$ can be the same.
    \begin{figure}[H]
      \centering
      \includegraphics[width=0.6\textwidth]{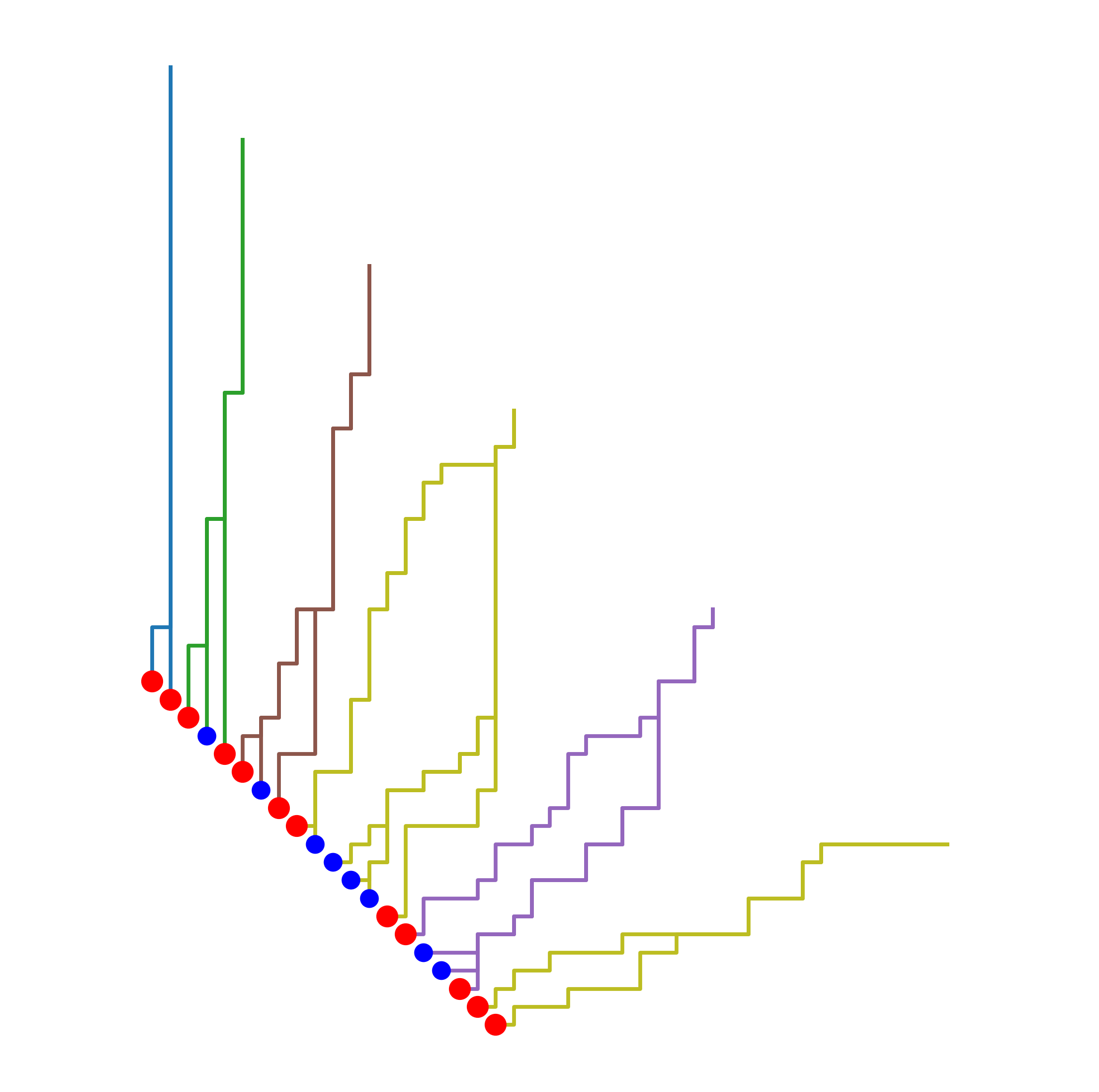}
      \caption{It is enough to consider the red points for such an event.}
      \label{fig:density_and_trajectories}
    \end{figure}
    
    In this case let us denote the density by
    \begin{equation} \label{eq:density_definition}
    \valseg{\bigcap_{k=1}^{n}\halm{Z_{l_k} = Z_{r_k} \in I_k}}
    =
    \int_{I_1\times I_2 \times \dots \times I_n}
    \begin{pmatrix}
        l_1 & r_1 & l_2 & r_2 &\dots & l_n & r_n\\
        x_1 & x_1 & x_2 & x_2 &\dots & x_n & x_n\\
    \end{pmatrix}\prod_{k=1}^n dx_k,
    \end{equation}
    where $0 < a_1 < b_1 < \dots < a_n < b_n < 1$ and $I_k = [a_k,b_k) \; (k=1,\dots,n)$.
    We have an explicit compact formula for the density in \cref{eq:density_definition}.
    It can be expressed as a determinant of a matrix with entries (depending on $x_1, \dots,x_n$) expressed with the probability density and cumulative distribution functions of $Z_{l_1}, Z_{r_1}, \dots,Z_{l_n},Z_{r_n}$.
    Recall that we denoted by $ f_{l_j}, f_{r_j}$ and $F_{l_j}, F_{r_j}$.
    We give this formula in the case when for some $k \in \N$
    \[
    \norm{l_j} = \norm{r_j} = k \qquad \text{for all $j=1,\dots n$}. 
    \]
    That is, when the coalescing Pólya Walks are started at the same universal time.
    It is not a significant restriction since for most of the applications this formula is useful enough. In this case the following theorem holds.
    \begin{thm} \label{thm:joint_density}
    For any $n \in \N^+$, $l_1 \succeq r_1 \succ l_2 \succeq r_2 \succ \dots \succ l_n \succeq r_n$ having $\norm{l_1} = \norm{r_1} = \dots = \norm{l_n} = \norm{r_n}$ and for any real numbers $0 < x_1 < x_2 < \dots x_n < 1$
    \[
    \begin{pmatrix}
        l_1 & r_1 & l_2 & r_2 &\dots & l_n & r_n\\
        x_1 & x_1 & x_2 & x_2 &\dots & x_n & x_n\\
    \end{pmatrix}
    \]
    \[
    =
    \begin{vNiceMatrix}[c,margin]
        1-F_{l_1}(x_1) & f_{l_1}(x_1) & 1-F_{l_1}(x_2) & f_{l_1}(x_2) & \dots &
        1-F_{l_1}(x_n) & f_{l_1}(x_n)\\
        -F_{r_1}(x_1) & f_{r_1}(x_1) & 1-F_{r_1}(x_2) & f_{r_1}(x_2) & \dots &
        1-F_{r_1}(x_n) & f_{r_1}(x_n)\\
        -F_{l_2}(x_1) & f_{l_2}(x_1) & 1-F_{l_2}(x_2) & f_{l_2}(x_2) & \dots &
        1-F_{l_2}(x_n) & f_{l_2}(x_n)\\
        -F_{r_2}(x_1) & f_{r_2}(x_1) & -F_{r_2}(x_2) & f_{r_2}(x_2) & \dots &
        1-F_{r_2}(x_n) & f_{r_2}(x_n)\\
        \vdots & \vdots & \vdots & \vdots & \ddots & \vdots &\vdots\\
        -F_{l_n}(x_1) & f_{l_n}(x_1) & -F_{l_n}(x_2) & f_{l_n}(x_2) & \dots &
        1-F_{l_n}(x_n) & f_{l_n}(x_n)\\
        -F_{r_n}(x_1) & f_{r_n}(x_1) & -F_{r_n}(x_2) & f_{r_n}(x_2) & \dots &
        -F_{r_n}(x_n) & f_{r_n}(x_n)\\
    \end{vNiceMatrix}_{2n \times 2n}.
    \]
    \end{thm}
    Before we prove the theorem we need to recall the Karlin-McGregor formula \cite{karlin_mcgregor:1959} and apply it in a special case.
    For that first recall the definition of birth and death processes.
    A discrete time process $(X_n)_{n=0}^{\infty}$ whose state space is the set of integers is a birth and death process, if
    \begin{align*}
        \valseg{X_{n+1} =j+1 | \, X_n =j} &= p_{n,j},\\
        \valseg{X_{n+1} =j-1 | \, X_n =j} &= q_{n,j},\\
        \valseg{X_{n+1} =j | \, X_n =j} &= 1-(p_{n,j}+q_{n,j}).
    \end{align*}
    Then the next theorem holds.
    \begin{thm}[Karlin-McGregor formula] \label{thm:karlin_mcgregor}
    Consider a birth and death Markov process whose state space is the set of integers.
    Let us denote its transition probability function with $P_{i,j}(t)$.
    Then for $i_1 < i_2 < \dots i_n$ and $j_1 < j_2 < \dots j_n$ point in the state space
    the determinant
    \[
    \begin{vNiceMatrix}[c, margin]
        P_{i_1,j_1}(t) & P_{i_1,j_2}(t) & \dots & P_{i_1,j_n}(t)\\
        P_{i_2,j_1}(t) & P_{i_2,j_2}(t) & \dots & P_{i_2,j_n}(t)\\
        \vdots & \vdots & \ddots & \vdots \\
        P_{i_n,j_1}(t) & P_{i_n,j_2}(t) & \dots & P_{i_n,j_n}(t)\\
    \end{vNiceMatrix}
    \]
    is the probability that $n$ independent processes started form $i_1, i_2, \dots, i_n$
    end up in states $j_1, j_2, \dots, j_n$ until time $t$ without any two of them being in the same state during the given time.
    \end{thm}

    Now we apply the Karlin-McGergor formula for the case of coalescing Pólya Walks started at the same universal time.
    \begin{prp} \label{prp:karlin_mcgregor_for_coalescing_polya}
        For any $\lambda_1 \succ \dots \succ \lambda_n$ such that $\norm{\lambda_1} = \dots = \norm{\lambda_n}$ and for any integers $0 \leq k_1 < \dots < k_n$ we have for any $N \geq \norm{\lambda_1}$
        \[
        \valseg{\bigcap_{j=1}^{n}\left\{X_{\lambda_j}(N) = k_j\right\}}=
            \begin{vNiceMatrix}[c,margin]
                \valseg{X_{\lambda_1}(N)=k_1} &\dots & \valseg{X_{\lambda_1}(N) = k_n}\\
                \vdots  &\ddots & \vdots\\
                \valseg{X_{\lambda_n}(N)=k_1} & \dots & \valseg{X_{\lambda_n}(N) = k_n}\\
            \end{vNiceMatrix}
        \]
    \end{prp}
    \begin{proof}
    Clearly we can embed the Pólya Walks into $\Z$ as a birth and death process.
    Moreover, since the walks are started at the same universal time and they run until the same universal time, the formula in \Cref{thm:karlin_mcgregor} is applicable.
    \end{proof}

    With the help of \Cref{prp:karlin_mcgregor_for_coalescing_polya} we can find the density for a special case, when none of the coalescing Pólya Walks meet in finite time.
    Then the density is stated in the following lemma.
    \begin{lmm} \label{lmm:density_of_different_points}
    For any $\lambda_1 \succ \lambda_2 \succ \dots \succ \lambda_n$ having $\norm{\lambda_1} = \norm{\lambda_2} = \dots = \norm{\lambda_n}$ the joint density at $0 < x_1 < x_2 < \dots < x_n < 1$ is
    \[
    \begin{pmatrix}
        \lambda_1 & \lambda_1 & \lambda_2 & \lambda_2 &\dots & \lambda_n & \lambda_n\\
        x_1 & x_1 & x_2 & x_2 &\dots & x_n & x_n\\
    \end{pmatrix}
    =
    \begin{vNiceMatrix}[c, margin]
        f_{\lambda_1}(x_1) & f_{\lambda_1}(x_2) & \dots & f_{\lambda_1}(x_n)\\
        f_{\lambda_2}(x_1) & f_{\lambda_2}(x_2) & \dots & f_{\lambda_2}(x_n)\\
        \vdots & \vdots & \ddots & \vdots \\
        f_{\lambda_n}(x_1) & f_{\lambda_n}(x_2) & \dots & f_{\lambda_n}(x_n)\\
    \end{vNiceMatrix}.
    \]
    \end{lmm}
    \begin{proof}
        Let $0 < a_1 < b_1 < a_2 < b_2 < \dots <a_n < b_n < 1$. Then let us denote
        \[
        I_j \defeq \left[a_j,b_j \right) \qquad \text{for $j=1,2,\dots,n$}.
        \]
        Then by almost sure convergence and the definition of $Z_{\lambda_j}$
        \begin{equation} \label{eq:def_of_density_for_n}
        \valseg{\bigcap_{j=1}^{n}\left\{Z_{\lambda_j} \in I_j\right\}}=
        \lim_{N \to \infty} \valseg{\bigcap_{j=1}^{n}\left\{Z_{\lambda_j}(N) \in I_j\right\}} =
        \lim_{N \to \infty} \valseg{\bigcap_{j=1}^{n}\left\{\frac{X_{\lambda_j}(N)}{N} \in I_j\right\}}.
        \end{equation}
        Notice that by the definition of $I_j$ for $j=1,2,\dots,n$
        \begin{equation} \label{eq:intervals_for_the_density}
        \valseg{\bigcap_{j=1}^{n}\left\{X_{\lambda_j}(N) \in N \cdot I_j\right\}} =
        \valseg{\bigcap_{j=1}^{n}\left\{X_{\lambda_j}(N) \in [N a_j, \; N b_j)\right\}}.
        \end{equation}
        Then by just partitioning the event and using that $X_{\lambda_j}(N)$ is discrete by \cref{eq:def_of_density_for_n} and \cref{eq:intervals_for_the_density}
        \begin{equation} \label{eq:before_karlin_mcgregor}
        \valseg{\bigcap_{j=1}^{n}\left\{Z_{\lambda_j} \in I_j\right\}} =
        \lim_{N \to \infty} \sum_{j=1}^{n} \;\sum_{k_j=\floor*{N a_j}}^{\floor*{Nb_j}}
        \valseg{X_{\lambda_1}(N) = k_1, \dots, X_{\lambda_n}(N) = k_n}.
        \end{equation}
        However, after large enough $N \geq 1$ since $a_1 < b_1 < \dots < a_n < b_n$ the integers $k_j$ are disjoint for any $j=1,2,\dots,n$.
        Therefore, we can apply the Karlin-McGregor formula in our special case in \Cref{prp:karlin_mcgregor_for_coalescing_polya}.
        It follows that for large enough $N \geq 1$
        \begin{equation} \label{eq:after_karlin_mcgregor}
            \valseg{X_{\lambda_1}(N) = k_1, \dots, X_{\lambda_n}(N) = k_n}=
            \begin{vNiceMatrix}[c,margin]
                \valseg{X_{\lambda_1}(N)=k_1} &\dots & \valseg{X_{\lambda_1}(N) = k_n}\\
                \vdots  &\ddots & \vdots\\
                \valseg{X_{\lambda_n}(N)=k_1} & \dots & \valseg{X_{\lambda_n}(N) = k_n}\\
            \end{vNiceMatrix}
        \end{equation}
        Moreover, using the elementary property of the determinant (namely the fact that if two matrices only differ in one column the sum of their determinants equals to the determinant of their sum)
        \begin{equation} \label{eq:elementary_determinant}
        \setlength{\jot}{12pt}
        \begin{split}
            &\sum_{j=1}^{n} \;\sum_{k_j=\floor*{N a_j}}^{\floor*{Nb_j}}
            \begin{vNiceMatrix}[c,margin]
                \valseg{X_{\lambda_1}(N)=k_1} &\dots & \valseg{X_{\lambda_1}(N) = k_n}\\
                \vdots &\ddots & \vdots\\
                \valseg{X_{\lambda_n}(N)=k_1} &\dots & \valseg{X_{\lambda_n}(N) = k_n}\\
            \end{vNiceMatrix}\\&=
            \begin{vNiceMatrix}[c,margin]
                \sum_{k_1=\floor*{Na_1}}^{\floor*{Nb_1}}\valseg{X_{\lambda_1}(N)=k_1}&\dots & \sum_{k_n=\floor*{Na_n}}^{\floor*{Nb_n}}\valseg{X_{\lambda_1}(N) = k_n}\\
                \vdots &\ddots & \vdots\\
                \sum_{k_1=\floor*{Na_1}}^{\floor*{Nb_1}}\valseg{X_{\lambda_n}(N)=k_1} &\dots & \sum_{k_n=\floor*{Na_n}}^{\floor*{Nb_n}}\valseg{X_{\lambda_n}(N) = k_n}\\
            \end{vNiceMatrix}\\
            &=
            \begin{vNiceMatrix}[c,margin]
                \valseg{X_{\lambda_1}(N) \in N \cdot I_1} & \dots & \valseg{X_{\lambda_1}(N) \in N \cdot I_n}\\
                \vdots & \ddots & \vdots\\
                \valseg{X_{\lambda_n}(N) \in N \cdot I_1} & \dots & \valseg{X_{\lambda_n}(N) \in N \cdot I_n}\\
            \end{vNiceMatrix}\\
            &=
            \begin{vNiceMatrix}[c,margin]
                \valseg{Z_{\lambda_1}(N) \in I_1} & \dots & \valseg{Z_{\lambda_1}(N) \in I_n}\\
                \vdots & \ddots & \vdots\\
                \valseg{Z_{\lambda_n}(N) \in I_1} & \dots & \valseg{Z_{\lambda_n}(N) \in I_n}\\
            \end{vNiceMatrix}
        \end{split}
        \end{equation}
        Therefore, after comparing \cref{eq:before_karlin_mcgregor}, \cref{eq:after_karlin_mcgregor} and \cref{eq:elementary_determinant} it follows from the almost sure convergence and the fact that we only integrate every column with respect to different variables
        \begin{equation*}
        \begin{split}
        \valseg{\bigcap_{j=1}^{n}\left\{Z_j \in I_j\right\}} &= \lim_{N \to \infty}
        \begin{vNiceMatrix}[c,margin]
                \valseg{Z_{\lambda_1}(N) \in I_1} & \dots & \valseg{Z_{\lambda_1}(N) \in I_n}\\
                \vdots & \ddots & \vdots\\
                \valseg{Z_{\lambda_n}(N) \in I_1} & \dots & \valseg{Z_{\lambda_n}(N) \in I_n}\\
        \end{vNiceMatrix}\\
        &=
        \begin{vNiceMatrix}[c,margin]
            \int_{I_1}f_{\lambda_1}(x_1) \, dx_1 & \dots & \int_{I_n}f_{\lambda_1}(x_n) \, dx_n\\
            \vdots & \ddots & \vdots\\
            \int_{I_1}f_{\lambda_n}(x_1) \, dx_1 & \dots & \int_{I_n}f_{\lambda_n}(x_n) \, dx_n\\
        \end{vNiceMatrix}\\
        &=
        \int_{I_1 \times \dots \times I_n}
        \begin{vNiceMatrix}
            f_{\lambda_1}(x_1) & \dots & f_{\lambda_1}(x_n)\\
            \vdots & \ddots & \vdots\\
            f_{\lambda_n}(x_1) & \dots & f_{\lambda_n}(x_n)\\
        \end{vNiceMatrix}
        \prod_{j=1}^{n}dx_j.
        \end{split}
        \end{equation*}
        Therefore, the density is indeed the one stated in the lemma.
    \end{proof}

    Now with all of the needed tools in our hand we can prove the main theorem of this section.
        
    \begin{proof}[Proof of \Cref{thm:joint_density}]
     We show by induction on $k$ that for any $1 \leq k$ and any $n \geq k$ the joint density for the points $l_1 \succeq r_1 \succ \dots \succ l_k \succeq r_k \succ l_{k+1} = r_{k+1} \succ \dots \succ l_n = r_n$ is
        \[
        \begin{pmatrix}
        l_1 & r_1 & l_2 & r_2 &\dots & l_n & r_n\\
        x_1 & x_1 & x_2 & x_2 &\dots & x_n & x_n\\
        \end{pmatrix}
        \]
        \[
        =
        \begin{vNiceMatrix}[c,margin]
            \CodeBefore
            \rectanglecolor{blue!20}{1-1}{5-5}
            \rectanglecolor{orange!20}{6-1}{8-5}
            \rectanglecolor{red!20}{1-6}{8-8}
            \Body
            1-F_{l_1}(x_1) & f_{l_1}(x_1) & \dots  & 1-F_{l_1}(x_k) & f_{l_1}(x_k) & f_{l_1}(x_{k+1})&\dots &
            f_{l_1}(x_n)\\
            -F_{r_1}(x_1) & f_{r_1}(x_1) & \dots & 1-F_{r_1}(x_k) & f_{r_1}(x_k) & f_{r_1}(x_{k+1})& \dots &
            f_{r_1}(x_n)\\
            \vdots & \vdots & \ddots & \vdots & \vdots & \vdots &\ddots & \vdots\\
            -F_{l_k}(x_1) & f_{l_k}(x_1) & \dots & 1-F_{l_k}(x_k) & f_{l_k}(x_k) & f_{l_k}(x_{k+1})&\dots &
            f_{l_k}(x_n)\\
            -F_{r_k}(x_1) & f_{r_k}(x_1) & \dots & -F_{r_k}(x_k) & f_{r_k}(x_k) & f_{r_k}(x_{k+1})&\dots &
            f_{r_k}(x_n)\\
            -F_{l_{k+1}}(x_1) & f_{l_{k+1}}(x_1) & \dots & -F_{l_{k+1}}(x_k) & f_{l_{k+1}}(x_k) & f_{l_{k+1}}(x_{k+1})& \dots & f_{l_{k+1}}(x_n)\\
            \vdots & \vdots & \ddots & \vdots & \vdots & \vdots &\ddots &\vdots\\
            -F_{l_n}(x_1) & f_{l_n}(x_1) & \dots & -F_{l_n}(x_2) & f_{l_n}(x_2) & f_{l_n}(x_{k+1})&\dots &
            f_{l_n}(x_n)\\
        \end{vNiceMatrix}
        \]
        Notice the \textcolor{blue}{upper left $2k\times 2k$ sub-matrix} is the same as in the one in the stated theorem.
        The \textcolor{orange}{lower left $\zarojel{n-k}\times 2k$ sub-matrix} is similar, but it contains only every second row.
        Moreover, the \textcolor{red}{right $\zarojel{n+k}\times\zarojel{n-k}$ sub-matrix} is similar, but it does not contain the cumulative distribution functions just the densities.
        Therefore, applying it for $k=n$ gives the formula stated in the theorem.
        
        First for $k=1$. Suppose $l_1 = r_1$. Then adding \textcolor{blue}{one row} to \textcolor{orange}{another} we get by \Cref{lmm:density_of_different_points}
        \begin{small}
        \[
        \begin{vNiceMatrix}[c,margin]
            \CodeBefore
            \rowcolor{orange!20}{1}
            \rowcolor{blue!20}{2}
            \Body
            1-F_{l_1}(x_1) & f_{l_1}(x_1) & f_{l_1}(x_2) & \dots & f_{l_1}(x_n)\\
            \rowcolor{blue!20}
            -F_{l_1}(x_1) & f_{l_1}(x_1) & f_{l_1}(x_2) & \dots & f_{l_1}(x_n)\\
            -F_{l_2}(x_1) & f_{l_2}(x_1) & f_{l_2}(x_2) & \dots & f_{l_2}(x_n)\\
            \vdots & \vdots & \vdots & \ddots & \vdots\\
            -F_{l_n}(x_1) & f_{l_n}(x_1) & f_{l_n}(x_2) & \dots & f_{l_n}(x_n)\\
        \end{vNiceMatrix}
        =
        \begin{vNiceMatrix}[c,margin]
            \CodeBefore
            \rowcolor{yellow!20}{1}
            \columncolor{yellow!20}{1}
            \Body
            1 & 0 & 0 & \dots & 0\\
            -F_{l_1}(x_1) & f_{l_1}(x_1) & f_{l_1}(x_2) & \dots & f_{l_1}(x_n)\\
            -F_{l_2}(x_1) & f_{l_2}(x_1) & f_{l_2}(x_2) & \dots & f_{l_2}(x_n)\\
            \vdots & \vdots & \vdots & \ddots & \vdots\\
            -F_{l_n}(x_1) & f_{l_n}(x_1) & f_{l_n}(x_2) & \dots & f_{l_n}(x_n)\\
        \end{vNiceMatrix}
        \]
        \[
        =
        \begin{vmatrix}
            f_{l_1}(x_1) & f_{l_1}(x_2) & \dots & f_{l_1}(x_n)\\
            f_{l_2}(x_1) & f_{l_2}(x_2) & \dots & f_{l_2}(x_n)\\
            \vdots & \vdots & \ddots & \vdots\\
            f_{l_n}(x_1) & f_{l_n}(x_2) & \dots & f_{l_n}(x_n)\\
        \end{vmatrix}
        =
        \begin{pmatrix}
        l_1 & l_1 & l_2 & l_2 &\dots & l_n & l_n\\
        x_1 & x_1 & x_2 & x_2 &\dots & x_n & x_n\\
        \end{pmatrix}.
        \]
        \end{small}
        Now suppose $l_1 \succ r_1$. Then using \Cref{lmm:density_of_different_points} for $n$ and $n+1$
        \[
        \begin{pmatrix}
        l_1 & r_1 & l_2 & l_2 &\dots & l_n & l_n\\
        x_1 & x_1 & x_2 & x_2 &\dots & x_n & x_n\\
        \end{pmatrix}
        \]
        \[
        =
        \begin{pmatrix}
        r_1 & r_1 & l_2 & l_2 &\dots & l_n & l_n\\
        x_1 & x_1 & x_2 & x_2 &\dots & x_n & x_n\\
        \end{pmatrix}
        -
        \int_{-\infty}^{x_1}
        \begin{pmatrix}
        l_1 &  l_1 & r_1 & r_1 &  l_2 & l_2 &\dots & l_n & l_n\\
        u &  u & x_1 & x_1 & x_2 & x_2 &\dots & x_n & x_n\\
        \end{pmatrix} \, du.
        \]
        The right hand side equals to
        \[
         \begin{vmatrix}
            1 & 0 & 0 & \dots & 0\\
            0 & f_{r_1}(x_1) & f_{r_1}(x_2) & \dots & f_{r_1}(x_n)\\
            0 & f_{l_2}(x_1) & f_{l_2}(x_2) & \dots & f_{l_2}(x_n)\\
            \vdots & \vdots & \vdots & \ddots & \vdots\\
            0 & f_{l_n}(x_1) & f_{l_n}(x_2) & \dots & f_{l_n}(x_n)\\
        \end{vmatrix}
        -
        \begin{vmatrix}
            F_{l_1}(x_1) & f_{l_1}(x_1) & f_{l_1}(x_2) & \dots & f_{l_1}(x_n)\\
            F_{r_1}(x_1) & f_{r_1}(x_1) & f_{r_1}(x_2) & \dots & f_{r_1}(x_n)\\
            F_{l_2}(x_1) & f_{l_2}(x_1) & f_{l_2}(x_2) & \dots & f_{l_2}(x_n)\\
            \vdots & \vdots & \vdots & \ddots & \vdots\\
            F_{l_n}(x_1) & f_{l_n}(x_1) & f_{l_n}(x_2) & \dots & f_{l_n}(x_n)\\
        \end{vmatrix}
        \]
        \[
        =
        \begin{vmatrix}
            1-F_{l_1}(x_1) & f_{l_1}(x_1) & f_{l_1}(x_2) & \dots & f_{l_1}(x_n)\\
            -F_{r_1}(x_1) & f_{r_1}(x_1) & f_{r_1}(x_2) & \dots & f_{r_1}(x_n)\\
            -F_{l_2}(x_1) & f_{l_2}(x_1) & f_{l_2}(x_2) & \dots & f_{l_2}(x_n)\\
            \vdots & \vdots & \vdots & \ddots & \vdots\\
            -F_{l_n}(x_1) & f_{l_n}(x_1) & f_{l_n}(x_2) & \dots & f_{l_n}(x_n)\\
        \end{vmatrix}
        \]
        Therefore, the formula is true for $k=1$. Now suppose it holds for $k$. Suppose
        $l_{k+1}=r_{k+1}$. Then by manipulating the determinant again by adding \textcolor{blue}{one row} to \textcolor{orange}{another} we get
        \begin{footnotesize}
        \[
        \begin{vNiceMatrix}[r,margin]
            \CodeBefore
            \rowcolor{orange!20}{6}
            \rowcolor{blue!20}{7}
            \Body
            1-F_{l_1}(x_1) & \dots  & 1-F_{l_1}(x_k) & f_{l_1}(x_k) & 1-F_{l_1}(x_{k+1}) & f_{l_1}(x_{k+1}) &\dots &
            f_{l_1}(x_n)\\
            -F_{r_1}(x_1) & \dots & 1-F_{r_1}(x_k) & f_{r_1}(x_k)  & 1-F_{r_1}(x_{k+1}) & f_{r_1}(x_{k+1}) & \dots &
            f_{r_1}(x_n)\\
            \vdots  & \ddots & \vdots & \vdots & \vdots &\vdots & \ddots & \vdots\\
            -F_{l_k}(x_1) & \dots & 1-F_{l_k}(x_k) & f_{l_k}(x_k) & 1-F_{l_k}(x_{k+1}) & f_{l_k}(x_{k+1}) &\dots &
            f_{l_k}(x_n)\\
            -F_{r_k}(x_1) & \dots & -F_{r_k}(x_k) & f_{r_k}(x_k) & 1-F_{r_k}(x_{k+1}) & f_{r_k}(x_{k+1}) &\dots &
            f_{r_k}(x_n)\\
            -F_{l_{k+1}}(x_1) & \dots & -F_{l_{k+1}}(x_k) & f_{l_{k+1}}(x_k) &  1-F_{l_{k+1}}(x_{k+1}) & f_{l_{k+1}}(x_{k+1}) & \dots & f_{l_{k+1}}(x_n)\\
            -F_{l_{k+1}}(x_1)  & \dots & -F_{l_{k+1}}(x_k) & f_{l_{k+1}}(x_k) &  -F_{l_{k+1}}(x_{k+1}) & f_{l_{k+1}}(x_{k+1}) & \dots & f_{l_{k+1}}(x_n)\\
            \vdots & \ddots & \vdots & \vdots & \vdots &\vdots &\ddots & \vdots\\
            -F_{l_n}(x_1) & \dots & -F_{l_n}(x_k) & f_{l_n}(x_k) & -F_{l_n}(x_{k+1}) & f_{l_n}(x_{k+1})&\dots &
            f_{l_n}(x_n)\\
        \end{vNiceMatrix}
       \]
       \[
        =
        \begin{vNiceMatrix}[c,margin]
            \CodeBefore
            \rowcolor{yellow!20}{6}
            \columncolor{yellow!20}{5}
            \Body
            1-F_{l_1}(x_1) & \dots  & 1-F_{l_1}(x_k) & f_{l_1}(x_k) & 1-F_{l_1}(x_{k+1}) & f_{l_1}(x_{k+1}) &\dots &
            f_{l_1}(x_n)\\
            -F_{r_1}(x_1) & \dots & 1-F_{r_1}(x_k) & f_{r_1}(x_k)  & 1-F_{r_1}(x_{k+1}) & f_{r_1}(x_{k+1}) & \dots &
            f_{r_1}(x_n)\\
            \vdots & \ddots & \vdots & \vdots & \vdots &\vdots & \ddots & \vdots\\
            -F_{l_k}(x_1)  & \dots & 1-F_{l_k}(x_k) & f_{l_k}(x_k) & 1-F_{l_k}(x_{k+1}) & f_{l_k}(x_{k+1}) &\dots &
            f_{l_k}(x_n)\\
            -F_{r_k}(x_1)  & \dots & -F_{r_k}(x_k) & f_{r_k}(x_k) & 1-F_{r_k}(x_{k+1}) & f_{r_k}(x_{k+1}) & \dots &
            f_{r_k}(x_n)\\
            0 & \dots & 0 &0 &1 &0 &\dots & 0\\
            -F_{l_{k+1}}(x_1)  & \dots & -F_{l_{k+1}}(x_k) & f_{l_{k+1}}(x_k) &  -F_{l_{k+1}}(x_{k+1}) & f_{l_{k+1}}(x_{k+1}) &  \dots & f_{l_{k+1}}(x_n)\\
            \vdots & \ddots & \vdots & \vdots & \vdots &\vdots &\ddots & \vdots\\
            -F_{l_n}(x_1)  & \dots & -F_{l_n}(x_k) & f_{l_n}(x_k) & -F_{l_n}(x_{k+1}) & f_{l_n}(x_{k+1}) &\dots &
            f_{l_n}(x_n)\\
        \end{vNiceMatrix}
        \]
        \[
        =
        \begin{vmatrix}
            1-F_{l_1}(x_1)  & \dots  & 1-F_{l_1}(x_k) & f_{l_1}(x_k) & f_{l_1}(x_{k+1})&\dots &
            f_{l_1}(x_n)\\
            -F_{r_1}(x_1) & \dots & 1-F_{r_1}(x_k) & f_{r_1}(x_k) & f_{r_1}(x_{k+1})& \dots &
            f_{r_1}(x_n)\\
            \vdots & \ddots & \vdots & \vdots & \vdots &\ddots & \vdots\\
            -F_{l_k}(x_1) & \dots & 1-F_{l_k}(x_k) & f_{l_k}(x_k) & f_{l_k}(x_{k+1})&\dots &
            f_{l_k}(x_n)\\
            -F_{r_k}(x_1)  & \dots & -F_{r_k}(x_k) & f_{r_k}(x_k) & f_{r_k}(x_{k+1})&\dots &
            f_{r_k}(x_n)\\
            -F_{l_{k+1}}(x_1)  & \dots & -F_{l_{k+1}}(x_k) & f_{l_{k+1}}(x_k) & f_{l_{k+1}}(x_{k+1})& \dots & f_{l_{k+1}}(x_n)\\
            \vdots  & \ddots & \vdots & \vdots & \vdots &\ddots &\vdots\\
            -F_{l_n}(x_1) & \dots & -F_{l_n}(x_2) & f_{l_n}(x_2) & f_{l_n}(x_{k+1})&\dots &
            f_{l_n}(x_n)\\
        \end{vmatrix}
        \]
        \end{footnotesize}
        which is exactly the formula for $k$.
        Now suppose $l_{k+1} \succ r_{k+1}$. First notice
        \begin{footnotesize}
        \[
        \begin{pmatrix}
            l_1 & r_1 & \dots & l_k & r_k & l_{k+1} & r_{k+1} & \dots & l_n & l_n\\
            x_1 & x_1 & \dots & x_k & x_k & x_{k+1} & x_{k+1} & \dots & x_n & x_n\\
        \end{pmatrix}
        \]
        \[
        =
        \begin{pmatrix}
            l_1 & r_1 & \dots & l_k & r_k & r_{k+1} & r_{k+1} & \dots & l_n & l_n\\
            x_1 & x_1 & \dots & x_k & x_k & x_{k+1} & x_{k+1} & \dots & x_n & x_n\\
        \end{pmatrix}
        \]
        \[
        - \int_{x_k}^{x_{k+1}}
        \begin{pmatrix}
            l_1 & r_1 & \dots & l_k & r_k & l_{k+1} & l_{k+1} &r_{k+1} & r_{k+1} & \dots & l_n & l_n\\
            x_1 & x_1 & \dots & x_k & x_k & u & u & x_{k+1} & x_{k+1} & \dots & x_n & x_n\\
        \end{pmatrix} \, du
        \]
        \[
        -
         \begin{pmatrix}
            l_1 & r_1 & \dots & l_k  & l_{k+1} &r_{k+1} & r_{k+1} & \dots & l_n & l_n\\
            x_1 & x_1 & \dots & x_k & x_k & x_{k+1} & x_{k+1} & \dots & x_n & x_n\\
        \end{pmatrix}.
        \]
        \end{footnotesize}
        Then using the formula for $k$ in the case of $n$ and $n+1$ the RHS equals to
        \begin{footnotesize}
        \[
        \begin{vmatrix}
            1-F_{l_1}(x_1) & \dots  & 1-F_{l_1}(x_k) & f_{l_1}(x_k) & 0 & f_{l_1}(x_{k+1}) &\dots &
            f_{l_1}(x_n)\\
            -F_{r_1}(x_1)  & \dots & 1-F_{r_1}(x_k) & f_{r_1}(x_k)  &  0 & f_{r_1}(x_{k+1}) & \dots &
            f_{r_1}(x_n)\\
            \vdots & \ddots & \vdots & \vdots & \vdots &\vdots & \ddots & \vdots\\
            -F_{l_k}(x_1) & \dots & 1-F_{l_k}(x_k) & f_{l_k}(x_k) & 0 & f_{l_k}(x_{k+1}) & \dots &
            f_{l_k}(x_n)\\
            -F_{r_k}(x_1) & \dots & -F_{r_k}(x_k) & f_{r_k}(x_k) & 0 & f_{r_k}(x_{k+1}) & \dots &
            f_{r_k}(x_n)\\
            0  & \dots & 0 & 0 &  1 & 0 & \dots & 0\\
            -F_{r_{k+1}}(x_1)  & \dots & -F_{r_{k+1}}(x_k) & f_{r_{k+1}}(x_k) &  0 & f_{r_{k+1}}(x_{k+1}) & \dots & f_{r_{k+1}}(x_n)\\
            \vdots  & \ddots & \vdots & \vdots & \vdots &\vdots & \ddots & \vdots\\
            -F_{l_n}(x_1) & \dots & -F_{l_n}(x_k) & f_{l_n}(x_k) & 0 & f_{l_n}(x_{k+1}) & \dots &
            f_{l_n}(x_n)\\
        \end{vmatrix}
       \]
       \[
       -
        \begin{vNiceMatrix}[c,margin]
            \CodeBefore
            \columncolor{orange!20}{3}
            \columncolor{blue!20}{5}
            \Body
            1-F_{l_1}(x_1) & \dots  & 1-F_{l_1}(x_k) & f_{l_1}(x_k) & F_{l_1}(x_{k+1})-F_{l_1}(x_k) & f_{l_1}(x_{k+1}) & \dots &
            f_{l_1}(x_n)\\
            -F_{r_1}(x_1)  & \dots & 1-F_{r_1}(x_k) & f_{r_1}(x_k)  &  F_{r_1}(x_{k+1})-F_{r_1}(x_k) & f_{r_1}(x_{k+1}) &  \dots &
            f_{r_1}(x_n)\\
            \vdots  & \ddots & \vdots & \vdots & \vdots &\vdots & \ddots & \vdots\\
            -F_{l_k}(x_1)  & \dots & 1-F_{l_k}(x_k) & f_{l_k}(x_k) & F_{l_k}(x_{k+1})-F_{l_k}(x_k) & f_{l_k}(x_{k+1}) & \dots &
            f_{l_k}(x_n)\\
            -F_{r_k}(x_1)  & \dots & -F_{r_k}(x_k) & f_{r_k}(x_k) & F_{r_k}(x_{k+1})-F_{r_k}(x_k) & f_{r_k}(x_{k+1}) &\dots &
            f_{r_k}(x_n)\\
            -F_{l_{k+1}}(x_1) & \dots & -F_{l_{k+1}}(x_k) & f_{l_{k+1}}(x_k) &  F_{l_{k+1}}(x_{k+1})-F_{l_{k+1}}(x_k) & f_{l_{k+1}}(x_{k+1}) & \dots & f_{l_{k+1}}(x_n)\\
            -F_{r_{k+1}}(x_1)  & \dots & -F_{r_{k+1}}(x_k) & f_{r_{k+1}}(x_k) &  F_{r_{k+1}}(x_{k+1})-F_{r_{k+1}}(x_k)  & f_{r_{k+1}}(x_{k+1}) & \dots & f_{r_{k+1}}(x_n)\\
            \vdots &  \ddots & \vdots & \vdots & \vdots &\vdots &\ddots & \vdots\\
            -F_{l_n}(x_1)  & \dots & -F_{l_n}(x_k) & f_{l_n}(x_k) & F_{l_n}(x_{k+1})-F_{l_n}(x_k)  & f_{l_n}(x_{k+1}) &\dots &
            f_{l_n}(x_n)\\
        \end{vNiceMatrix}
       \]
       \[
       -
        \begin{vNiceMatrix}[c,margin]
        \CodeBefore
            \rowcolor{red!20}{5}
            \rowcolor{red!20}{6}
            \Body
            1-F_{l_1}(x_1) & \dots  & 1-F_{l_1}(x_k) & f_{l_1}(x_k) & 0 & f_{l_1}(x_{k+1}) &\dots &
            f_{l_1}(x_n)\\
            -F_{r_1}(x_1) & \dots & 1-F_{r_1}(x_k) & f_{r_1}(x_k)  &  0 & f_{r_1}(x_{k+1}) & \dots &
            f_{r_1}(x_n)\\
            \vdots  & \ddots & \vdots & \vdots & \vdots &\vdots & \ddots & \vdots\\
            -F_{l_k}(x_1)  & \dots & 1-F_{l_k}(x_k) & f_{l_k}(x_k) & 0 & f_{l_k}(x_{k+1}) & \dots &
            f_{l_k}(x_n)\\
            -F_{l_{k+1}}(x_1)  & \dots & -F_{l_{k+1}}(x_k) & f_{l_{k+1}}(x_k) & 0 & f_{l_{k+1}}(x_{k+1}) & \dots &
            f_{l_{k+1}}(x_n)\\
            0  & \dots & 0 & 0 &  1 & 0  & \dots & 0\\
            -F_{r_{k+1}}(x_1)  & \dots & -F_{r_{k+1}}(x_k) & f_{r_{k+1}}(x_k) &  0 & f_{r_{k+1}}(x_{k+1}) &  \dots & f_{r_{k+1}}(x_n)\\
            \vdots  & \ddots & \vdots & \vdots & \vdots &\vdots &\ddots & \vdots\\
            -F_{l_n}(x_1) & \dots & -F_{l_n}(x_k) & f_{l_n}(x_k) & 0 & f_{l_n}(x_{k+1}) &\dots &
            f_{l_n}(x_n)\\
        \end{vNiceMatrix}.
       \]
       \end{footnotesize}
       First, we multiply \textcolor{blue}{one column} of the second one with minus one and add \textcolor{orange}{one column} to that one.
       Then we change \textcolor{red}{two rows} of the third one.
       After the expression equals to
       \begin{footnotesize}
       \[
        \begin{vNiceMatrix}[c,margin]
            \CodeBefore
            \rectanglecolor{blue!20}{1-5}{5-5}
            \rectanglecolor{blue!20}{7-5}{9-5}
            \rectanglecolor{red!20}{6-5}{6-5}
            \rectanglecolor{orange!20}{6-1}{6-4}
            \rectanglecolor{orange!20}{6-6}{6-9}
            \Body
            1-F_{l_1}(x_1) & \dots  & 1-F_{l_1}(x_k) & f_{l_1}(x_k) & 0 & f_{l_1}(x_{k+1})&\dots &
            f_{l_1}(x_n)\\
            -F_{r_1}(x_1)  & \dots & 1-F_{r_1}(x_k) & f_{r_1}(x_k)  &  0 & f_{r_1}(x_{k+1}) &  \dots &
            f_{r_1}(x_n)\\
            \vdots  & \ddots & \vdots & \vdots & \vdots &\vdots & \ddots & \vdots\\
            -F_{l_k}(x_1)  & \dots & 1-F_{l_k}(x_k) & f_{l_k}(x_k) & 0 & f_{l_k}(x_{k+1}) & \dots &
            f_{l_k}(x_n)\\
            -F_{r_k}(x_1)  & \dots & -F_{r_k}(x_k) & f_{r_k}(x_k) & 0 & f_{r_k}(x_{k+1}) & \dots &
            f_{r_k}(x_n)\\
            0 &  \dots & 0 & 0 &  1 & 0 & \dots & 0\\
            -F_{r_{k+1}}(x_1)  & \dots & -F_{r_{k+1}}(x_k) & f_{r_{k+1}}(x_k) &  0 & f_{r_{k+1}}(x_{k+1}) & \dots & f_{r_{k+1}}(x_n)\\
            \vdots  & \ddots & \vdots & \vdots & \vdots &\vdots &\ddots & \vdots\\
            -F_{l_n}(x_1)  & \dots & -F_{l_n}(x_k) & f_{l_n}(x_k) & 0 & f_{l_n}(x_{k+1}) &\dots &
            f_{l_n}(x_n)\\
        \end{vNiceMatrix}
       \]
       \[
        +
        \begin{vNiceMatrix}[c,margin]
            \CodeBefore
            \rectanglecolor{blue!20}{1-5}{4-5}
            \rectanglecolor{blue!20}{7-5}{9-5}
            \rectanglecolor{red!20}{6-5}{6-5}
            \rectanglecolor{orange!20}{6-1}{6-4}
            \rectanglecolor{orange!20}{6-6}{6-9}
            \rectanglecolor{green!20}{5-5}{5-5}
            \rectanglecolor{yellow!20}{5-1}{5-4}
            \rectanglecolor{yellow!20}{5-6}{5-9}
            \Body
            1-F_{l_1}(x_1)  & \dots  & 1-F_{l_1}(x_k) & f_{l_1}(x_k) & 1-F_{l_1}(x_{k+1}) & f_{l_1}(x_{k+1}) &\dots &
            f_{l_1}(x_n)\\
            -F_{r_1}(x_1)  & \dots & 1-F_{r_1}(x_k) & f_{r_1}(x_k)  &  1-F_{r_1}(x_{k+1})& f_{r_1}(x_{k+1}) &  \dots &
            f_{r_1}(x_n)\\
            \vdots  & \ddots & \vdots & \vdots & \vdots &\vdots & \ddots & \vdots\\
            -F_{l_k}(x_1)  & \dots & 1-F_{l_k}(x_k) & f_{l_k}(x_k) & 1-F_{l_k}(x_{k+1})& f_{l_k}(x_{k+1}) & \dots &
            f_{l_k}(x_n)\\
            -F_{r_k}(x_1)  & \dots & -F_{r_k}(x_k) & f_{r_k}(x_k) & -F_{r_k}(x_{k+1}) & f_{r_k}(x_{k+1}) & \dots &
            f_{r_k}(x_n)\\
            -F_{l_{k+1}}(x_1)  & \dots & -F_{l_{k+1}}(x_k) & f_{l_{k+1}}(x_k) &  -F_{l_{k+1}}(x_{k+1})& f_{l_{k+1}}(x_{k+1}) &  \dots & f_{l_{k+1}}(x_n)\\
            -F_{r_{k+1}}(x_1)  & \dots & -F_{r_{k+1}}(x_k) & f_{r_{k+1}}(x_k) &  -F_{r_{k+1}}(x_{k+1})  & f_{r_{k+1}}(x_{k+1}) &  \dots & f_{r_{k+1}}(x_n)\\
            \vdots  & \ddots & \vdots & \vdots & \vdots &\vdots &\ddots & \vdots\\
            -F_{l_n}(x_1)  & \dots & -F_{l_n}(x_k) & f_{l_n}(x_k) & -F_{l_n}(x_{k+1})& f_{l_n}(x_{k+1}) & \dots &
            f_{l_n}(x_n)\\
        \end{vNiceMatrix}
        \]
        \[
       +
        \begin{vNiceMatrix}[c,margin]
            \CodeBefore
            \rectanglecolor{blue!20}{1-5}{4-5}
            \rectanglecolor{blue!20}{6-5}{9-5}
            \rectanglecolor{green!20}{5-5}{5-5}
            \rectanglecolor{yellow!20}{5-1}{5-4}
            \rectanglecolor{yellow!20}{5-6}{5-9}
            \Body
            1-F_{l_1}(x_1)  & \dots  & 1-F_{l_1}(x_k) & f_{l_1}(x_k) & 0 & f_{l_1}(x_{k+1}) &\dots &
            f_{l_1}(x_n)\\
            -F_{r_1}(x_1)  & \dots & 1-F_{r_1}(x_k) & f_{r_1}(x_k)  &  0 & f_{r_1}(x_{k+1}) &  \dots &
            f_{r_1}(x_n)\\
            \vdots  & \ddots & \vdots & \vdots & \vdots &\vdots & \ddots & \vdots\\
            -F_{l_k}(x_1)  & \dots & 1-F_{l_k}(x_k) & f_{l_k}(x_k) & 0 & f_{l_k}(x_{k+1}) & \dots &
            f_{l_k}(x_n)\\
            0 & \dots & 0 & 0 &  1 & 0 & \dots & 0\\
            -F_{l_{k+1}}(x_1)  & \dots & -F_{l_{k+1}}(x_k) & f_{l_{k+1}}(x_k) & 0 & f_{l_{k+1}}(x_{k+1}) & \dots &
            f_{l_{k+1}}(x_n)\\
            -F_{r_{k+1}}(x_1)  & \dots & -F_{r_{k+1}}(x_k) & f_{r_{k+1}}(x_k) &  0 & f_{r_{k+1}}(x_{k+1}) & \dots & f_{r_{k+1}}(x_n)\\
            \vdots & \ddots & \vdots & \vdots & \vdots &\vdots &\ddots & \vdots\\
            -F_{l_n}(x_1) & \dots & -F_{l_n}(x_k) & f_{l_n}(x_k) & 0 & f_{l_n}(x_{k+1}) & \dots &
            f_{l_n}(x_n)\\
        \end{vNiceMatrix}.
       \]
       \end{footnotesize}
       Then since the first and the third determinants have the same elements as the second one except one row and column with all zeroes except their intersection, the above sum equals to
       \begin{footnotesize}
           \[
           \begin{vNiceMatrix}[c,margin]
            \CodeBefore
            \rectanglecolor{blue!20}{1-5}{4-5}
            \rectanglecolor{blue!20}{7-5}{9-5}
            \rectanglecolor{red!20}{6-5}{6-5}
            \rectanglecolor{orange!20}{6-1}{6-4}
            \rectanglecolor{orange!20}{6-6}{6-9}
            \rectanglecolor{green!20}{5-5}{5-5}
            \rectanglecolor{yellow!20}{5-1}{5-4}
            \rectanglecolor{yellow!20}{5-6}{5-9}
            \Body
            1-F_{l_1}(x_1)  & \dots  & 1-F_{l_1}(x_k) & f_{l_1}(x_k) & 1-F_{l_1}(x_{k+1}) & f_{l_1}(x_{k+1})&\dots &
            f_{l_1}(x_n)\\
            -F_{r_1}(x_1)  & \dots & 1-F_{r_1}(x_k) & f_{r_1}(x_k)  &  1-F_{r_1}(x_{k+1})& f_{r_1}(x_{k+1}) & \dots &
            f_{r_1}(x_n)\\
            \vdots & \ddots & \vdots & \vdots & \vdots &\vdots & \ddots & \vdots\\
            -F_{l_k}(x_1) & \dots & 1-F_{l_k}(x_k) & f_{l_k}(x_k) & 1-F_{l_k}(x_{k+1})& f_{l_k}(x_{k+1}) & \dots &
            f_{l_k}(x_n)\\
            -F_{r_k}(x_1) & \dots & -F_{r_k}(x_k) & f_{r_k}(x_k) & 1-F_{r_k}(x_{k+1}) & f_{r_k}(x_{k+1}) & \dots &
            f_{r_k}(x_n)\\
            -F_{l_{k+1}}(x_1) & \dots & -F_{l_{k+1}}(x_k) & f_{l_{k+1}}(x_k) &  1-F_{l_{k+1}}(x_{k+1})& f_{l_{k+1}}(x_{k+1}) &  \dots & f_{l_{k+1}}(x_n)\\
            -F_{r_{k+1}}(x_1)  & \dots & -F_{r_{k+1}}(x_k) & f_{r_{k+1}}(x_k) &  -F_{r_{k+1}}(x_{k+1})  & f_{r_{k+1}}(x_{k+1}) &  \dots & f_{r_{k+1}}(x_n)\\
            \vdots & \ddots & \vdots & \vdots & \vdots &\vdots &\ddots & \vdots\\
            -F_{l_n}(x_1) & \dots & -F_{l_n}(x_k) & f_{l_n}(x_k) & -F_{l_n}(x_{k+1})& f_{l_n}(x_{k+1}) &\dots &
            f_{l_n}(x_n)\\
        \end{vNiceMatrix}
           \]
       \end{footnotesize}
       Therefore, the formula is true for $k+1$.
    \end{proof}

    With the help of the joint densities we can express the probability for any possible scenario of finite number of coalescing Pólya Walks started from the same time by integrating it
        \[
         \int_{0 < x_1 < x_2 < \dots < x_n <1} 
        \begin{pmatrix}
        l_1 & r_1 & l_2 & r_2 &\dots & l_n & r_n\\
        x_1 & x_1 & x_2 & x_2 &\dots & x_n & x_n\\
        \end{pmatrix}
        \prod_{k=1}^{n}dx_k.
        \]
    However, it is a closed formula for large $n$ it is hard to calculate.
    For small $n$ we will give some exact calculations in the following section.

    \subsection{The number of components}
    We now investigate the number of components in the case of the Pólya Web (\Cref{fig:components_polya}).
    We defined it in \Cref{section_3} and proved a limit theorem in the general case.
    Now we just need to do some calculations to obtain the a formula for the Pólya Web.
    Recall from \cref{eq:number_of_components} that the number of components of the random graph which we get if we start coalescing walks from all point up to level $n \in \N$ is
    \[
    C_n = 1 + \sum_{k=1}^{n} \mathds{1}\left[\tau_{\lambda_{k-1},\lambda_{k}} = \infty \right],
    \qquad \text{where} \; \lambda_{k} = (k,n-k).
    \]
    To calculate the expectation, we need the probabilities
    \[
    \valseg{\tau_{\lambda_{k-1},\lambda_{k}} = \infty} \qquad \text{where} \; \lambda_{k} = (k,n-k)
    \]
    for all $1 \leq k \leq n$. The formula for them is stated in the following proposition.
    \begin{prp} \label{prp:hypergeometric}
        For all $0 \leq k \leq n-1$ if $\lambda_k = (k,n-k)$, then
        \[
        \valseg{\tau_{\lambda_{k-1},\lambda_{k}} = \infty} = {n-1 \choose k}^2{2(n-1) \choose 2k}^{-1}.
        \]
    \end{prp}
    \begin{figure}[H]
      \centering
      \includegraphics[width=0.6\textwidth]{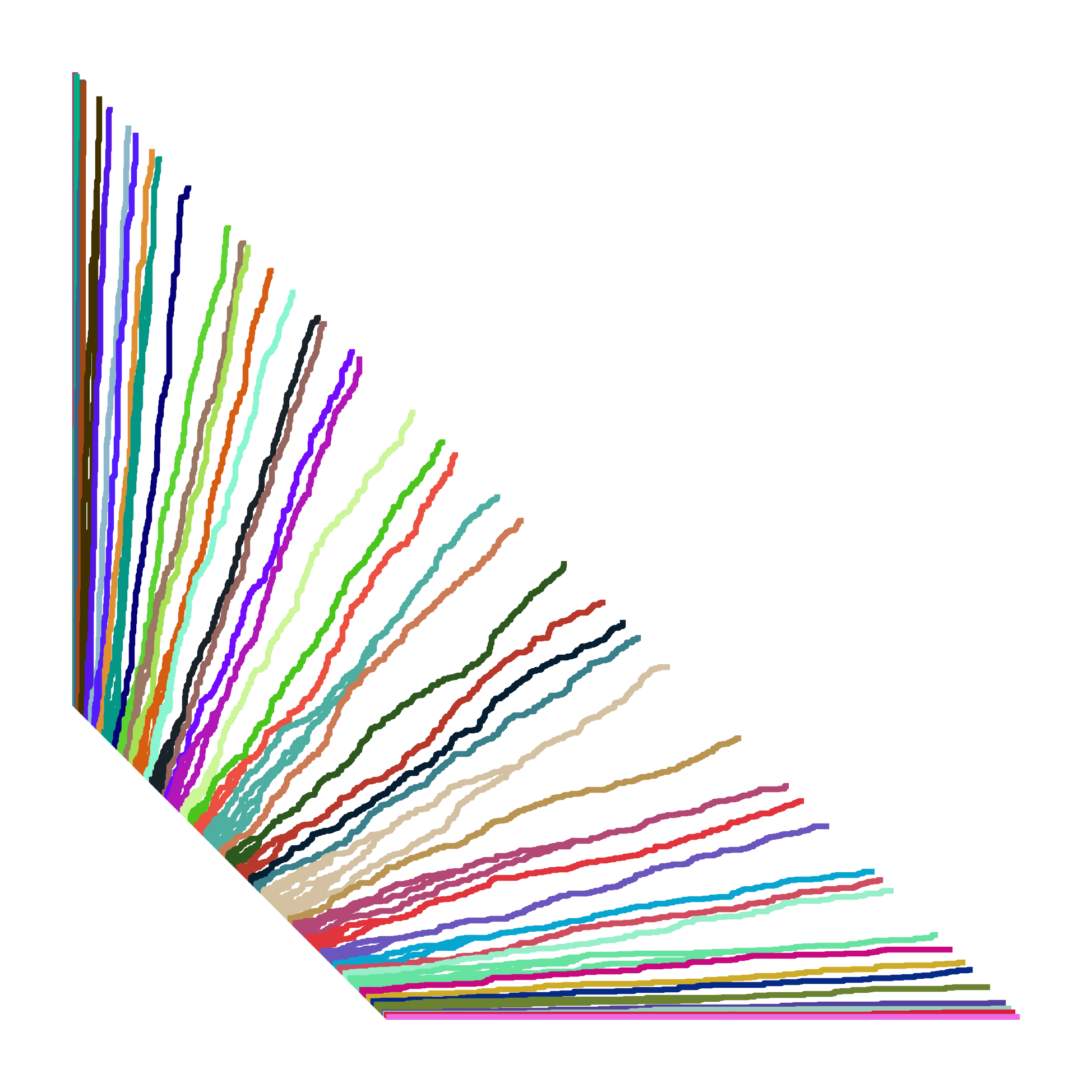}
      \caption{Coalescing Pólya Walks started at the same universal time. Different colours denote different components.}
      \label{fig:components_polya}
    \end{figure}
    \begin{proof}
    By \Cref{thm:joint_density} and \Cref{lmm:density_of_different_points}
    \[
    \valseg{\tau_{\lambda_{k-1},\lambda_{k}} = \infty} = \int_{0}^{1} \int_{0}^{y}
    \begin{vNiceMatrix}[c,margin]
        f_{\lambda_{k-1}}(x) & f_{\lambda_{k-1}}(y)\\
        f_{\lambda_{k}}(x) & f_{\lambda_{k}}(y)
    \end{vNiceMatrix}
    \,dx\,dy.
    \]
    Then a small calculus exercise gives us the desired result.
    \end{proof}
    Now we just need the asymptotics of the sum of those probabilities and the asymptotics of the expectation will follow immediately.
    \begin{prp} \label{prp:hpyergeometric_sum_limit}
    Let $\varepsilon > 0$,
      \[
      n_1 \defeq \ceil*{\varepsilon \sqrt{n}} \qquad \mathrm{and} \qquad n_2 \defeq n - \ceil*{\varepsilon \sqrt{n}}.
      \]
      Then
      \[
      \lim_{n \to \infty} \frac{1}{\sqrt{n}}\sum_{k=n_1}^{n_2} {n-1 \choose k}^{2} {2(n-1) \choose 2k}^{-1} = \sqrt{\pi}.
      \]
    \end{prp}
    \begin{proof}
    By Stirling's approximation it can be shown that there are sequences $a_n, b_n$ such that
    \[
    \lim_{n \to \infty} a_n = \lim_{n \to \infty} b_n = 1
    \]
    and
    \begin{equation} \label{eq:after_stirling}
      a_n \cdot I_n
      \leq \sum_{k=n_1}^{n_2} \frac{1}{\sqrt{n}}{n-1 \choose k}^{2} {2(n-1) \choose 2k}^{-1}
      \leq b_n \cdot I_n,
    \end{equation}
    where
    \[
    I_n \defeq \frac{1}{\sqrt{\pi}}\sum_{k=n_1}^{n_2}\ \frac{1}{n-1} \frac{1}{\sqrt{\frac{k}{n-1}\left(1-\frac{k}{n-1}\right)}}.
    \]
    However,
    \[
    \lim_{n \to \infty}I_n = \frac{1}{\sqrt{\pi}}\int_{0}^{1}\frac{dx}{\sqrt{x(1-x)}} = \sqrt{\pi}.
    \]
    \end{proof}
    Finally, we can show the asymptotic behaviour of the expectation and give any asymptotic bound on the variance.
    \begin{prp} \label{prp:expectation_and_variance_bounds}
    For any $n \in \N$
    \[
    \varhato{C_n} = \sqrt{n\pi}(1+o(1)) \qquad \text{and} \qquad
    \var{C_n} \leq \sqrt{n\pi}(1+o(1)).
    \]
    \end{prp}
    \begin{proof}
        First, we show the equality for the expectation.
        By \cref{eq:number_of_components} and \Cref{prp:hypergeometric} if we denote $\lambda_k = (k,n-k)$
        \begin{equation} \label{eq:components_expressed_with_indicators_1}
        \varhato{C_n} = 1 + \sum_{k=0}^{n-1}\valseg{\tau_{\lambda_{k-1},,\lambda_{k}} = \infty}
        = 1 +\sum_{k=0}^{n-1}{n-1 \choose k}^2{2(n-1) \choose2k}^{-1}.
        \end{equation}
        Then for any $\varepsilon > 0$ and
        \[
        n_1 \defeq \ceil*{\varepsilon\sqrt{n}} \qquad \text{and} \qquad
        n_2 \defeq n - \ceil*{\varepsilon\sqrt{n}}
        \]
        if we denote
        \[
        J_n \defeq \sum_{k=n_1}^{n_2} \frac{1}{\sqrt{n}}{n-1 \choose k}^{2} {2(n-1) \choose 2k}^{-1}
        \]
        we have the following inequality
        \[
        \frac{1}{\sqrt{n}} + J_n \leq \frac{\varhato{C_n}}{\sqrt{n}} \leq \frac{2\ceil*{\varepsilon\sqrt{n}}}{\sqrt{n}} + J_n.
        \]
        By \Cref{prp:hpyergeometric_sum_limit}
        \[
        \sqrt{\pi} \leq \liminf_{n \to \infty} \frac{\varhato{C_n}}{\sqrt{n}}
        \leq \limsup_{n \to \infty} \frac{\varhato{C_n}}{\sqrt{n}} \leq \sqrt{\pi}+2\varepsilon,
        \]
        which exactly means
        \[
        \varhato{C_n} = \sqrt{n\pi}(1+o(1)).
        \]
        Now we show the inequality for the variance.
        Recall that the indicator variables by \Cref{crl:limit_indicators_are_negatively_associated} are negatively associated. Therefore, by \cref{eq:bound_on_the_marginals} the covariance of any two different variables is not positive.
        Then
        \begin{multline*}
        \var{C_n} \leq \sum_{k=0}^{n-1}\var{\mathds{1}\left[\tau_{\lambda_{k-1},\lambda_{k}} = \infty\right]}
        = \sum_{k=0}^{n-1} \left[\valseg{\tau_{\lambda_{k-1},\lambda_{k}} = \infty}-\valseg{\tau_{\lambda_{k-1},\lambda_{k}} = \infty}^2\right]\\
        = \varhato{C_n} + \sum_{k=0}^{n-1}\valseg{\tau_{\lambda_{k-1},\lambda_{k}} = \infty}^2.
        \end{multline*}
        Applying the Stirling approximation again with the notations of \Cref{prp:hpyergeometric_sum_limit} for any $\varepsilon > 0$ there is a $c > 0$ such that
        \begin{multline*}
        \sum_{k=0}^{n-1}\valseg{\tau_{\lambda_{k-1},\lambda_{k}} = \infty}^2 \leq
        2\ceil*{\varepsilon \cdot \sqrt{n}} + b_n^2 \cdot \sum_{k=n_1}^{n_2}\left[\frac{1}{\sqrt{n}} \frac{1}{\sqrt{\frac{k}{n}\left(1-\frac{k}{n}\right)}}\right]^2\\ \leq
        2\ceil*{\varepsilon \cdot \sqrt{n}} + b_n^2 \cdot \sum_{k=n_1}^{n_2}\left(\frac{1}{k}+\frac{1}{n-k}\right)
        \leq 2\ceil*{\varepsilon \cdot \sqrt{n}} + b_n^2 \cdot c\log(n).
        \end{multline*}
        Therefore, for any $\varepsilon > 0$
        \[
        \lim_{n\to\infty} \frac{1}{\sqrt{n}}\sum_{k=0}^{n-1}\valseg{\tau_{\lambda_{k-1},\lambda_{k}} = \infty}^2 \leq 2\varepsilon.
        \]
        Then it follows by the previous result for $\varhato{C_n}$
        \[
        \limsup_{n \to \infty} \frac{\var{C_n}}{\sqrt{n}} \leq \liminf_{n \to \infty} \frac{\varhato{C_n}}{\sqrt{n}} = \sqrt{\pi}.
        \]
    \end{proof}
    \begin{prp} \label{prp:stong_law_for_the_polya_components}
        \[
        \lim_{n \to \infty} \frac{C_n}{\sqrt{n\pi}} = 1 \quad \text{almost surely}.
        \]
    \end{prp}
    From the formulas in \Cref{prp:expectation_and_variance_bounds} the convergence in probability already follows. However, by \Cref{thm:convergence_of_the_number_of_components} we know that we have almost sure convergence.
    \begin{proof}
        Since by \Cref{prp:expectation_and_variance_bounds}, we have for any $c > 1$
        \[
        \varhato{C_n} = \sqrt{n\pi}(1+o(1)) = \sqrt{\pi} \cdot n^{\frac{1}{2}}(1+o(1).
        \]
        Then applying \Cref{thm:convergence_of_the_number_of_components}
        \[
        \lim_{n \to \infty} \frac{C_n}{\varhato{C_n}} = 1 \quad \text{almost surely}.
        \]
        Moreover, using \Cref{prp:expectation_and_variance_bounds} again
        \[
        \lim_{n \to \infty} \frac{\varhato{C_n}}{\sqrt{n\pi}} = \lim_{n\to\infty}\frac{\sqrt{n\pi}(1+o(1))}{\sqrt{n\pi}} = \lim_{n\to\infty}1+o(1) = 1.
        \]
        Thus the limit holds.
    \end{proof}

    It is clear from \Cref{fig:components_polya} that the points do not contribute uniformly to the number of components.
    At the edges the probability of coalescing is lower compared to the middle.
    Therefore, we have more disjoint clusters at the edges.
    Now we investigate how different regions influence the components.

    For this fix an arbitrary ratio $0 < \alpha < 1$ and a small $0 < \varepsilon < \frac{1}{2}$.
    Let $\beta_n$ a sequence such that for some $c > 0$
    \[
    \beta_n \sim c \cdot n^{\frac{1}{2}+\varepsilon}.
    \]
    We use the standard notation for the relation $\sim$ between sequences.
    That is, $a_n \sim b_n$, if $\lim_{n \to \infty}a_n/b_n = 1$.

    Now let us consider for some $n \in \N$ the points $\lambda$ on the line $x+y=n$ in $\Lambda$ such that
    \[
    \lambda = (j,n-j) \qquad \text{such that} \qquad
    \alpha n - \beta_n \leq j \leq \alpha n + \beta_n.
    \]
    That is we consider approximately the $\beta_n$ neighbourhood of the point $(\alpha n, (1-\alpha)n)$.
    Let $C_{\alpha,n}$ denote the number of components of the infinite random graph obtained by starting coalescing Pólya Walks from these points.
    Then the following proposition holds.

    \begin{prp}
    \[
        \lim_{n \to \infty} \frac{\sqrt{n}}{\beta_n} C_{\alpha,n} = \frac{1}{\sqrt{\pi\alpha(1-\alpha)}}
        \qquad \text{almost surely}.
    \]
    \end{prp}
    \begin{proof}
    We just need to modify our previous results. Clearly modifying \cref{eq:components_expressed_with_indicators_1}
    \[
    \varhato{C_{\alpha,n}} = 1 + \sum_{k=\ceil*{\alpha n - \beta_n}}^{\floor*{\alpha n + \beta_n}}\valseg{\tau_{\lambda_{k-1},\lambda_{k}} = \infty} =
    1 + \sum_{k=\ceil*{\alpha n - \beta_n}}^{\floor*{\alpha n + \beta_n}}{n-1 \choose k}^2{2(n-1) \choose2k}^{-1}.
    \]
    Moreover, considering and modifying \cref{eq:after_stirling} as $n \to \infty$ we have
    \[
    \begin{split}
    \frac{n}{\beta_n} \cdot \frac{1}{\sqrt{n}}\sum_{k=\ceil*{\alpha n - \beta_n}}^{\floor*{\alpha n + \beta_n}}{n-1 \choose k}^2{2(n-1) \choose2k}^{-1}
    &\sim
    \frac{n}{\beta_n}\sum_{k=\ceil*{\alpha n - \beta_n}}^{\floor*{\alpha n + \beta_n}}\frac{1}{n-1} \frac{1}{\sqrt{\pi}}\frac{1}{\sqrt{\frac{k}{n-1}\left(1-\frac{k}{n-1}\right)}}\\
    &\sim
    \frac{n}{\beta_n}\int_{\alpha - \frac{\beta_n}{n}}^{\alpha+\frac{\beta_n}{n}}\frac{dx}{\sqrt{\pi x (1-x)}}.
    \end{split}
    \]
    Then, since $0 < \varepsilon < \frac{1}{2}$
    \[
    \frac{n}{\beta_n} \sim c \cdot n^{\frac{1}{2}-\varepsilon} \to \infty \qquad
    \text{as $n \to \infty$},
    \]
    we have that
    \[
    \lim_{n \to \infty}\frac{\sqrt{n}}{\beta_n}\sum_{k=\ceil*{\alpha n - \beta_n}}^{\floor*{\alpha n + \beta_n}}{n-1 \choose k}^2{2(n-1) \choose2k}^{-1}
    = \frac{1}{\sqrt{\pi \alpha (1-\alpha)}}.
    \]
    Therefore,
    \[
    \lim_{n \to \infty}\frac{\sqrt{n}}{\beta_n}\varhato{C_{\alpha,n}} = \frac{1}{\sqrt{\pi \alpha (1-\alpha)}}.
    \]
    Now clearly, for come $C > 0$
    \[
    \varhato{C_{n,\alpha}} \geq C n^{\varepsilon}.
    \]
    Therefore, considering the results so far and \Cref{thm:convergence_of_the_number_of_components} we conclude
    \[
    \lim_{n \to \infty} \frac{\sqrt{n}}{\beta_n}C_{\alpha,n} = \frac{1}{\sqrt{\pi \alpha (1-\alpha)}} \qquad \text{almost surely}.
    \]
    \end{proof}
    This result justifies our conjecture from the picture.
    Since the closer we are to the edges ($\alpha$ is closer to one or to zero), the contribution to the number of components in a small neighbourhood of points is larger. 

    \newpage

    \section{Scaling of the Pólya Web at the edges} \label{section_5}
    \thispagestyle{plain}
    In this section we focus on the scaling of the Pólya Web.
    Our goal is to find a limiting process after an appropriate scaling.
    We have to distinguish between two different regimes.
    The scaling at the edges and in the bulk.
    We will explain the meaning of these terms in the corresponding subsections.
    However, before that it is important to mention that we only managed to find the scaling on the edges.
    Therefore, this section is about the scaling at the edges.
    For the scaling in the bulk, we have a good conjecture but the technical details need to be worked out and will probably be demanding.
    
    \subsection{Poisson approximation}
    Consider the following regime.
    Fix and arbitrary $N,M \in \N$ and real numbers $0 < a < b$.
    Then let us consider the vertices in $\left([a,b] \times [N,M]\right) \cap \Lambda$.
    We are interested in the local asymptotic behaviour of the Pólya Web on the previously given domain.
    First, we show a Poisson approximation argument.
    \begin{figure}[H]
      \centering
      \includegraphics[width=0.6\textwidth]{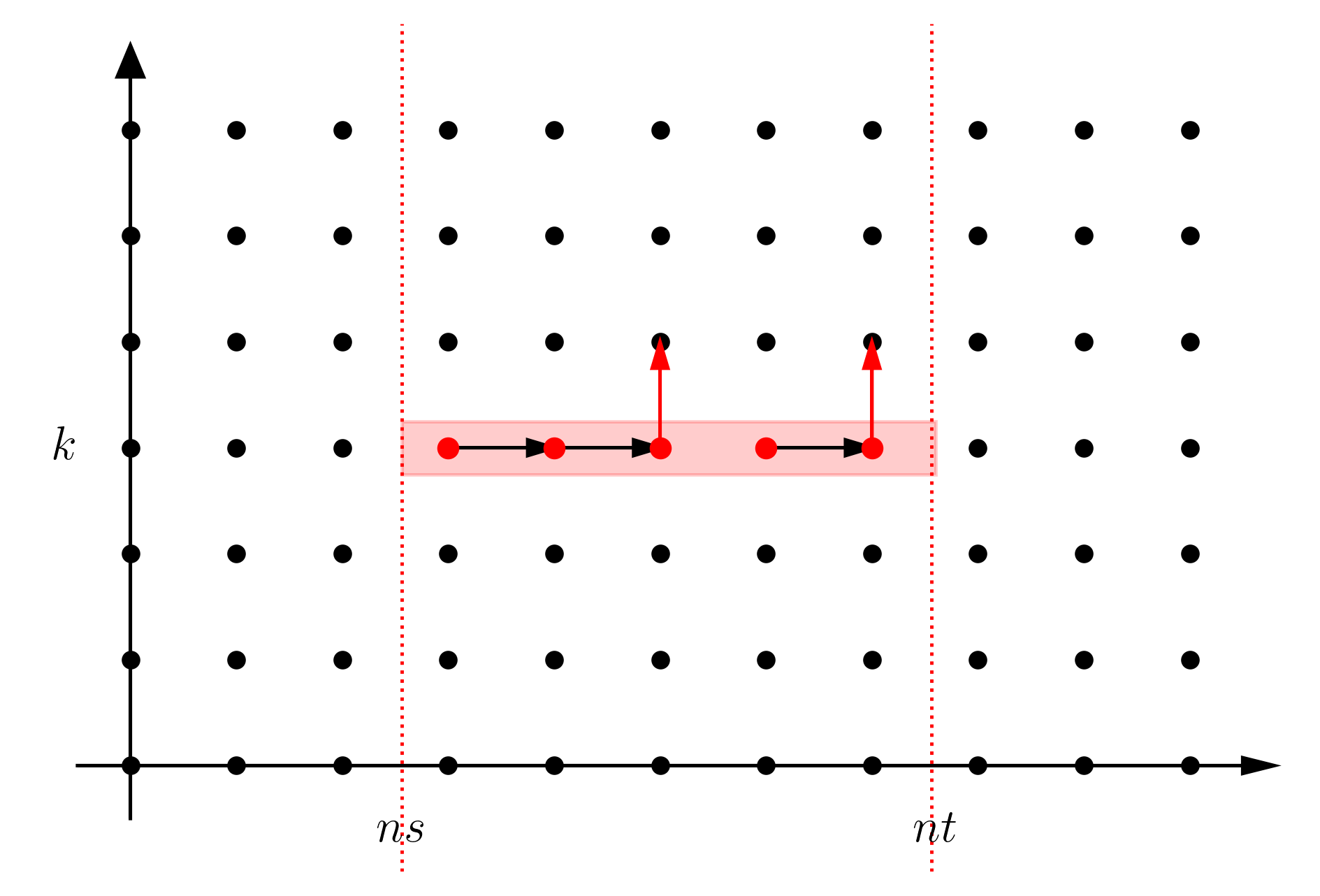}
      \caption{$J_n(k,s,t)$ defined in \cref{eq:number_of_upward_steps} is the number of steps upward in the red box. In this realization it equals to two.}
      \label{fig:poisson_box}
    \end{figure}

    For that let us define for $k \in \N$ and real numbers $0 < s \leq t$ let us denote
    \begin{equation} \label{eq:number_of_upward_steps}
        J_n(k,s,t) \defeq \# \left\{ \lambda \in  \left([ns,nt] \times \left\{ k \right\}\right) \cap \Lambda : \; \omega_\lambda = (0,1)\right\},
    \end{equation}
    that is the number of steps upward from points inside $[ns,nt] \times \{k\}$ in the Pólya Web (\Cref{fig:poisson_box}). In this case we have the following convergence. 
    
    \begin{prp} \label{prp:individual_poisson}
        \[
        J_n(k,s,t) \overset{d}{\to} \text{Poi}\left( \int_{s}^{t} \frac{k}{r}\,dr\right) \quad \text{as $n \to \infty$}.
        \]
    \end{prp}
    \begin{proof}
        First notice by the definition of $J_n(k,s,t)$ (\Cref{fig:poisson_box})
        \begin{equation} \label{eq:jumps_as_ber}
        J_n(k,s,t) \overset{d}{=} \sum_{j=\ceil*{ns}}^{\floor*{nt}} B_{k,j},
        \quad \text{where} \quad
        B_{k,j} \sim \text{Ber}\left(\frac{k}{j+k}\right) \; \text{independent}.
        \end{equation}
        Therefore, the characteristic function of $J_n(k,s,t)$ for any real number $x$ is
        \[
        \prod_{j=\ceil*{ns}}^{\floor*{nt}}\left( \left(1-\frac{k}{j+k}\right)+\frac{k}{j+k}e^{ix}\right)=
        \prod_{j=\ceil*{ns}+k}^{\floor*{nt}+k}\left( \left(1-\frac{k}{j}\right)+\frac{k}{j}e^{ix}\right).
        \]
        After taking the natural logarithm of the characteristic function it converges as $n \to \infty$ to
        \begin{multline*}
        \sum_{j=\ceil*{ns}+k}^{\floor*{nt}+k} \log \left(1+\frac{k}{j}(e^{ix}-1)\right) =
        \sum_{j=\ceil*{ns}+k}^{\floor*{nt}+k}\left[\frac{k}{j}(e^{ix}-1) + o\left(\frac{1}{j^2}\right)\right]\\=
        \sum_{j=\ceil*{ns}+k}^{\floor*{nt}+k} \frac{k}{j}(e^{ix}-1) + o(1) =
        \sum_{j=\ceil*{ns}+k}^{\floor*{nt}+k} \frac{1}{n}\frac{k}{\frac{j}{n}}(e^{ix}-1) + o(1)
        \to \int_{s}^{t}\frac{k}{r}\,dr \cdot (e^{ix}-1).
        \end{multline*}
        It follows that the characteristic function of $J_n(k,s,t)$ converges at every real $x$ as $n \to \infty$ to
        \[
        \exp\left( \int_{s}^{t}\frac{k}{r}\,dr \cdot (e^{ix}-1) \right).
        \]
        This completes the proof.
    \end{proof}
    Therefore, if we consider the processes $J_n(k,s,t)$ on the compact set $\left([a,b]\times[K,L]\right) \cap \Lambda$, then the following convergence holds (\Cref{fig:poisson_multiple_boxes}).
    \begin{prp}\label{prp:joint_poisson}
        For any $M \in \N$ and integers $1 \leq k_j \leq N \; (j=1,\dots,M)$ and real numbers $0 < s_j \leq t_j \; (j=1,\dots,M)$ such that the sets $[s_j,t_j] \times \left\{k_j\right\} \; (j=1,\dots,M)$ are disjoint
        \[
        \left(J_n(k_j,s_j,t_j)\right)_{j=1}^{M} \overset{d}{\to}
        \left(J(k_j,s_j,t_j)\right)_{j=1}^{M},
        \]
        where $J(k_j,s_j,t_j) \; (j=1,\dots,M)$ are independent Poisson with mean
        \[
        \varhato{J(k_j,s_j,t_j)} =
        \int_{s_j}^{t_j} \frac{k_j}{r}\,dr.
        \]
    \end{prp}
    \begin{proof}
        Since the sets $[s_j,t_j] \times \left\{k_j\right\} \; (j=1,\dots,M)$ are disjoint, we can write each $J_n(k_j,s_j,t_j)$ as the sum of independent Bernoulli random variables as in \cref{eq:jumps_as_ber} without any of the two having the same variables.
        Therefore, the variables $J_n(k_j,s_j,t_j)$ are independent and by \Cref{prp:individual_poisson} the convergence holds for any $j=1,\dots,M$ and the limiting variables are independent.
    \end{proof}

    \begin{figure}[H]
      \centering
      \includegraphics[width=0.6\textwidth]{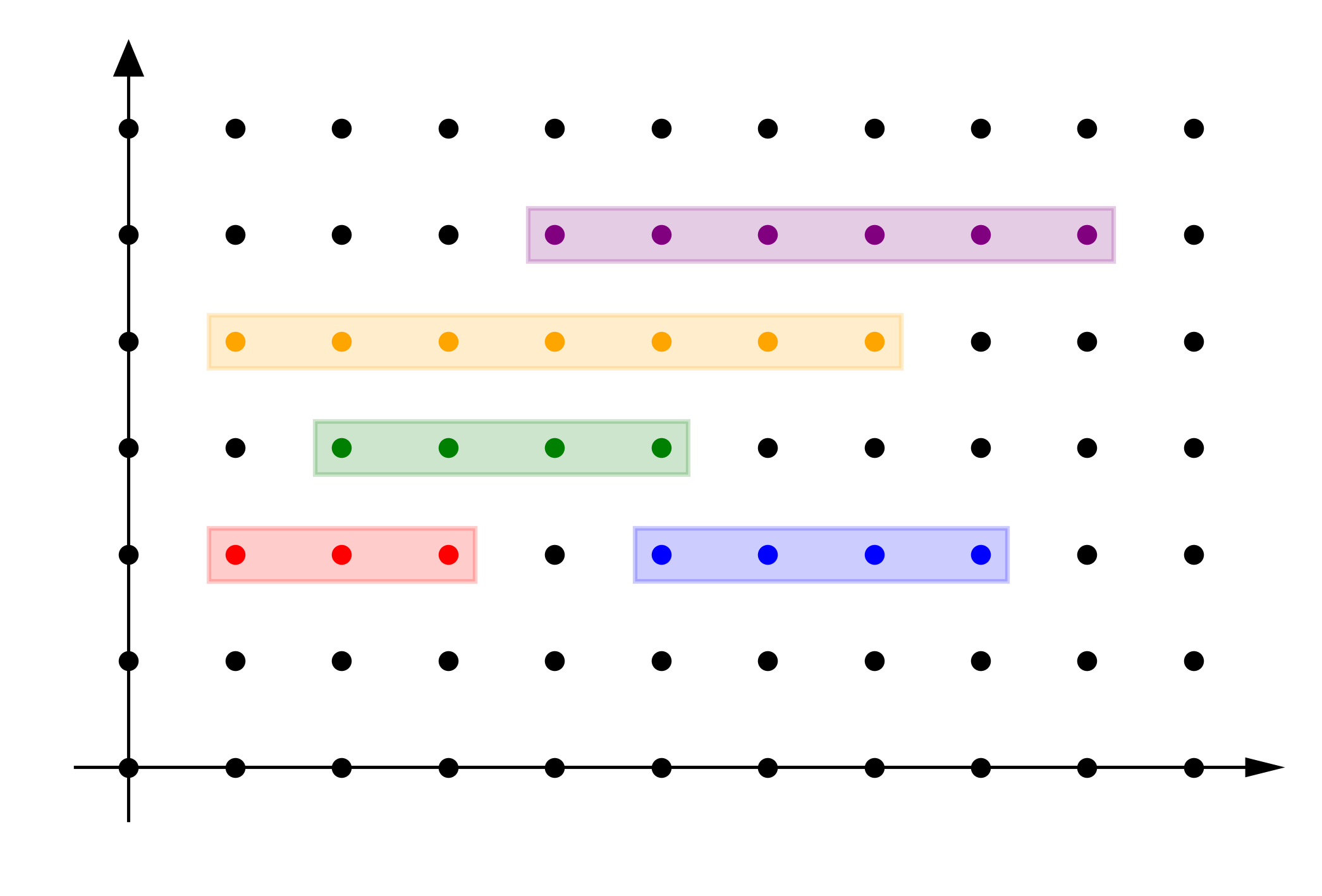}
      \caption{The random variables in \Cref{prp:joint_poisson} converge in distribution to independent random variables with Poisson distribution.}
      \label{fig:poisson_multiple_boxes}
    \end{figure}
    
    From this it is clear that the Pólya Web locally on the compact set $[a,b]\times[K,L]$ converges to the collection of coalescing time-inhomogeneous Poisson jump process.
    That is the processes for a fixed $k \geq 1$ and $s > 0$
    \[
    t \mapsto Y_{(k,\ceil*{ns})}(\floor*{nt})
    \]
    locally converge to the appropriate coalescing Poisson processes.
    Finally notice that if we rescale time (by taking the $\log$) then we can transform our time-inhomogeneous Poisson processes into time-homogeneous ones with extending them to the whole real line.

    \begin{prp}
        Let $(J_k(t))_{t>0}$ be a time inhomogeneous Poisson process with mean
        \[
        \varhato{J_k(t)-J_k(s)} = \int_{s}^{t}\frac{k}{r}\,dr.
        \]
        Then the process $\tilde{J}_k(t) \defeq J_k(e^{t}), \; t \in \R$ is a time-homogeneous Poisson process with mean
        \[
        \varhato{\tilde{J}_k(t)- \tilde{J}_k(s)} = k(t-s).
        \]
    \end{prp}
    \begin{proof}
        It is clearly a Poisson process with mean
        \[
        \varhato{\tilde{J}_k(t)- \tilde{J}_k(s)} = \varhato{J_k(e^{t})- J_k(e^{s})} =
        \int_{e^{s}}^{e^{t}}\frac{k}{r}\,dr = \int_{s}^{t}k\,dr = k(t-s).
        \]
    \end{proof}
    This motivates the construction in the next subsection.

    \subsection{The Yule Web} \label{yule_web}

    Consider for $1 \leq k$ independent Poisson point processes $A_k$ on the real line with mean $k$.
    Then let
    \[
    \mathcal{A}_k \defeq \left\{s \in \R : \; \text{there is a jump at $s$ in $A_k$}\right\}
    \]
    be the set of the points of jumps on the real line of the Poisson point process with rate $k$.
    Moreover, let
    \[
    \alpha_{k,s} \defeq \inf\left\{t > s : t \in \mathcal{A}_{k} \right\}.
    \]
    Then let us define the following sequence of stopping times (\Cref{fig:yule_stopping_times})
    \begin{equation} \label{eq:stopping_times_for_yule}
    \mathcal{T}_{k,s}(0) \defeq s \qquad \text{and} \qquad
    \mathcal{T}_{k,s}(n+1) \defeq \alpha_{k+n,\mathcal{T}_{k,s}(n)}.
    \end{equation}
    \begin{figure}[H]
      \centering
      \includegraphics[width=0.75\textwidth]{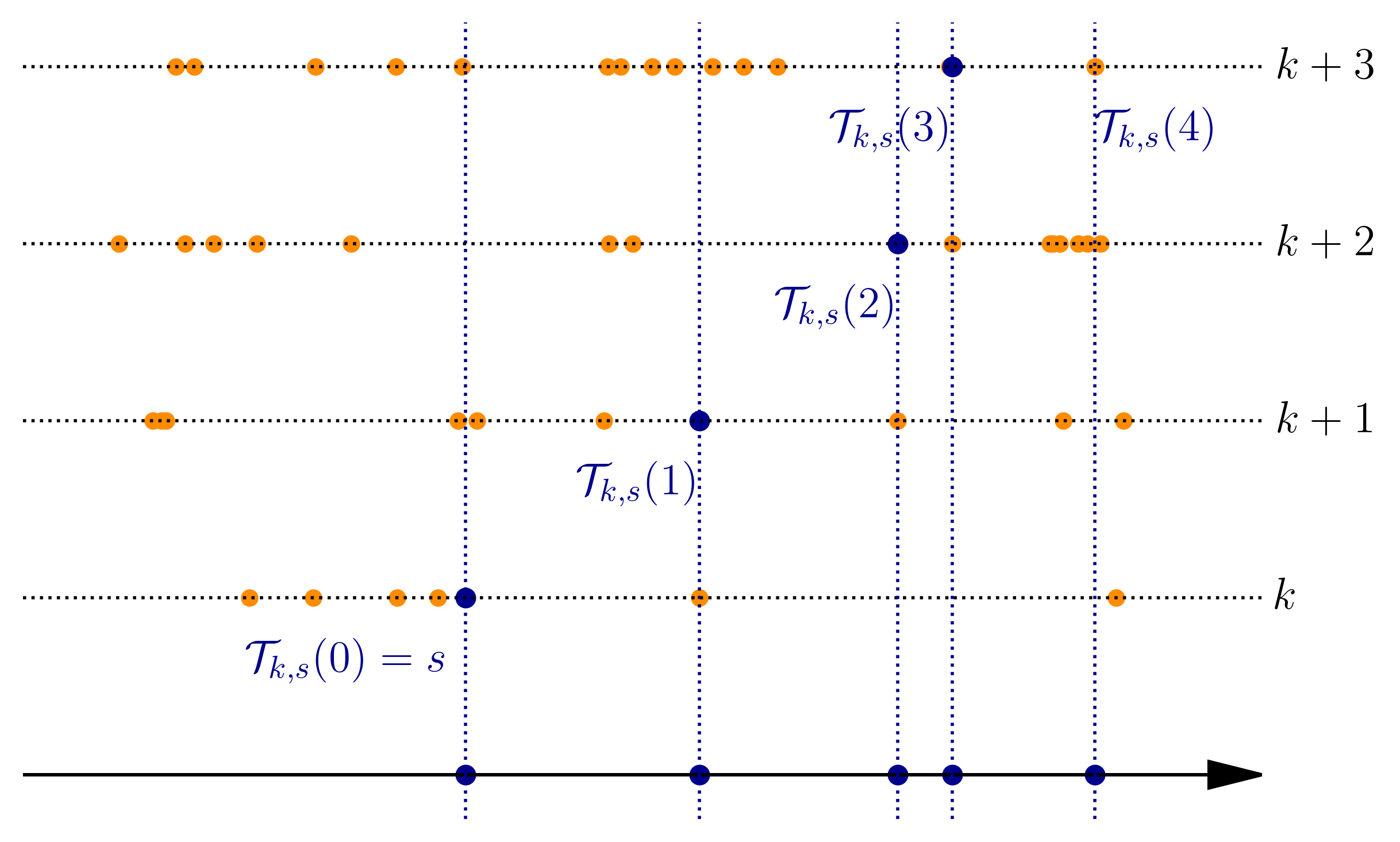}
      \caption{The stopping times for the Yule Web in \cref{eq:stopping_times_for_yule}.}
      \label{fig:yule_stopping_times}
    \end{figure}
    
    Then we can define the Yule Walk started from the point $(k,s)$ (\Cref{fig:single_yule_walk}) by the following way
    \begin{equation} \label{eq:yule_walk_definition}
    V_{k,s}(t)\defeq k + \sup\left\{ n \geq 0 : \; \mathcal{T}_{k,s}(n) \leq t\right\}
    \qquad (t \geq s).
    \end{equation}
    Then a trivial observation is that
     \begin{equation} \label{eq:exponential_ladder_times}
         \mathcal{T}_{k,s}(0) = s \qquad \text{and} \qquad
         \mathcal{T}_{k,s}(n+1) - \mathcal{T}_{k,s}(n) \sim \text{Exp}(k+n).
     \end{equation}
    \begin{figure}[H]
      \centering
      \includegraphics[width=0.75\textwidth]{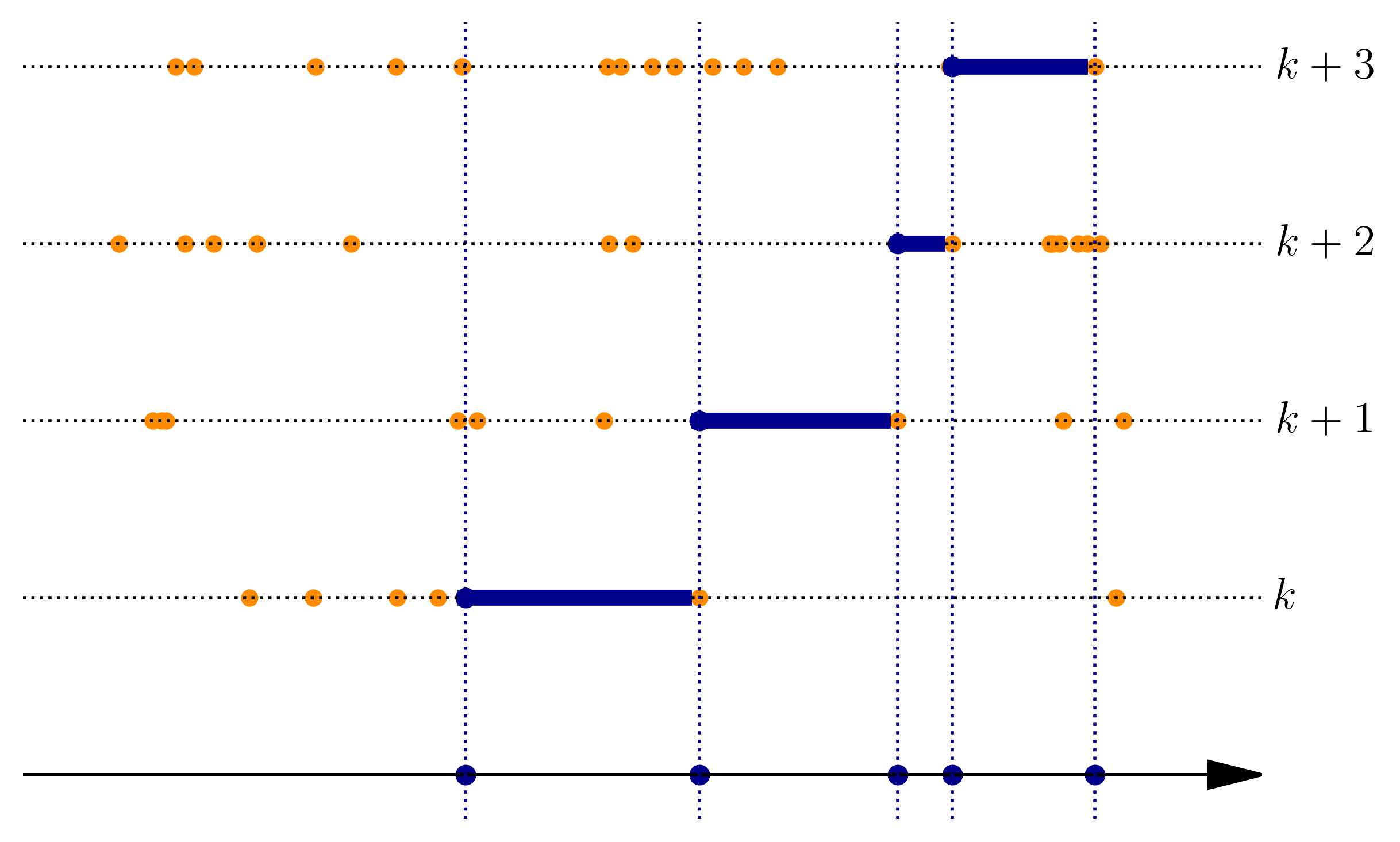}
      \caption{The path of a Yule Walk defined in \cref{eq:yule_walk_definition} in the same realization as in \Cref{fig:yule_stopping_times}.}
      \label{fig:single_yule_walk}
    \end{figure}
    This justifies the name Yule Walk.
    Since the time between two jumps are independent exponentials with parameter $n$ from the $n$th level to the $(n+1)$th, it follows that the process $(V_{k,s}(t))$ is a Yule process started from $V_{k,s}(s) = k$.
    We call this object of jointly defined Yule Walks the \textit{Yule Web}.
    The following propositions are well-known.
    Yet we give elementary proofs using our construction.
    The first one is about the distribution of the position of Yule Walk at a fixed time $t \geq s$.
    \begin{prp} \label{prp:yule_and_negbinom}
    For any $t \geq s$
    \[
        V_{k,s}(t) \sim \text{Negbinom}(k,e^{-(t-s)})
    \]
    \end{prp}
    \begin{proof}
        Notice that for $n \geq k$
        \[
        \valseg{V_{k,s}(t) \geq n} = \valseg{\mathcal{T}_{k,s}(n-k)\leq t}.
        \]
        Notice that by \cref{eq:exponential_ladder_times}
        \[
        \mathcal{T}_{k,s}(n-k) = s +
        \sum_{j=0}^{n-k-1} \underbrace{\mathcal{T}_{k,s}(j+1) - \mathcal{T}_{k,s}(j)}_{\text{independent Exp($k+j$)}}.
        \]
        There is a well-known formula for the distribution function of the sum of independent exponentials with different parameters.
        That is for $E_j \sim \text{Exp}(\lambda_j) \; j = 1,\dots,n$ such that $\lambda_j \neq \lambda_i \; (i\neq j )$ the cumulative distribution function of $E_1+ \dots+ E_n$ at $t > 0$ is
        \[
        1 - \sum_{j=1}^{n}e^{-\lambda_jt}\prod_{\substack{i=1\\i\neq j}}^{n}\frac{\lambda_j}{\lambda_j - \lambda_i}
        \]
        In this case this implies after an algebraic manipulation
        \[
        \begin{split}
            \valseg{\mathcal{T}_{k,s}(n-k)\leq t} &= 1 - \sum_{j=0}^{n-k-1}e^{-(k+j)(t-s)}\prod_{\substack{i=0\\i\neq j}}^{n-k-1}\frac{k+i}{(k+i)-(k+j)}\\
        &= 1 - \sum_{j=0}^{n-k-1}(-1)^{j}e^{-(k+j)(t-s)}\frac{n-1}{k+j}{n-2 \choose k-1}{n-1-k \choose j}.
        \end{split}
        \]
        Then notice
        \[
        \valseg{V_{k,s}(t) = n} = \valseg{V_{k,s}(t) \geq n} - \valseg{V_{k,s}(t) \geq n+1},
        \]
        thus, after comparing it with the previous result denoting $p = e^{-(t-s)}$
        \[
        \begin{split}
            &\valseg{V_{k,s}(t) = n}\\
            &=\sum_{j=0}^{n-k}(-1)^{j}p^{k+j}\frac{n}{k+j}{n-1 \choose k-1}{n-k \choose j}
            - \sum_{j=0}^{n-k-1}(-1)^{j}p^{k+j}\frac{n-1}{k+j}{n-2 \choose k-1}{n-1-k \choose j}.\\
            &= {n-1 \choose k-1}p^k \left[\sum_{j=0}^{n-k}(-p)^{j}\frac{n}{k+j}{n-k \choose j} - \sum_{j=0}^{n-1-k}(-p)^{j}\frac{n-k}{k+j}{n-1-k \choose j}\right]\\
            &= {n-1 \choose k-1}p^k \left[(-p)^{n-k} + \sum_{j=0}^{n-1-k}(-p)^{j}\left(\frac{n}{k+j}{n-k \choose j}-\frac{n-k}{k+j}{n-1-k \choose j}\right)\right]\\
            &= {n-1 \choose k-1}p^k \left[(-p)^{n-k} + \sum_{j=0}^{n-1-k}(-p)^{j}{n-k \choose j}\right]\\
            &={n-1 \choose k-1}p^k\sum_{j=0}^{n-k}(-p)^{j}{n-k \choose j}\\
            &= {n-1 \choose k-1}p^k(1-p)^{n-k}\\
        \end{split}
        \]
        which is the probability mass function of $\text{Negbinom}(k,p)$ at $n \geq k$, where $p=e^{-(t-s)}$.
    \end{proof}
    The second one shows that the Yule Walk is a branching process in continuous time.
    \begin{prp} \label{prp:yule_branching}
    For any $s \leq r \leq t$
        \[
        V_{k,s}(t) \overset{d}{=} \sum_{j=1}^{V_{k,s}(r)}\tilde{V}_{1,r}^{(j)}(t),
        \]
        where $\tilde{V}_{1,r}^{(j)}(t)$ are i.i.d. Yule Walks.
    \end{prp}
    \begin{proof}
        The probability generating function of ${V_{k,s}(r)}$ by \Cref{prp:yule_and_negbinom} is
        \[
        P(z) = \left(\frac{e^{-(r-s)}z}{1-(1-e^{-(r-s)})z}\right)^k,
        \]
        and the probability generating function of ${V_{1,r}^{(j)}(t)}$ is
        \[
        Q(z) = \frac{e^{-k(t-r)}z}{1-(1-e^{-k(t-r)})z}.
        \]
        Therefore, the probability generating function of $\sum_{j=1}^{V_{k,r}(t)}\tilde{V}_{1,r}^{(j)}(t)$ is
        \[
        P(Q(z)) = \left(\frac{e^{-(t-s)}z}{1-(1-e^{-(t-s)})z}\right)^k,
        \]
        which is exactly the probability generating function of $V_{k,s}(t)$.
    \end{proof}
    After an appropriate scaling of a Yule Walk we obtain a martingale. This is the analogue of the martingale in the case of the Pólya Walk defined in \cref{eq:def_of_the_ration}.
    This will play a similar role in the case of the Yule Web.
    \begin{prp}\label{prp:yule_martingale}
        The process
        \[
        U_{k,s}(t) \defeq e^{-(t-s)}V_{k,s}(t) \qquad (t \geq s)
        \]
        is a  martingale with respect to the natural filtration of $(V_{k,s}(t))_{t\geq s}$. 
    \end{prp}
    \begin{proof}
        The process is clearly adapted and has a finite first absolute moment for any $t \geq s$.
        Moreover, for any $t \geq r \geq s$ by \Cref{prp:yule_branching} if we denote $\mathcal{F}_r \defeq \sigma\left(V_{k,s}(u) : \; s \leq u \leq r \right)$
        \[
        \varhato{U_{k,s}(t)| \, \mathcal{F}_r} =
        e^{-(t-s)}\varhato{ \left. \sum_{j=0}^{V_{k,s}(r)}\tilde{V}_{1,r}^{(j)}(t)\,\right| \, \mathcal{F}_r} =
        e^{-(t-s)}V_{k,s}(r)\varhato{\tilde{V}_{1,r}^{(1)}(t)}.
        \]
        Notice that using \Cref{prp:yule_and_negbinom}
        \[
        \varhato{\tilde{V}_{1,r}^{(1)}(t)} = e^{t-r}.
        \]
        Therefore,
        \[
        \varhato{U_{k,s}(t)| \, \mathcal{F}_r} = e^{-(r-s)}V_{k,s}(r) = U_{k,s}(r).
        \]
    \end{proof}
    Finally we identity the limiting random variable of the previously obtained martingale.
    \begin{prp}\label{prp:yule_convergence_to_gamma}
        \[
        \lim_{t \to \infty}U_{k,s}(t) \eqdef U_{k} \sim \text{Gamma}(k,1)
        \qquad \text{almost surely}.
        \]
    \end{prp}
    \begin{proof}
    Since $U_{k,s}(t)$ is a non-negative martingale, it converges almost surely as $t \to \infty$. Therefore, $U_k$ is well-defined.
    For its distribution using \Cref{prp:yule_and_negbinom} the characteristic function of $U_{k,s}(t)$ at any $x \in \R$
    \[
    \varhato{e^{ixU_{k,s}(t)}} =
    \left(\frac{e^{-(t-s)}e^{ie^{-(t-s)}x}}{1-(1-e^{-(t-s)})e^{ie^{-(t-s)}x}}\right)^k
    \to \left(\frac{1}{1-ix}\right)^k \quad \text{as $t \to \infty$}.
    \]
    This is exactly the characteristic function of $\text{Gamma}(k,1)$ at $x$.
    \end{proof}

    Now consider two Yule Walks processes started from different points $(k_1,s_1)$ and $(k_2,s_2)$.
    Then clearly by the definition of the Yule Web they move independently until a time they meet.
    From that time they follow the same paths for the rest of the trajectory.
    Therefore, the Yule Web is the web of coalescing Yule Walks started from the points $\left\{ (k,s) : k \in \N^+, s \in \R\right\}$ (\Cref{fig:yule_web}).
    \begin{figure}[H]
      \centering
      \includegraphics[width=0.75\textwidth]{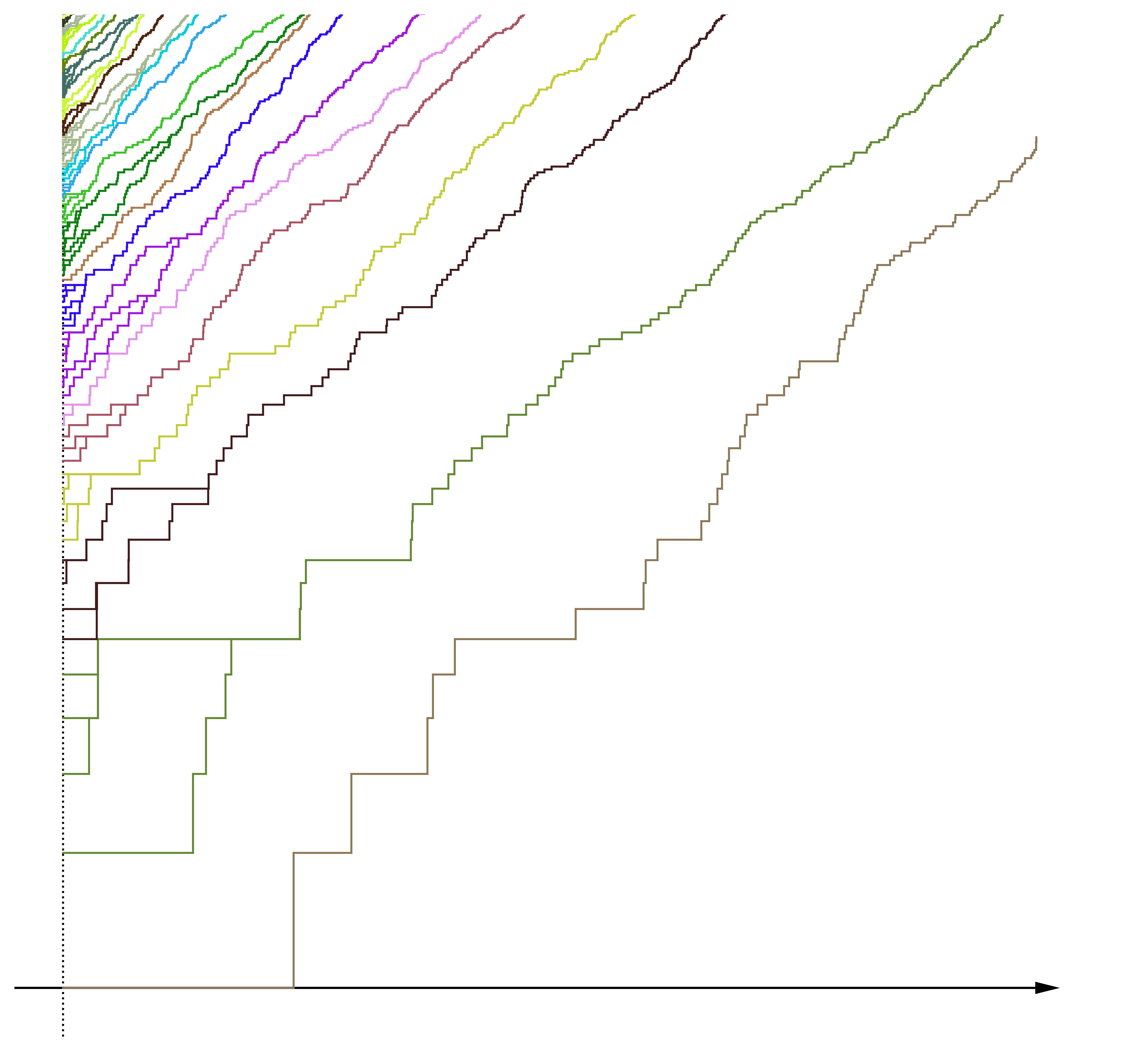}
      \caption{Coalescing Yule Walks started from the same time on a log-scale.}
      \label{fig:yule_web}
    \end{figure}
    It is easy to check that the equivalent versions of the statements in \Cref{section_4} hold for the for the Yule Web.
    Most importantly if we denote the density of $\text{Gamma}(k,1)$ by
    \[
    g_{k}(x) \defeq \frac{1}{\Gamma(k)}x^{k-1}e^{-x}\cdot\mathds{1}\left[x>0\right]
    \]
    and the cumulative distribution function by
    \[
    G_{k}(x) \defeq \int_{-\infty}^{x}g_k(u)\,du,
    \]
    then the following statements hold for coalescing Yule Walks started from the same fixed time $s\in\R$.
    \begin{thm} \label{joint_density_for_yule}
        For any $1 \leq l_1 \leq r_1 < \dots l_n \leq r_n$ the joint density of the limiting random variables $U_{l_1}, U_{r_1},\dots,U_{l_n},U_{r_n}$ for any $0 < x_1 < \dots < x_n$ is
        \[
        \begin{pmatrix}
        l_1 & r_1 & l_2 & r_2 &\dots & l_n & r_n\\
        x_1 & x_1 & x_2 & x_2 &\dots & x_n & x_n\\
        \end{pmatrix}
        \]
        \[
        =
        \begin{vNiceMatrix}[c,margin]
        1-G_{l_1}(x_1) & g_{l_1}(x_1) & 1-G_{l_1}(x_2) & g_{l_1}(x_2) & \dots &
        1-G_{l_1}(x_n) & g_{l_1}(x_n)\\
        -G_{r_1}(x_1) & g_{r_1}(x_1) & 1-G_{r_1}(x_2) & g_{r_1}(x_2) & \dots &
        1-G_{r_1}(x_n) & g_{r_1}(x_n)\\
        -G_{l_2}(x_1) & g_{l_2}(x_1) & 1-G_{l_2}(x_2) & g_{l_2}(x_2) & \dots &
        1-G_{l_2}(x_n) & g_{l_2}(x_n)\\
        -G_{r_2}(x_1) & g_{r_2}(x_1) & -G_{r_2}(x_2) & g_{r_2}(x_2) & \dots &
        1-G_{r_2}(x_n) & g_{r_2}(x_n)\\
        \vdots & \vdots & \vdots & \vdots & \ddots & \vdots &\vdots\\
        -G_{l_n}(x_1) & g_{l_n}(x_1) & -G_{l_n}(x_2) & g_{l_n}(x_2) & \dots &
        1-G_{l_n}(x_n) & g_{l_n}(x_n)\\
        -G_{r_n}(x_1) & g_{r_n}(x_1) & -G_{r_n}(x_2) & g_{r_n}(x_2) & \dots &
        -G_{r_n}(x_n) & g_{r_n}(x_n)\\
    \end{vNiceMatrix}_{2n \times 2n}.
        \]
    \end{thm}
    This theorem is clearly the analogous version of \Cref{thm:joint_density}.
    Moreover, if we introduce the stopping times
    \[
    \tau_{k,l} \defeq \inf\left\{t \geq s : \; V_{k,s}(t) = V_{l,s}(t) \right\}
    \]
    then for example for the special case $k \geq 1$ we have
    \[
    \valseg{\tau_{k,k+1} = \infty} =
    \int_{0}^{\infty} \int_{0}^{y}
    \begin{vNiceMatrix}[c,margin]
    g_{k-1}(x) & g_{k-1}(y)\\
    g_{k}(x) & g_{k}(y)
    \end{vNiceMatrix}
    \,dx\,dy
    = \frac{1}{2^{2k}}{2k \choose k}.
    \]
    It is interesting to notice that this is exactly the probability that a simple symmetric random walk returns to its starting point in $2k$ steps.
    
    \newpage

    \section{Summary and outlook} \label{summary_and_outlook}
    \thispagestyle{plain}
    As mentioned in the introduction, the Pólya Web has turned out to be a very interesting and rich mathematical structure.
    The fact that we can characterize the the event whether two walks coalesced in finite time or not by only looking at the limiting random variables (with Beta distribution (\Cref{lmm:witness_functions_and_stopping_times}) is quite an interesting observation.
    The joint density (\Cref{thm:joint_density}) helps us to calculate these probabilities for any scenario of finite number of Pólya Walks.
    We have showed a strong law for the number of components (\Cref{prp:stong_law_for_the_polya_components}).
    In the future we will try to show a large deviation bound and a limit theorem.

    A local scaling came naturally on the edges from the Poisson approximation (\Cref{prp:joint_poisson}), and from it the Yule Web emerged (\cref{yule_web}).
    It had analogous properties as the Pólya Web, using the appropriate limiting variables (with Gamma distribution).
    We intend to continue with the scaling in the bulk.
    Consider the scales processes for a fixed $\alpha \in (0,1)$ defined for $0 < s \leq t$ and $h \in \R$
    \[
    W_n(s,t,h) \defeq \frac{1}{\sqrt{n}}\left(X_{\left(\floor*{n\alpha s + \sqrt{n}h},\ceil*{n(1-\alpha)s-\sqrt{n}h}\right)}(\floor*{nt})-n\alpha t\right).
    \]
    Then we aim to prove process convergence of $W_n(s,t,h)$ as $n \to \infty$ to a particular variant of the Brownian Web distorted in both "space" and time, $W(s,t,h)$.
    
    Another interesting aspect is to analyse the random set which we get if we consider the realizations in $[0,1]$ of the limiting (Beta distributed) random variables from all the points in $\N \times \N$ (similar to the ones in \cite{tsirelson:2006}).
    Moreover, it is also an interesting aspect to consider random measures on $[0,1]$ induced by the Pólya Web.

    One word after another, we will continue the work we have started here with great enthusiasm in the future.

    \newpage
    \phantomsection
    \bibliographystyle{plain}
    \bibliography{references}

@article{toth_wendelin:1998,
title = "The true self-repelling motion",
author = "B. Tóth and W. Werner",
year = "1998",
language = "English",
volume = "111",
pages = "375--452",
journal = "Probability Theory and Related Fields",
issn = "0178-8051",
publisher = "Springer, New York, NY",
number = "3",
}

@article{polya_eggenberger:1923,
author = {F. Eggenberger and G. Pólya},
title = {Über die {S}tatistik verketteter {V}orgänge},
journal = {ZAMM - Journal of Applied Mathematics and Mechanics / Zeitschrift für Angewandte Mathematik und Mechanik},
volume = {3},
number = {4},
pages = {279-289},
doi = {https://doi.org/10.1002/zamm.19230030407},
url = {https://onlinelibrary.wiley.com/doi/abs/10.1002/zamm.19230030407},
eprint = {https://onlinelibrary.wiley.com/doi/pdf/10.1002/zamm.19230030407},
year = {1923}
}

@article{cannizzaro_hairer:2023,
author = {G. Cannizzar and M. Hairer},
year = {2023},
pages = {},
title = {The {B}rownian {W}eb as a random $\mathbb{R}$-tree},
volume = {28},
journal = {Electronic Journal of Probability},
doi = {10.1214/23-EJP984}
}

@inproceedings{wajc:2017,
  title={Negative {A}ssociation-{D}efinition, {P}roperties, and {A}pplications},
  author={D. Wajc},
  year={2017},
  url={https://api.semanticscholar.org/CorpusID:30644334}
}

@article{karlin_mcgregor:1959,
  author = {S. Karlin and {J. L.} McGregor},
  title = {Coincidence probabilities},
  journal = {Pacific Journal of Mathematics},
  volume = {9},
  number = {4},
  year = {1959},
  pages = {1141--1164}
}

@book{durett:2010,
  author    = {R. Durrett},
  title     = {Probability: Theory and Examples},
  edition   = {4th},
  year      = {2010},
  publisher = {Cambridge University Press},
  isbn      = {9780521765398},
  series    = {Cambridge Series in Statistical and Probabilistic Mathematics},
  doi       = {10.1017/CBO9780511779398}
}

@book{williams:1991,
  author    = {D. Williams},
  title     = {Probability with Martingales},
  year      = {1991},
  publisher = {Cambridge University Press},
  series    = {Cambridge Mathematical Textbooks},
  isbn      = {9780521406055},
}

@article{tsirelson:2006,
  author    = {B. Tsirelson},
  title     = {Brownian local minima, random dense countable sets and random equivalence classes},
  journal   = {Electronic Journal of Probability},
  volume    = {11},
  year      = {2006},
  pages     = {162--198},
  doi       = {10.1214/EJP.v11-309},
  url       = {https://projecteuclid.org/euclid.ejp/1465067674}
}

@article{dubhashi-ranjan:1998,
author = {D. Dubhashi and D. Ranjan},
title = {Balls and {B}ins: {A} {S}tudy in {N}egative {D}ependence},
year = {1998},
issue_date = {1998},
publisher = {John Wiley \& Sons, Inc.},
address = {USA},
volume = {13},
number = {2},
issn = {1042-9832},
journal = {Random Struct. Algorithms},
pages = {99–124},
numpages = {26}
}

@article{shao:200,
journal={Journal of Theoretical Probability},
author={Qi-Man Shao},
title={A Comparison Theorem on Moment Inequalities Between Negatively Associated and Independent Random Variables},
year={2000},
pages={343-356},
volume={13},
number={2},
doi={10.1023/A:1007849609234},
url={https://ideas.repec.org/a/spr/jotpro/v13y2000i2d10.1023_a1007849609234.html},
}

@article{van_den_berg_kesten:1985,
title={Inequalities with applications to percolation and reliability},
volume={22}, DOI={10.2307/3213860},
number={3},
journal={Journal of Applied Probability},
author={{J. van den} Berg and H. Kesten},
year={1985}, pages={556–569}
}

@inbook{borgs1999,
author="C. Borgs
and J. T. Chayes
and D. Randall",
editor="Bramson, Maury
and Durrett, Rick",
title="The van den Berg-Kesten-Reimer Inequality: A Review",
bookTitle="Perplexing Problems in Probability: Festschrift in Honor of Harry Kesten",
year="1999",
publisher="Birkh{\"a}user Boston",
address="Boston, MA",
pages="159--173",
isbn="978-1-4612-2168-5",
doi="10.1007/978-1-4612-2168-5_9",
url="https://doi.org/10.1007/978-1-4612-2168-5_9"
}

@article{harris:1960, title={A lower bound for the critical probability in a certain percolation process}, volume={56}, DOI={10.1017/S0305004100034241}, number={1}, journal={Mathematical Proceedings of the Cambridge Philosophical Society}, author={T. E. Harris}, year={1960}, pages={13–20}}

@article {reimer:2000,
    AUTHOR = {D. Reimer},
     TITLE = {Proof of the van den {B}erg-{K}esten conjecture},
   JOURNAL = {Combin. Probab. Comput.},
  FJOURNAL = {Combinatorics, Probability and Computing},
    VOLUME = {9},
      YEAR = {2000},
    NUMBER = {1},
     PAGES = {27--32},
      ISSN = {0963-5483,1469-2163},
   MRCLASS = {60C05 (60E15 60K35)},
  MRNUMBER = {1751301},
MRREVIEWER = {H.\ Kesten},
       DOI = {10.1017/S0963548399004113},
       URL = {https://doi-org.bris.idm.oclc.org/10.1017/S0963548399004113},
}
    
    \end{document}